\documentclass{amsart}
\usepackage{amsmath}
\input xy

\xyoption{all}
\CompileMatrices
\usepackage{xcolor,sseq,amssymb}
\usepackage{wasysym}
\usepackage{pdflscape}
\usepackage{todonotes}
\usepackage{xypic}
\usepackage{datetime}
\usepackage{accents}
\usepackage{amsthm}
\usepackage{mathrsfs}
\usepackage{tikz}
\usepackage{mathtools}
\usepackage{hyperref}
\usepackage{graphicx}
\usepackage{pdfpages}
\usepackage[noadjust]{cite}
\usepackage{pigpen}
\usepackage{mathtools}
\usepackage{float}
\usepackage{caption}
\usepackage{setspace}
\usepackage{chngcntr}

\newtheoremstyle{mystyle}%                % Name
  {}%                                     % Space above
  {}%                                     % Space below
  {\itshape}%                                     % Body font
  {}%                                     % Indent amount
  {\bfseries}%                            % Theorem head font
  {.}%                                    % Punctuation after theorem head
  { }%                                    % Space after theorem head, ' ', or \newline
  {}%                                     % Theorem head spec (can be left empty, meaning `normal')

\theoremstyle{mystyle}

\newtheorem{thm}{Theorem}[section]
\newtheorem{lem}[thm]{Lemma}
\newtheorem{defi}[thm]{Definition}
\newtheorem{prop}[thm]{Proposition}
\newtheorem{exam}[thm]{Example}
\newtheorem{rmk}[thm]{Remark}
\newtheorem{cor}[thm]{Corollary}

\newtheorem{innercustomthm}{Theorem}
\newenvironment{customthm}[1]
  {\renewcommand\theinnercustomthm{#1}\innercustomthm}
  {\endinnercustomthm}
\newcommand{\Punderline}[1]{\protect\underline{#1}}
\newcommand{\UZ}{\Punderline{\mathbb{Z}}}
\newcommand{\U}[1]{\Punderline{#1}}
\newcommand{\UA}{\U{A}}
\newcommand{\ZZ}{\mathbb{Z}}
\newcommand{\MS}[1]{\mathscr{#1}}
\newcommand{\UR}{\U{R}}
\newcommand{\UM}{\U{M}}
\newcommand{\UN}{\U{N}}

\newcommand{\UX}{\U{X}}

\newcommand{\RR}{\mathbb{R}}
\newcommand{\QQ}{\mathbb{Q}}
\newcommand{\FF}{\mathbb{F}}
\newcommand{\UB}{\U{\mathbb{B}}}
\newcommand{\JJJ}{\blacktriangle}

\newcommand{\pushout}{\ar@{}[dr]|{\text{\pigpenfont R}}}
\newcommand{\pullback}{\ar@{}[dr]|{\text{\pigpenfont J}}}
\renewcommand{\star}{\bigstar}
\newcommand{\ssdiff}[1]{\ssarrow{-1}{#1}\ssmove{1}{-#1}}
\newcommand{\ssreddiff}[1]{\ssarrow[color=red]{-1}{#1}\ssmove{1}{-#1}}

\newcommand{\ssblueline}[1]{\ssline[color=blue]{0}{#1}}
\newcommand{\ssgreenline}[1]{\ssline[color=green]{0}{#1}}

\newcommand{\sscurvedblueline}[1]{\ssline[color=blue, curve=-.1]{0}{#1}}

\def\JJ{\blacktriangledown}

\def\halfbox{
     \begin{tikzpicture}[x=1.2ex,y=1.2ex]
        \draw (0,0) rectangle +(1,1);
        \fill (0,0) -- (0,.5) -- (1,.5) -- (1,0) -- cycle;
     \end{tikzpicture}
            }

\def\diagbox{
     \begin{tikzpicture}[x=1.2ex,y=1.2ex]
        \draw (0,0) rectangle +(1,1);
        \fill (0,0) -- (0,1) -- (1,0) -- cycle;
     \end{tikzpicture}
            }
\def\twobox{\diagbox}
\def\fourbox{
     \begin{tikzpicture}[x=1.2ex,y=1.2ex]
        \draw (0,0) rectangle +(1,1);
        \fill (0,0) -- (1,1) -- (1,0) -- cycle;
     \end{tikzpicture}
            }

\newcount\xval
\newcount\yval
\newcount\difflim

\newcount\ssdoublexval
\newcount\ssdoubleyval

\def\ssray#1#2#3#4#5#6#7{
\ifnum#3>#1
\else
\ifnum#4>#2
\else
   \ssmoveto{#3}{#4}\ssdrop{#5}
   \xval=#3
   \yval=#4
   \advance \xval by #6
   \advance \yval by #7
   \ssray{#1}{#2}{\xval}{\yval}{#5}{#6}{#7}
\fi
\fi
}

\def\sslineray#1#2#3#4#5#6#7#8{
\ifnum#3>#1
\else
\ifnum#4>#2
\else
   \ssmoveto{#3}{#4}\ssline{#5}{#6}\ssmove{-#5}{-#6}
   \xval=#3
   \yval=#4
   \advance \xval by #7
   \advance \yval by #8
   \sslineray{#1}{#2}{\xval}{\yval}{#5}{#6}{#7}{#8}
\fi
\fi
}

\def\sseightfanray#1#2#3#4#5#6{
\ifnum#3>#1\else \ifnum#4>#2\else
   \xval=#3
   \yval=#4
   \ssmoveto{#3}{#4}
   \advance \xval by 1
   \advance \yval by 7
   \ifnum\xval>#1\else \ifnum\yval>#2\else
         \ssline{1}{7}\ssmove{-1}{-7}
   \fi\fi
   \advance \xval by 2
   \advance \yval by -2
   \ifnum\xval>#1\else \ifnum\yval>#2\else
         \ssline{3}{5}\ssmove{-3}{-5}
   \fi\fi
   \advance \xval by 2
   \advance \yval by -2
   \ifnum\xval>#1\else \ifnum\yval>#2\else
         \ssline{5}{3}\ssmove{-5}{-3}
   \fi\fi
   \advance \xval by 2
   \advance \yval by -2
   \ifnum\xval>#1\else \ifnum\yval>#2\else
         \ssline{7}{1}\ssmove{-7}{-1}
   \fi\fi
   \advance \xval by -7
   \advance \yval by -1
   \advance \xval by #5
   \advance \yval by #6
   \sseightfanray{#1}{#2}{\xval}{\yval}{#5}{#6}
\fi
\fi
}

\def\sscoloreightfanray#1#2#3#4#5#6{
\ifnum#3>#1\else \ifnum#4>#2\else
   \xval=#3
   \yval=#4
   \ssmoveto{#3}{#4}
   \advance \xval by 1
   \advance \yval by 7
   \ifnum\xval>#1\else \ifnum\yval>#2\else
         \ssline{1}{7}\ssmove{-1}{-7}
   \fi\fi
   \advance \xval by 2
   \advance \yval by -2
   \ifnum\xval>#1\else \ifnum\yval>#2\else
         \ssline[color=blue]{3}{5}\ssmove{-3}{-5}
   \fi\fi
   \advance \xval by 2
   \advance \yval by -2
   \ifnum\xval>#1\else \ifnum\yval>#2\else
         \ssline[color=green]{5}{3}\ssmove{-5}{-3}
   \fi\fi
   \advance \xval by 2
   \advance \yval by -2
   \ifnum\xval>#1\else \ifnum\yval>#2\else
         \ssline[color=cyan]{7}{1}\ssmove{-7}{-1}
   \fi\fi
   \advance \xval by -7
   \advance \yval by -1
   \advance \xval by #5
   \advance \yval by #6
   \sscoloreightfanray{#1}{#2}{\xval}{\yval}{#5}{#6}
\fi
\fi
}

\def\sseightfandoubleray#1#2#3#4#5#6#7#8{
\ifnum#3>#1\else \ifnum#4>#2\else
   \sseightfanray{#1}{#2}{#3}{#4}{#5}{#6}
   \ssdoublexval=#3
   \ssdoubleyval=#4
   \advance \ssdoublexval by #7
   \advance \ssdoubleyval by #8
   \sseightfandoubleray{#1}{#2}{\ssdoublexval}{\ssdoubleyval}{#5}{#6}{#7}{#8}
\fi\fi
}

\def\sscoloreightfandoubleray#1#2#3#4#5#6#7#8{
\ifnum#3>#1\else \ifnum#4>#2\else
   \sscoloreightfanray{#1}{#2}{#3}{#4}{#5}{#6}
   \ssdoublexval=#3
   \ssdoubleyval=#4
   \advance \ssdoublexval by #7
   \advance \ssdoubleyval by #8
   \sscoloreightfandoubleray{#1}{#2}{\ssdoublexval}{\ssdoubleyval}
                            {#5}{#6}{#7}{#8}
\fi\fi
}

\def\ssantiray#1#2#3#4#5#6#7{
\ifnum#3<#1
\else
\ifnum#4<#2
\else
   \ssmoveto{#3}{#4}\ssdrop{#5}
   \xval=#3
   \yval=#4
   \advance \xval by #6
   \advance \yval by #7
   \ssantiray{#1}{#2}{\xval}{\yval}{#5}{#6}{#7}
\fi
\fi
}

\def\ssraydiff#1#2#3#4#5#6#7#8{
\difflim=#2
\advance \difflim by -#8
\ifnum#3>#1
\else
\ifnum#4>#2
\else
   \ssmoveto{#3}{#4}\ssdrop{#5}
   \ifnum#4>\difflim
   \else
       \ssdiff{#8}
   \fi
   \xval=#3
   \yval=#4
   \advance \xval by #6
   \advance \yval by #7
   \ssraydiff{#1}{#2}{\xval}{\yval}{#5}{#6}{#7}{#8}
\fi
\fi
}

\def\ssrayreddiff#1#2#3#4#5#6#7#8{
\difflim=#2
\advance \difflim by -#8
\ifnum#3>#1
\else
\ifnum#4>#2
\else
   \ssmoveto{#3}{#4}\ssdrop{#5}
   \ifnum#4>\difflim
   \else
       \ssreddiff{#8}
   \fi
   \xval=#3
   \yval=#4
   \advance \xval by #6
   \advance \yval by #7
   \ssrayreddiff{#1}{#2}{\xval}{\yval}{#5}{#6}{#7}{#8}
\fi
\fi
}

\def\ssdoubleray#1#2#3#4#5#6#7#8#9{
\ifnum#3>#1
\else
\ifnum#4>#2
\else
   \ssray{#1}{#2}{#3}{#4}{#5}{#6}{#7}
   \ssdoublexval=#3
   \ssdoubleyval=#4
   \advance \ssdoublexval by #8
   \advance \ssdoubleyval by #9
   \ssdoubleray{#1}{#2}{\ssdoublexval}{\ssdoubleyval}
               {#5}{#6}{#7}{#8}{#9}
\fi
\fi
}

\def\ssantidoubleray#1#2#3#4#5#6#7#8#9{
\ifnum#3<#1
\else
\ifnum#4<#2
\else
   \ssantiray{#1}{#2}{#3}{#4}{#5}{#6}{#7}
   \ssdoublexval=#3
   \ssdoubleyval=#4
   \advance \ssdoublexval by #8
   \advance \ssdoubleyval by #9
   \ssantidoubleray{#1}{#2}{\ssdoublexval}{\ssdoubleyval}{#5}{#6}{#7}{#8}{#9}
\fi
\fi
}

\def\ssrayblueline#1#2#3#4#5#6#7{
\ifnum#3>#1
\else
\ifnum#4>#2
\else
   \ssmoveto{#3}{#4}\ssdrop{#5}\ssblueline{-2}
   \xval=#3
   \yval=#4
   \advance \xval by #6
   \advance \yval by #7
   \ssrayblueline{#1}{#2}{\xval}{\yval}{#5}{#6}{#7}
\fi
\fi
}

\def\ssraygreenline#1#2#3#4#5#6#7{
\ifnum#3>#1
\else
\ifnum#4>#2
\else
   \ssmoveto{#3}{#4}\ssdrop{#5}\ssgreenline{-2}
   \xval=#3
   \yval=#4
   \advance \xval by #6
   \advance \yval by #7
   \ssraygreenline{#1}{#2}{\xval}{\yval}{#5}{#6}{#7}
\fi
\fi
}

\def\ssdoublerayblueline#1#2#3#4#5#6#7#8#9{
\ifnum#3>#1
\else
\ifnum#4>#2
\else
   \ssrayblueline{#1}{#2}{#3}{#4}{#5}{#6}{#7}
   \ssdoublexval=#3
   \ssdoubleyval=#4
   \advance \ssdoublexval by #8
   \advance \ssdoubleyval by #9
   \ssdoublerayblueline{#1}{#2}{\ssdoublexval}{\ssdoubleyval}
                       {#5}{#6}{#7}{#8}{#9}
\fi
\fi
}

\def\ssdoubleraygreenline#1#2#3#4#5#6#7#8#9{
\ifnum#3>#1
\else
\ifnum#4>#2
\else
   \ssraygreenline{#1}{#2}{#3}{#4}{#5}{#6}{#7}
   \ssdoublexval=#3
   \ssdoubleyval=#4
   \advance \ssdoublexval by #8
   \advance \ssdoubleyval by #9
   \ssdoubleraygreenline{#1}{#2}{\ssdoublexval}{\ssdoubleyval}
                       {#5}{#6}{#7}{#8}{#9}
\fi
\fi
}

\def\ssrayblueblueline#1#2#3#4#5#6#7{
\ifnum#3>#1
\else
\ifnum#4>#2
\else
   \ssmoveto{#3}{#4}\ssdrop{#5}\ssblueline{-4}
   \xval=#3
   \yval=#4
   \advance \xval by #6
   \advance \yval by #7
   \ssrayblueblueline{#1}{#2}{\xval}{\yval}{#5}{#6}{#7}
\fi
\fi
}

\def\ssraygreengreenline#1#2#3#4#5#6#7{
\ifnum#3>#1
\else
\ifnum#4>#2
\else
   \ssmoveto{#3}{#4}\ssdrop{#5}\ssgreenline{-4}
   \xval=#3
   \yval=#4
   \advance \xval by #6
   \advance \yval by #7
   \ssraygreengreenline{#1}{#2}{\xval}{\yval}{#5}{#6}{#7}
\fi
\fi
}

\def\ssdoublerayblueblueline#1#2#3#4#5#6#7#8#9{
\ifnum#3>#1
\else
\ifnum#4>#2
\else
   \ssrayblueblueline{#1}{#2}{#3}{#4}{#5}{#6}{#7}
   \ssdoublexval=#3
   \ssdoubleyval=#4
   \advance \ssdoublexval by #8
   \advance \ssdoubleyval by #9
   \ssdoublerayblueblueline{#1}{#2}{\ssdoublexval}{\ssdoubleyval}
                           {#5}{#6}{#7}{#8}{#9}
\fi
\fi
}

\def\ssdoubleraygreengreenline#1#2#3#4#5#6#7#8#9{
\ifnum#3>#1
\else
\ifnum#4>#2
\else
   \ssraygreengreenline{#1}{#2}{#3}{#4}{#5}{#6}{#7}
   \ssdoublexval=#3
   \ssdoubleyval=#4
   \advance \ssdoublexval by #8
   \advance \ssdoubleyval by #9
   \ssdoubleraygreengreenline{#1}{#2}{\ssdoublexval}{\ssdoubleyval}
                           {#5}{#6}{#7}{#8}{#9}
\fi
\fi
}

\def\ssreddiffray#1#2#3#4#5#6#7{
\difflim=#2
\advance \difflim by -#7
\ifnum#3>#1
\else
\ifnum#4>\difflim
\else
       \ssmoveto{#3}{#4}\ssreddiff{#7}
       \xval=#3
       \yval=#4
       \advance \xval by #5
       \advance \yval by #6
       \ssreddiffray{#1}{#2}{\xval}{\yval}{#5}{#6}{#7}
\fi
\fi
}

\def\reddiffdoubleray#1#2#3#4#5#6#7#8#9{
\difflim=#2
\advance \difflim by -#7
\ifnum#3 > #1
\else
\ifnum#4 > \difflim
\else
       \ssreddiffray{#1}{#2}{#3}{#4}{#5}{#6}{#7}
       \ssdoublexval=#3
       \ssdoubleyval=#4
       \advance \ssdoublexval by #8
       \advance \ssdoubleyval by #9
       \reddiffdoubleray{#1}{#2}{\ssdoublexval}{\ssdoubleyval}{#5}{#6}{#7}{#8}{#9}
\fi
\fi
}

\def\ssreddiffantiray#1#2#3#4#5#6#7{
\difflim=#1
\advance \difflim by 1
\ifnum#3< \difflim
\else
\ifnum#4<#2
\else
       \ssmoveto{#3}{#4}\ssreddiff{#7}
       \xval=#3
       \yval=#4
       \advance \xval by #5
       \advance \yval by #6
       \ssreddiffantiray{#1}{#2}{\xval}{\yval}{#5}{#6}{#7}
\fi
\fi
}

\def\reddiffantidoubleray#1#2#3#4#5#6#7#8#9{
\difflim=#2
\ifnum#3<#1
\else
\ifnum#4<\difflim
\else
       \ssreddiffantiray{#1}{#2}{#3}{#4}{#5}{#6}{#7}
       \ssdoublexval=#3
       \ssdoubleyval=#4
       \advance \ssdoublexval by #8
       \advance \ssdoubleyval by #9
       \reddiffantidoubleray{#1}{#2}{\ssdoublexval}{\ssdoubleyval}{#5}{#6}{#7}{#8}{#9}
\fi
\fi
}

\def\ssbluelineray#1#2#3#4#5#6#7{
\difflim=#2
\advance \difflim by -#7
\ifnum#3>#1
\else
\ifnum#4>\difflim
\else
       \ssmoveto{#3}{#4}\ssblueline{#7}
       \xval=#3
       \yval=#4
       \advance \xval by #5
       \advance \yval by #6
       \ssbluelineray{#1}{#2}{\xval}{\yval}{#5}{#6}{#7}
\fi
\fi
}

\def\sscurvedbluelineray#1#2#3#4#5#6#7{
\difflim=#2
\advance \difflim by -#7
\ifnum#3>#1
\else
\ifnum#4>\difflim
\else
       \ssmoveto{#3}{#4}\sscurvedblueline{#7}
       \xval=#3
       \yval=#4
       \advance \xval by #5
       \advance \yval by #6
       \sscurvedbluelineray{#1}{#2}{\xval}{\yval}{#5}{#6}{#7}
\fi
\fi
}

\def\ssbluelineantiray#1#2#3#4#5#6#7{
\ifnum#3<#1
\else
\ifnum#4<#2
\else
       \ssmoveto{#3}{#4}\ssblueline{#7}
       \xval=#3
       \yval=#4
       \advance \xval by #5
       \advance \yval by #6
       \ssbluelineantiray{#1}{#2}{\xval}{\yval}{#5}{#6}{#7}
\fi
\fi
}

\def\ssgreenlineantiray#1#2#3#4#5#6#7{
\ifnum#3<#1
\else
\ifnum#4<#2
\else
       \ssmoveto{#3}{#4}\ssgreenline{#7}
       \xval=#3
       \yval=#4
       \advance \xval by #5
       \advance \yval by #6
       \ssgreenlineantiray{#1}{#2}{\xval}{\yval}{#5}{#6}{#7}
\fi
\fi
}

\def\bluelinedoubleray#1#2#3#4#5#6#7#8#9{
\difflim=#2
\advance \difflim by -#7
\ifnum#3>#1
\else
\ifnum#4>\difflim
\else
       \ssbluelineray{#1}{#2}{#3}{#4}{#5}{#6}{#7}
       \ssdoublexval=#3
       \ssdoubleyval=#4
       \advance \ssdoublexval by #8
       \advance \ssdoubleyval by #9
       \bluelinedoubleray{#1}{#2}{\ssdoublexval}{\ssdoubleyval}{#5}{#6}{#7}{#8}{#9}
\fi
\fi
}

\def\bluelineantidoubleray#1#2#3#4#5#6#7#8#9{
\difflim=#2
\advance \difflim by -#7
\ifnum#3<#1
\else
\ifnum#4<\difflim
\else
       \ssbluelineantiray{#1}{#2}{#3}{#4}{#5}{#6}{#7}
       \ssdoublexval=#3
       \ssdoubleyval=#4
       \advance \ssdoublexval by #8
       \advance \ssdoubleyval by #9
       \bluelineantidoubleray{#1}{#2}{\ssdoublexval}{\ssdoubleyval}{#5}{#6}{#7}{#8}{#9}
\fi
\fi
}

\def\greenlineantidoubleray#1#2#3#4#5#6#7#8#9{
\difflim=#2
\advance \difflim by -#7
\ifnum#3<#1
\else
\ifnum#4<\difflim
\else
       \ssgreenlineantiray{#1}{#2}{#3}{#4}{#5}{#6}{#7}
       \ssdoublexval=#3
       \ssdoubleyval=#4
       \advance \ssdoublexval by #8
       \advance \ssdoubleyval by #9
       \greenlineantidoubleray{#1}{#2}{\ssdoublexval}{\ssdoubleyval}{#5}{#6}{#7}{#8}{#9}
\fi
\fi
}

\def\bluessray#1#2#3#4#5#6#7{
\ifnum#3>#1
\else
\ifnum#4>#2
\else
   \ssmoveto{#3}{#4}\ssdrop[color=blue]{#5}
   \xval=#3
   \yval=#4
   \advance \xval by #6
   \advance \yval by #7
   \bluessray{#1}{#2}{\xval}{\yval}{#5}{#6}{#7}
\fi
\fi
}

\def\violetssray#1#2#3#4#5#6#7{
\ifnum#3>#1
\else
\ifnum#4>#2
\else
   \ssmoveto{#3}{#4}\ssdrop[color=violet]{#5}
   \xval=#3
   \yval=#4
   \advance \xval by #6
   \advance \yval by #7
   \violetssray{#1}{#2}{\xval}{\yval}{#5}{#6}{#7}
\fi
\fi
}

\def\violetssantiray#1#2#3#4#5#6#7{
\ifnum#3<#1
\else
\ifnum#4<#2
\else
   \ssmoveto{#3}{#4}\ssdrop[color=violet]{#5}
   \xval=#3
   \yval=#4
   \advance \xval by #6
   \advance \yval by #7
   \violetssantiray{#1}{#2}{\xval}{\yval}{#5}{#6}{#7}
\fi
\fi
}

\def\bluessantiray#1#2#3#4#5#6#7{
\ifnum#3<#1
\else
\ifnum#4<#2
\else
   \ssmoveto{#3}{#4}\ssdrop[color=blue]{#5}
   \xval=#3
   \yval=#4
   \advance \xval by #6
   \advance \yval by #7
   \bluessantiray{#1}{#2}{\xval}{\yval}{#5}{#6}{#7}
\fi
\fi
}

\def\ssbluelinedoubleray#1#2#3#4#5#6#7#8#9{
\ifnum#3>#1
\else
\ifnum#4>#2
\else
   \ssbluelineray{#1}{#2}{#3}{#4}{#5}{#6}{#7}
   \ssdoublexval=#3
   \ssdoubleyval=#4
   \advance \ssdoublexval by #8
   \advance \ssdoubleyval by #9
   \ssbluelinedoubleray{#1}{#2}{\ssdoublexval}{\ssdoubleyval}
                       {#5}{#6}{#7}{#8}{#9}
\fi
\fi
}

\def\ssgreenlineray#1#2#3#4#5#6#7{
\difflim=#2
\advance \difflim by -#7
\ifnum#3>#1
\else
\ifnum#4>\difflim
\else
       \ssmoveto{#3}{#4}\ssgreenline{#7}
       \xval=#3
       \yval=#4
       \advance \xval by #5
       \advance \yval by #6
       \ssgreenlineray{#1}{#2}{\xval}{\yval}{#5}{#6}{#7}
\fi
\fi
}

\def\ssgreenlinedoubleray#1#2#3#4#5#6#7#8#9{
\ifnum#3>#1
\else
\ifnum#4>#2
\else
   \ssgreenlineray{#1}{#2}{#3}{#4}{#5}{#6}{#7}
   \ssdoublexval=#3
   \ssdoubleyval=#4
   \advance \ssdoublexval by #8
   \advance \ssdoubleyval by #9
   \ssgreenlinedoubleray{#1}{#2}{\ssdoublexval}{\ssdoubleyval}
                       {#5}{#6}{#7}{#8}{#9}
\fi
\fi
}

\title[$H\UZ$]{Equivariant Eilenberg-Mac Lane Spectra in Cyclic $p$-Groups}
\author{Mingcong Zeng}
\address{University of Rochester, Rochester, NY, 14627}
\email{mingcongzeng@gmail.com}

\begin{document}
\maketitle

\begin{abstract}
In this paper we compute $RO(G)$-graded homotopy Mackey functors of $H\UZ$, the Eilenberg-Mac Lane spectrum of the constant Mackey functor of integers for cyclic $p$-groups and give a complete computation for $C_{p^2}$. We also discuss homological algebra of $\UZ$-modules for cyclic $p$-groups, and interactions between these two. The goal of computation in this paper is to understand various slice spectral sequences as $RO(G)$-graded spectral sequences of Mackey functors.
\end{abstract}

\setcounter{tocdepth}{1}

\tableofcontents

\section{Introduction}

In the groundbreaking paper \cite{HHR}, Hill, Hopkins, and Ravenel prove that the Kervaire invariant one elements $\theta_j$, which are defined in \cite{Kerv-Mil} and known to exist for $j \leq 5$, do not exist for $j \geq 7$. Their approach to the problem is through equivariant stable homotopy theory. A $256$-periodic spectrum $\Omega$, which can detect $\theta_j$ and has trivial homotopy groups in the corresponding dimensions, is constructed as the $C_8$-fixed point spectrum of $\Omega_{\mathbb{O}}$, a localization of $N_2^8 MU_{\RR}$, the $C_8$-norm of the $C_2$-equivariant real cobordism spectrum.

Equivariant stable homotopy theory has two fundamental differences from its non-equivariant brother. First, the role of abelian groups in the non-equivariant stable homotopy theory is replaced by Mackey functors. Fixing a finite group $G$, a Mackey functor is an algebraic structure recording algebraic invariants for different subgroups of $G$ and relations between them. Analogous to the non-equivariant case, given a Mackey functor $\UM$, one can construct an Eilenberg-Mac Lane spectrum $H\UM$ with the desired property. Second, one can suspend an equivariant spectrum not only by spheres, but also by representation spheres $S^V$, which is the one point compactification of $V \in RO(G)$, the group of virtual $G$-representations. Therefore, the fundamental homotopy invariant of an equivariant spectrum $X$ would be $\U{\pi}_{\star}(X)$, its $RO(G)$-graded homotopy Mackey functor.

The analysis of $N_2^8 MU_{\RR}$ and its localization is made possible by the slice filtration. For any finite group $G$, the slice filtration is an equivariant filtration generalizing the Postnikov filtration in non-equivariant stable homotopy theory. Given an equivariant spectrum $X$, the slice filtration gives the slice tower $P^*X$, which then gives a spectral sequence of Mackey functors convergent to the $RO(G)$-graded homotopy Mackey functor of $X$. For $X = N_2^{2^n} MU_{\RR}$, one of the most technical result of \cite{HHR} is the slice and reduction theorem, which determines the slices of $N_2^{2^n} MU_{\RR}$.

\begin{thm}[{\cite[Theorem~6.1]{HHR}}]\label{thm-reduction}
Let $\UZ$ be the constant Mackey functor of integers. Then all slices of $N_2^{2^n}MU_{\RR}$ are suspensions of $H\UZ$ by induced representation spheres. Particularly, the zeroth slice is $H\UZ$ and $\U{\pi}_0(N_2^{2^n}MU_{\RR}) = \UZ$.
\end{thm}

Recently, Jeremy Hahn and Danny Shi in \cite{Hahn-Shi} prove that the complex orientation of Morava $E_n$ at $p = 2$ lifts to a real orientation $MU_{\RR} \rightarrow E_n$, with the subgroup $C_2$ of Morava stabilizer group acting on the latter. This allows us to apply slice techniques to investigate Morava $E$-theories and their homotopy fixed points.

One subtlety of these theorems is that, unlike abelian groups, $\UZ$ is not the tensor unit of Mackey functors. The tensor unit of Mackey functors is $\UA$, the Burnside Mackey functor. However, $\UZ$ is still a ring of Mackey functor, which is commonly called Green functor, or even more, a Tambara functor. This means that in slice spectral sequences of modules over $N_2^{2^n} MU_{\RR}$, all differentials and extensions are in the category of $\UZ$-modules. This observation has already been exploited by Hill in \cite{Hill:eta}, where he proves that $\eta^3$ cannot be detected in the fixed point spectrum of $N_2^{2^n}MU_{\RR}$ by showing that in $3$-stem of the slice spectral sequence, there is no room for extensions of $\UZ$-modules that can fit in an element of order $8$. This motivates the first part of this paper: a computational discussion of homological algebra of $\UZ$-modules.

Homological algebra of Mackey functors and $\UZ$-modules is not a new subject. In \cite{Greenlees-Mackey}, Greenlees shows that projective dimension of a Mackey functor for any nontrivial finite group is either $0,1$ or $\infty$, and we encounter the last case quite often. On the other hand, James Arnold, in a series of papers \cite{Arnold:Cyclic},\cite{Arnold:infinite},\cite{Arnold:Dim2},\cite{Arnold:p} and \cite{Arnold:Q}, without mentioning Mackey functors, computes projective dimension of $\UZ$-modules for various finite groups. His work is translated into the language of Mackey functor by Bouc, Stancu, and Webb in \cite{BSW}:

\begin{thm}[{\cite[Corollary~7.2]{BSW}}, Arnold]
The category of $\UZ$-modules in $G$-Mackey functors has finite projective dimension if and only if for each prime $p > 2$, the Sylow $p$-group of $G$ is cyclic, and the Sylow $2$-group of $G$ is either cyclic or dihedral.
\end{thm}

In this paper, we gives various computations and examples around $\U{Ext}_{\UZ}$ and $\U{Tor}^{\UZ}$, the derived functors of the internal Hom and internal tensor product of $\UZ$-modules for $G = C_{p^n}$. The highlight is that these derived functors have peculiar but computable phenomena.

\begin{customthm}{\ref{thm-ext}}
Let $G = C_{p^n}$ and $\UM$ be a $\UZ$-module that $\UM(G/e) \cong 0$, then
\[
    \U{Ext}_{\UZ}^i(\UM,\UZ) \cong \left\{ \begin{array}{ll}
                \UM^E   & \textrm{for $i = 3$}\\
                \U{0}   & \textrm{otherwise}
                                           \end{array} \right.
\]
where $\UM^E$ is the levelwise dual of $\UM$.
\end{customthm}

We provide two proofs of this theorem. The first one is purely algebraic, making use of the finiteness of projective dimensions for cyclic $p$-groups. The second one is through the lens of equivariant stable homotopy theory. Similar to the Quillen equivalence between the category of chain complexes of abelian groups and $H\ZZ$-modules, we can translate $\U{Ext}_{\UZ}$-computation into the category of $H\UZ$-modules. However, in the equivariant world, the category of $H\UZ$-modules has richer structures: suspensions by any representation spheres, not only $S^n$, are invertible and we have equivariant dualities. Making use of these advantages, we obtain an easier proof.

Another consequence of Theorem \ref{thm-reduction} is that the slices of $N_2^{2^n} MU_{\RR}$ are all smash product of $H\UZ$ and induced representation spheres. Therefore by understanding the $RO(G)$-graded homotopy Mackey functors of $H\UZ$, we can understand the $E_2$-page of the slice spectral sequence. Indeed, by only computing a small portion of $\U{\pi}_{\star}(H\UZ)$, Hill, Hopkins, and Ravenel prove the gap theorem.

\begin{thm}[{\cite[Theorem~8.3]{HHR}}]
For $-4 < i < 0$, $\U{\pi}_{i}(\Omega_{\mathbb{O}}) \cong \U{0}$.
\end{thm}

Another interesting phenomenon about the slice filtration is that, even though we mainly care about elements in the integer degrees, multiplicative generators almost never lie in integer degrees. It is common that the integer degree part of a slice spectral sequence itself is difficult to describe, but belongs to a very nice $RO(G)$-graded spectral sequence. Also, structure maps of Mackey functors commute with differentials, and the norm can produce new differentials from known ones. Therefore when computing the slice spectral sequences, we really want to compute any slice spectral sequence as an $RO(G)$-graded spectral sequence of Mackey functors. This motivates the second part of this paper, the computation of $\U{\pi}_{\star}(H\UZ)$ for $G = C_{p^n}$.

The $RO(G)$-graded homotopy Mackey functors of an Eilenberg-Mac Lane spectrum can be thought as the $RO(G)$-graded ordinary cohomology of a point for the coefficient Mackey functor. As an $RO(G)$-graded Green functor, its structure is very complicated. Lewis in \cite[Theorem~2.1,2.3]{Lewis:ROG} computed the $RO(G)$-graded cohomology of a point with $G = C_p$ and coefficient $\UA$, the Burnside Mackey functor, and use them to prove a freeness theorem of cohomology of a $G$-cell complex \cite[Theorem~2.6]{Lewis:ROG}. For $\UZ$-coefficient, when $G = C_2$, $\U{\pi}_{\star}(H\UZ)$ is well-known and is computed in \cite[Theorem~2.8]{Dugger} and \cite[Section~2]{Greenlees-Four}. In this paper, we give a inductive procedure of computing $\U{\pi}_{\star}(H\UZ)$ through the Tate diagram, and use this procedure to completely compute the case $G = C_{p^2}$.

\begin{customthm}{\ref{thm-Cp2oriented}}
For $G = C_{p^2}$, the $RO(G)$-graded homotopy Mackey functor of $H\UZ$ is computed in terms of generators and relations.
\end{customthm}

There are two different paths of making inductive arguments in equivariant stable homotopy theory. One is through induction on subgroups. An example is the isotropy separation sequence \cite[Section~2.5.2]{HHR}. But this is not the path we take. For $G = C_{p^n}$, we make induction through its quotient groups. We construct a pullback functor $\Psi^*_K$ from $G/K$-Mackey functors to $G$-Mackey functors, for any normal subgroup $K \subset G$. $\Psi^*_K$ is strongly monoidal, exact and preserves projective objects, therefore preserves any homological invariants. Using $\Psi^*_K$ and the Tate diagram, we boil down the computation of $\U{\pi}_{\star}(H\UZ)$ into two different periodic $RO(G)$-graded Mackey functors, and the gold relation \cite[Lemma~3.6]{HHR:C4} describes how these two parts interact with each other.

Inside $\U{\pi}_{\star}(H\UZ)$, there are some important torsion free Mackey functors in degree $\U{\pi}_{|V| - V}(H\UZ)$. They are equivariant generalization of orientations $H_n(S^n;\ZZ)$ in representation spheres. Many different Mackey functors can appear as orientations of $S^V$, but all of them satisfy a structural property: the orientation Mackey functor evaluating at each orbit is $\ZZ$. We prove that the converse is also true: Any such Mackey functor can appear as an orientation. Furthermore, we can find representation spheres modelling the Moore spectra of these $\UZ$-modules.

\begin{customthm}{\ref{thm-formz}}
For $G = C_{p^n}$, if $\UM$ is a $\UZ$-module that $\UM(G/K) \cong \ZZ$ for all subgroups $K \subset G$, then there is a virtual representation $V$ of diemension $0$ that
\[
    H\UM \simeq S^{V} \wedge H\UZ.
\]
\end{customthm}

Another topic we explore about $H\UZ$ is equivariant duality. In \cite{Hill-Meier}, Hill and Meier show that as a $C_2$-spectrum, the topological modular form with level $3$ structure $Tmf_1(3)$ is equivariantly self-dual with an $RO(C_2)$-shift. The shift is a combination of the Serre duality of the cohomology of the underlying stack, and the self-duality of $H\UZ$. In general, for $G = C_{p^n}$, there are two different kinds of self-duality presented in $\U{\pi}_{\star}(H\UZ)$. One of them is the Anderson duality, constructed using the injective resolution $\ZZ \rightarrow \QQ \rightarrow \QQ/\ZZ$ of abelian groups and an equivariant Brown representability theorem. Another one is the Spainer-Whitehead duality of $H\UZ$-modules, and its computation involves a universal coefficient spectral sequence, which requires homological algebra of $\UZ$-modules as input. We show that how these two dualities interplay with each other and show that if we understand orientations of $H\UZ$ in prior, we can use these two dualities iteratively to compute $\U{\pi}_{\star}(H\UZ)$.

The structure of this paper is the following. In Section \ref{sec-Mack} we give definitions of Mackey, Green and Tambara functors and build up basic properties we need for computation. Section \ref{sec-HA} discusses homological algebra of $\UZ$-modules and gives various computations. These two sections are purely algebraic. In section \ref{sec-EM} we talk about constructions needed in equivariant stable homotopy theory, including fixed point and orbit constructions, the Tate diagram and the equivariant universal coefficient spectral sequence. In Section \ref{sec-HZ1} we give proofs of Theorem \ref{thm-ext} and Theorem \ref{thm-formz}. Lastly, Section \ref{sec-HZ2} contains a complete computation of $\U{\pi}_{\star}(H\UZ)$ for $C_{p^2}$ with multiplicative structure, an investigation of the dualities of $\U{\pi}_{\star}(H\UZ)$, and more homological computation through equivariant topology.

\subsection*{Acknowledgements}
I want to thank organizers, mentors and participants of Talbot workshop 2016 and European Talbot workshop 2017, where I learned the background of this paper and benefits from enormous amounts of conversations. I also want to thank John Greenlees, Tyler Lawson, Nicolas Ricka, Danny XiaoLin Shi, Guozhen Wang, Qiaofeng Zhu and Zhouli Xu for illuminating conversations. Lastly and most deeply, I want to thank Mike Hill for his generous help in all stages of this paper, and Doug Ravenel, my advisor, for everything he gave me: support, guidance and inspiration.

\section{Mackey Functors and $\UZ$-Modules}\label{sec-Mack}

\subsection{Mackey functors}
We use the definition of Mackey functor in \cite{Lewis:Green}. Notice that there are equivalent definitions, like the one of Dress \cite{Dress:Mackey}.

\begin{defi}\label{Def-semi-Burnside}
Let $G$ be a finite group. The \textbf{Lindner category} $\MS{B}_G^+$ is the following:
\begin{itemize}
    \item Objects: Finite left $G$-sets.
    \item Morphisms: A morphism $f: X \rightarrow Y$ is represented by a diagram of finite $G$-set maps.
        \[
            X \xleftarrow{f_1} Z_f \xrightarrow{f_2} Y
        \]
        Two diagrams $f,g$ represent the same morphism if there is a $G$-set isomorphism $\theta$ such that the following diagram commutes:
        \[
        \xymatrix{
        Z_f \ar[d]_{f_1} \ar[r]^{f_2} \ar@{=}[dr]_{\theta} & Y \\
        X                                              & Z_g \ar[l]^{g_1} \ar[u]_{g_2}
        }
        \]
    \item Compositions: Given $f:X \rightarrow Y$ and $g: Y \rightarrow Z$, their composition $g \circ f$is the pullback diagram
        \[
        \xymatrix{
        Z_{g \circ f} \ar[d] \ar[r] \pullback & Z_g \ar[d]_{g_1} \ar[r]_{g_2} & Z \\
        Z_f \ar[d]_{f_1} \ar[r]_{f_2} & Y & \\
        X & &
        }
        \]
\end{itemize}
\end{defi}

Notice that morphism set $\MS{B}_G^+(X,Y)$ is a commutative monoid under the disjoint union, and the composition is bilinear. Thus we can define the Burnside category by turning morphism monoids in $\MS{B}_G^+$ into abelian groups.

\begin{defi}\label{Def-Burnside}
The \textbf{Burnside category} $\MS{B}_G$ is obtained from $\MS{B}_G^+$ by forming formal differences in each morphism monoid.
\end{defi}

There is an evident functor $D:\MS{B}_G \rightarrow \MS{B}_G^{op}$ which is identical on objects and switches two legs of a morphism. With Cartesian product of $G$-sets, one can verify that they make $\MS{B}_G$ into a symmetric monoidal category with duality, that is, we have natural isomorphism
\[ \label{eq-dualBG}
\MS{B}_G(X \times Y,Z) \cong \MS{B}_G(X, DY \times Z)
\]

\begin{defi}\label{Def-Mackey}
The category of \textbf{$G$-Mackey functors} $Mack_G$ ($G$ will be omitted if the group is clear) is the category of contravariant enriched additive functor from $\MS{B}_G$ to $Ab$, the category of abelian groups.
\end{defi}

Throughout this paper, all Mackey functors will be presented with underline, i.e. $\UM$ and $\UN$.

The category $Mack_G$ is an abelian category, and all operations are done levelwise.

We can think of a Mackey functor $\UM$ as an assignment that for each orbit $G/H$ we have an abelian group $\UM(G/H)$ and morphisms between $G/H$ and $G/K$ gives structure maps of a Mackey functor. Among these maps, there are a few with significant importance:
\begin{itemize}
\item \textbf{Restrictions} Let $K \subset H \subset G$ be subgroups of $G$, then the map
    \[
    G/K \xleftarrow{id} G/K \twoheadrightarrow G/H
    \]
    induces a homomorphism $Res^H_{K}:\UM(G/H) \rightarrow \UM(G/K)$. These maps are called restrictions.
\item \textbf{Transfers} Let $K \subset H \subset G$ be subgroups of $G$, then the map
    \[
    G/H \twoheadleftarrow G/K \xrightarrow{id} G/K
    \]
    induces a homomorphism $Tr^H_{K}:\UM(G/H') \rightarrow \UM(G/H)$. These maps are called transfers.

\item \textbf{Weyl group actions} Let $H \subset G$ be a subgroup of $G$. Then given an element $\gamma \in W_G(H)$, the Weyl group of $H$ in $G$, the map
    \[
    G/H \xleftarrow{\gamma} G/H \xrightarrow{id} G/H
    \]
    induces a left action of $W_G(H)$ on $\UM(G/H)$.
\end{itemize}

These structure maps are required to satisfied certain compatibility conditions, by the definition of Mackey functor as a functor. In fact, it is sufficient to construct a Mackey functor $\UM$ by constructing all $\UM(G/H)$ and all restrictions, transfers and Weyl group actions in between. An equivalent but more concrete definition of Mackey functor along this line is \cite[1.1.2]{MazurMackey}. Given $K \subset H$, an important compatibility condition is the following:
\[
Res^H_K(Tr^H_K)(x) = \sum_{\gamma \in W_K(H)} \gamma(x)
\]

A common way of describing a Mackey functor is the \textbf{Lewis diagram} \cite{Lewis:ROG}. Let $\UM$ be a Mackey functor. We will put $\UM(G/G)$ on the top and $\UM(G/e)$ on the bottom. Thus restrictions are maps going downwards and transfers are maps going upwards. If $G$ is abelian, Weyl group action will be indicated by $G$-module structure on each $\UM(X)$, otherwise will be omitted. For example, a Lewis diagram of a $C_{p^2}$-Mackey functor $\UM$ is the following:
\[
\xymatrix
@R=7mm
@C=10mm{
\text{$\UM(C_{p^2}/C_{p^2})$}\ar@/_/[d]_{Res^{p^2}_p}\\
\text{$\UM(C_{p^2}/C_p)$}\ar@/_/[d]_{Res^p_e} \ar@/_/[u]_{Tr^{p^2}_p}\\
\text{$\UM(C_{p^2}/e)$}\ar@/_/[u]_{Tr^p_e}
}
\]

\begin{exam}\label{exam-Burnside}
Given a finite $G$-set $X$, the representable functor $\MS{B}_G(-,X)$ gives a Mackey functor $\U{A}_X$, the \textbf{Burnside Mackey functor} of $X$. When $X = G/G$, we write $\U{A}$ for $\U{A}_X$. For any finite group $G$, $\UA(X)$ is the free abelian group generated by $G$-sets over $X$, with restrictions given by pullbacks of $G$-sets and transfers given by compositions.
\end{exam}

\begin{exam}\label{exam-fixMF}
Given a $G$-module $M$, the \textbf{fixed point Mackey functor} $\UM$ is defined as $\UM(G/H) = M^{H}$, the subgroup of $M$ fixed by $H \subset G$. Restrictions are inclusions of fixed points, and transfers are summations over cosets. The Mackey functor $\UZ$, the fixed point Mackey functor of trivial $G$-module $\ZZ$, plays an important role through this paper. $\U{0}$ will stand for the trivial Mackey functor.
\end{exam}

\begin{exam}\label{exam-orbitMF}
Given a $G$-module $M$, the \textbf{orbit Mackey functor} $\U{O}(M)$ is defined as $\U{O}(M)(G/H) \cong M/H$, the orbit of $H \subset G$. Transfers are quotient maps, and restrictions are summations over representatives.
\end{exam}

\subsection{The box product and $\UZ$-modules}
The advantage of defining the category of Mackey functors as a functor category is that we can define a symmetric monoidal product and internal Hom by categorical tools.

\begin{defi}\label{Def-Box}
Given Mackey functors $\UM$ and $\UN$, their \textbf{box product} $\UM \square \UN$ is the left Kan extension \cite[X.3]{MacLane} of the following diagram.
\[
\xymatrix
@R=7mm
@C=20mm
{
\MS{B}_G \times \MS{B}_G \ar[r]^{\UM(-) \otimes \UN(-)} \ar[d]_{(-) \times (-)} & Ab\\
\MS{B}_G \ar@{-->}[ur]_{\UM \square \UN} &
}
\]
Where the horizontal arrow is $(X,Y) \mapsto \UM(X) \otimes \UM(Y)$ and the vertical arrow is $(X,Y) \mapsto X \times Y$, and $\UM \square \UN$ is the left Kan extension
\[
Lan_{(-)\times (-)}\UM(-) \otimes \UN(-).
\]

For any Mackey functor $\UM$, the functor $-\square \UM$ has a right adjoint $\U{Hom}(\UM,-)$, which is the internal Hom of Mackey functors.
\end{defi}

This process is known as the Day convolution, as it was first studied by Brian Day in \cite{Day}. He showed that the Day convolution gives a closed symmetric monoidal structure on the functor category. In our case, it means that the box product is associative, commutative and unital with unit $\UA$. We use $Hom(\UM,\UN)$ for the abelian group of natural transformation from $\UM$ to $\UN$. The internal Hom Mackey functor and the Hom abelian group is related by the following:
\[
\U{Hom}(\UM,\UN)(G/G) \cong Hom(\UM,\UN)
\]

One can show that the box product and the internal Hom commutes with the forgetful functor to a subgroup, thus specially, $(\UM \square \UN)(G/e) \cong \UM(G/e) \otimes \UN(G/e)$ and $\U{Hom}(\UM,\UN)(G/e) \cong Hom_{Ab}(\UM(G/e),\UN(G/e))$, where $\otimes$ and $Hom_{Ab}$ are tensor and Hom of abelian groups.

Before we start to do computation, we need an important operation, the lift.

\begin{defi}\label{def-lift}
Given a Mackey functor $\UM$ and a $G$-set $X$, $\UM_{X}$, the \textbf{lift Mackey functor} of $\UM$ by $X$ is defined as the composition of functors
\[
    \MS{B}_G \xrightarrow{- \times X} \MS{B}_G \xrightarrow{\UM} Ab
\]
\end{defi}

Using the monoidal structure on $\MS{B}_G$, one can verify that two definitions of $\UA_{X}$ agree. Furthermore, there is a natural isomorphism in $\UM$
\[
    \UM \square \UA_{X} \cong \UM_X \cong \U{Hom}(\UA_X,\UM)
\]

Now we are ready for the computational aspect of the box product and the internal Hom.

\begin{itemize}
\item \textbf{Box Product.} In \cite[1.2.1]{MazurMackey}, Mazur shows that for $G = C_{p^n}$, the box product $\UM \square \UN$ can be computed inductively as follows:

    Given $H \subset G$ and $K$ the maximal proper subgroup of $H$, we have
    \[
        (\UM \square \UN)(G/H) \cong (\UM(G/H) \otimes \UN(G/H) \oplus (\UM \square \UN)(G/K)/_{W_H(K)})/_{FR}
    \]
    Where the transfer map $Tr^{H}_{K}$ is the quotient map onto $\UM \square \UN)(G/K)/_{W_H(K)}$.

    The Frobenius reciprocity $FR$ (see also \ref{def-Green2}) is generated by elements of the form, where $K'$ varies as subgroups of $K$:
    \[
        a \otimes Tr^{H}_{K'}(b) - Tr^{H}_{K}Tr^{K}_{K'}(Res^{H}_{K'}(a) \otimes b)
    \]
    and
    \[
        Tr^{H}_{K'}(c) \otimes d - Tr^{H}_{K}Tr^{K}_{K'}(c \otimes Res^{H}_{K'}(d))
    \]

    The Weyl group acts diagonally on $\UM(G/H) \otimes \UN(G/H)$.

    The restriction map $Res^{H}_{K}$ is defined diagonally on $\UM(G/H) \otimes \UN(G/H)$, and for $x \in (\UM \square \UN)(G/K)/_{W_H(K)}$, by
    \[
        Res^{H}_{K}(Tr^{H}_{K}(x)) = \sum_{\gamma \in W_H(K)} \gamma (x)
    \]

    Intuitively, we can think of the box product as the smallest Mackey functor obtained from coefficient system $\UM(G/H) \otimes \UN(G/H)$ by adding all transfers via Frobenius reciprocity.
\item \textbf{Internal Hom.} By \cite[Definition~1.3]{Lewis:Green}, internal Hom $\U{Hom}(\UM,\UN)$ can be computed as follows:
    \[
    \U{Hom}(\UM,\UN)(X) \cong Hom(\UM,\UN_{X})
    \]
    With restrictions and transfers given by structure maps induced by the factor $X$.
\end{itemize}

\begin{exam}\label{exam-box}
Let $G = C_{p^2}$ and $\UM$ be the Mackey functor with the following Lewis diagram:
\[
\xymatrix
@R=7mm
@C=10mm{
\ZZ \ar@/_/[d]_{1}\\
\ZZ \ar@/_/[d]_{p} \ar@/_/[u]_{p}\\
\ZZ \ar@/_/[u]_{1}
}
\]
We can compute $\UM \square \UM$ by Mazur's formula, and see that it is $\UN_1 \oplus \UN_2$, where $\UN_1$ is
\[
\xymatrix
@R=7mm
@C=10mm{
\ZZ \ar@/_/[d]_{p}\\
\ZZ \ar@/_/[d]_{p} \ar@/_/[u]_{1}\\
\ZZ \ar@/_/[u]_{1}
}
\]
and $\UN_2$ is
\[
\xymatrix
@R=7mm
@C=10mm{
\ZZ/p \ar@/_/[d]\\
0     \ar@/_/[d] \ar@/_/[u]\\
0 \ar@/_/[u]
}
\]
This example is generalized in Example \ref{exam-tor}.
\end{exam}

\begin{defi}\label{def-Green}
A \textbf{Green functor} $\UR$ is a Mackey functor $\UR$ equipped with a structure of a monoid in $(Mack_G,\square,\UA)$.
\end{defi}

We can show that it is equivalent to the following definition.

\begin{defi}\label{def-Green2}
A \textbf{Green functor} $\UR$ is a Mackey functor that
\begin{itemize}
\item $\UR$(G/H) is a ring for each subgroup $H \subset G$ and $Res^H_K$ are ring homomorphisms.
\item Transfers satisfy the Frobenius reciprocity: If $K \subset H \subset G$, then
    \[
        Tr^H_K(a) \cdot b = Tr^H_K(a \cdot Res^H_K(b))
    \]
    for all $a \in \UR(G/K)$ and $b \in \UR(G/H)$
\end{itemize}
\end{defi}

One very important Mackey functor is $\UZ$, the fixed point Mackey functor of the trivial $G$-module $\ZZ$. It is a commutative monoid in $Mack_G$, thus we can use the box product to define a closed symmetric monoidal category $Mod_{\UZ}$, the category of $\UZ$-modules.

\begin{defi}\label{def-Z}
A \textbf{$\UZ$-module} $\UM$ is a Mackey functor $\UM$ equipped with an associative and unital map $\UZ \square \UM \rightarrow \UM$. The $\UZ$-box product $\square_{\UZ}$ is defined using the coequalizer diagram
\[
\xymatrix{
\text{$\UM$} \square \UZ \square \UN \ar@<0.7ex>[r] \ar@<-0.7ex>[r] & \text{$\UM$} \square \UN \ar[r]^{coeq} & \text{$\UM$} \square_{\UZ} \UN
}
\]
$\UM \square_{\UZ} -$ has a right adjoint $\U{Hom}_{\UZ}(\UM,-)$, the internal Hom of $\UZ$-modules, defined as an equalizer
\[
\xymatrix{
\text{$\U{Hom}_{\UZ}(\UM,\UN)$} \ar[r]^{eq} & \text{$\U{Hom}(\UM,\UN)$} \ar@<0.7ex>[r] \ar@<-0.7ex>[r] & \text{$\U{Hom}(\UZ \square \UM,N) \cong \U{Hom}(\UM,\U{Hom}(\UZ,\UN))$}
}
\]
We use $(Mod_{\UZ},\square_{\UZ},\U{Hom}_{\UZ},\ZZ)$ for the closed symmetric monoidal category of $\UZ$-modules.
\end{defi}

\begin{rmk}\label{rmk-Z}
Since $\UZ \square \UZ \cong \UZ$, the forgetful functor $Mod_{\UZ} \rightarrow Mack_G$ is fully faithful, therefore being a $\UZ$-module is a condition rather than a structure, and $\UZ$-module maps are no different than Mackey functor maps. The condition of being a $\UZ$-module is made clear by decoding the box product $\UZ \square \UM$ using the concrete formula. We see that for $H \subset G$ and $K$ varies as subgroups of $H$,
\[
    (\UZ \square \UM)(G/H) \cong \UM(G/H)/{([H:K]x - Tr^H_K(Res^H_K(x)))}
\]
Therefore, $\UZ \square \UM$ is simply a levelwise quotient of $\UM$ by equating multiplication of $x$ by index of subgroup and transfer of restriction of $x$. Thus $\UM$ is a $\UZ$-module if and only if the quotient map induced by the unit map $\UA \rightarrow \UZ$
\[
\UM \cong \UA \square \UM \rightarrow \UZ \square \UM
\]
is an isomorphism. This condition is called "cohomological" in classical literatures \cite[Proposition~16.3]{Thevenaz-Webb}.
\end{rmk}

In fact, we can define the closed symmetric monoidal category of $\UZ$-modules in the same way as the definition of Mackey functors. This definition is useful in the proof of Corollary \ref{cor-equivalence}.

\begin{defi}\label{def-BZG}
The \textbf{Burnside $\UZ$-category} for a finite group $G$, $\MS{B} \UZ_{G}$ is the following:
\begin{itemize}
    \item Objects: Finite $G$-sets.
    \item Morphisms: $\MS{B} \UZ_{G}(X,Y) := Hom_G(\ZZ[X],\ZZ[Y])$.
    \item Composition: Composition of $G$-maps.
\end{itemize}
\end{defi}

There is a functor $Q:\MS{B}_G \rightarrow \MS{B} \UZ_G$, which carries exactly the same information as the unit map $\UA \rightarrow \UZ$. $Q$ is identity on objects and sends a span $X \xleftarrow{f} Z \xrightarrow{g} Y$ to the composition $\ZZ[X] \xrightarrow{f^*} \ZZ[Z] \xrightarrow{g_*} \ZZ[Y]$, where $f^*$ is defined using coinduction $\ZZ[X] \cong Set(X,\ZZ)$ and $g_*$ is defined using induction $\ZZ[X] \cong \oplus_{x \in X} \ZZ$.

\begin{prop}\label{prop-FunctorZmod}
A Mackey functor $\UM$ is a $\UZ$-module if and only if it factor through $Q:\MS{B}_G \rightarrow \MS{B} \UZ_G$. Therefore the category of $\UZ$-modules is isomorphic to the category of additive enriched contravariant functors from $\MS{B} \UZ_G$ to $Ab$.
\end{prop}

\begin{proof}
We only need to prove that the composition $\MS{B}\UZ_G(Q(-),G/G)$ is $\UZ$, and the result follows formally. Let $\UM := \MS{B}\UZ_G(Q(-),G/G)$, we see that $\UM(G/H) \cong Hom_G(\ZZ[G/H],\ZZ) \cong \ZZ$ and if $K \subset H \subset G$, $Res^H_K$ is induced by quotient map $G/K \rightarrow G/H$, by the definition of $Q$, therefore is an isomorphism.
\end{proof}

Since $\UM \square \UZ$ is the same as quotient $\UM(G/H)$ by the cohomology condition in Remark \ref{rmk-form-Z}, the left Kan extension $Lan_{Q}\UM$ is isomorphic to $\UM \square \UZ$ and therefore the Day convolution on the functor category $\MS{B} \UZ_G \rightarrow Ab$ gives the same box product $\square_{\UZ}$ and internal Hom $\U{Hom}_{\UZ}$ on $\UZ$-modules.

\begin{exam}\label{exam-gpcoh}
$\UZ$-modules occur naturally in group homology and cohomology. Given a finite group $G$ and $M$ a $G$-module, $H_{i}(G;M)$ the group homology with coefficient $M$ can be computed as the $i$-th left derived functors of the orbit functor $(-)/G$, and $H^{i}(G;M)$ the group cohomology with coefficient $M$ can be computed as the $i$-th right derived functors of the fixed point functor $(-)^G$. Since the forgetful functors has both left and right adjoint, it preserves projective and injective resolutions. Take a projective resolution of $P_{\bullet} \rightarrow M$ in $G$-modules, then take the orbit Mackey functor and take differentials, we obtain a Mackey functor structure on group homologies for $H_{i}(K;M)$ for all $K \subset G$. Similarly, we obtain a Mackey functor structure on group cohomologies for $H^{i}(K;M)$ for all $K \subset G$. Since orbit and fixed point Mackey functors are $\UZ$-modules, the group (co)homology Mackey functors are also $\UZ$-modules.
\end{exam}

There are some special $\UZ$-modules that will be important for our computation, namely forms of $\UZ$.

\begin{defi}\label{def-form-Z}
A $\UZ$-module $\UM$ is a \textbf{form of $\UZ$} if $\UM(G/H) \cong \ZZ$ with trivial Weyl group action for all $H \subset G$.
\end{defi}

\begin{rmk}\label{rmk-form-Z}
If $G = C_{p^n}$, we see that in adjacent levels of Lewis diagram of a form of $\UZ$, one of the restruction and transfer is an isomorphism and another is a multiplication by $p$. Therefore there are $2^n$ isomorphism classes of forms of $\UZ$.
\end{rmk}

\begin{defi}\label{def-form-Z-index}
For $G = C_{p^n}$, let $\UZ_{t_1,t_2,...,t_n}$, where $t_i = 0$ or $1$, be the form of $\UZ$ such that $Res^{p^i}_{p^{i-1}} = p^{t^i}$ for $1 \leq i \leq n$. Let $\UB_{t_1,t_2,...,t_n}$ be the cokernel of $\UZ_{t_1,t_2,...,t_n} \rightarrow \UZ$, where the map is an isomorphism on $G/e$-level.
\end{defi}

\begin{exam}

$\UZ_{1,0}$ for $C_{p^2}$ has the following Lewis diagram
\[
\xymatrix
@R=7mm
@C=10mm{
\ZZ \ar@/_/[d]_{1}\\
\ZZ \ar@/_/[d]_{p} \ar@/_/[u]_{p}\\
\ZZ \ar@/_/[u]_{1}
}
\]

$\UB_{1,0}$ for $C_{p^2}$ has the following Lewis diagram
\[
\xymatrix
@R=7mm
@C=10mm{
\ZZ/p \ar@/_/[d]_{1}\\
\ZZ/p \ar@/_/[d] \ar@/_/[u]_0\\
0 \ar@/_/[u]
}
\]

\end{exam}

\begin{exam}\label{exam-box2}
In this notation, Example \ref{exam-box} says that
\[
\UZ_{1,0} \square_{\UZ} \UZ_{1,0} \cong \UZ_{1,1} \oplus \UB_{0,1}
\]
\end{exam}

\begin{exam}
If $\UM$ is a form of $\UZ$, then $\U{Hom}_{\UZ}(\UM,\UZ) \cong \UZ$:
\[
\U{Hom}_{\UZ}(\UM,\UZ)(G/H) \cong Hom(i^*_{H}\UM,i^*_H\UZ) \cong \ZZ
 \]
which is generated by the map that is an isomorphism on $G/e$-level, for every \linebreak $H \subset G$. Since the restrictions in the internal Hom are forgetful maps, we see that all restrictions in $\U{Hom}_{\UZ}(\UM,\UZ)$ are isomorphisms, therefore $\U{Hom}_{\UZ}(\UM,\UZ) \cong \UZ$.
\end{exam}

The following lemma is simple but useful.

\begin{lem}\label{lem-torsion}
Let $\UM$ be a $\UZ$-module that $\UM(G/e)$ is torsion, then $\UM(G/H)$ is torsion for all $H \subset G$.
\end{lem}

\begin{proof}
Let $x \in \UM(G/H)$, the cohomological condition implies that \[|H|x = Tr^H_e(Res^H_e(x))\]
is torsion.
\end{proof}

We will call a $\UZ$-module $\UM$ torsion, if $\UM(G/e)$ is a torsion abelian group.

We close this section with some discussion about Tambara functors. The Tambara structure will not affect computations in this paper, however we need the fact that $\UZ$ is a Tambara functor in the proof of Corollary \ref{cor-equivalence}.

\begin{defi}[{\cite[Definition~2.3]{Mazur:EquiTensor}}] \label{def-Tambara}
A \textbf{$G$-Tambara functor} $\UR$ is a commutative Green functor $\UR$ with norm maps
\[
    N_K^H: \UR(G/K) \rightarrow \UR(G/H)
\]
for subgroups $K \subset H \subset G$. These are maps of multiplicative monoids that satisfy formulas about norm of sums and norm of transfers.
\end{defi}

\begin{prop}[{\cite[Example~1.4.5]{MazurMackey}}]\label{prop-fixTambara}
A fixed point Green functor $\UR$ is naturally a Tambara functor. Given $K \subset H \subset G$ the norm map $N^H_K:R^K \rightarrow R^H$ is defined by
\[
    N^H_K(x) = \prod_{\gamma \in W_H(K)} \gamma x
\]
\end{prop}

\section{Homological Algebra of $\UZ$-Modules}\label{sec-HA}

\subsection{The internal homological algebra}
One important feature of the category of Mackey functors is that it has enough projective and injective objects, therefore combining with the box product and the internal Hom, we can define "internal" derived functors $\U{Tor}$ and $\U{Ext}$. A more detailed argument about derived functors in $Mack_G$ is in \cite[Section~4]{LewisMandell}.

\begin{prop}\label{prop-proj-inj}
The category of Mackey functors has enough projective and injective objects.
\end{prop}

\begin{proof}
We first prove there are enough projective objects. By (enriched) Yoneda lemma \cite[Section~2.4]{Kelly}, we have $Hom(A_{X},\UM) \cong \UM(X)$. Since surjection is defined levelwisely, we see that for any $G$-set $X$, $A_{X}$ is projective and coproducts of them form enough projective objects.

For enough injective objects, given an abelian group $E$, let $\U{I}(X,E)$ be the Mackey functor $Hom_{Ab}(\MS{B}_{G}(X,-),E)$. By Yoneda lemma again we have
\[Hom(\UM,\U{I}(X,E)) \cong Hom_{Ab}(\UM(X),E).\] Thus products of Mackey functors of the form $\U{I}(X,E)$ where $E$ is an injective abelian group forms enough injective objects.
\end{proof}

The following proposition is standard by using induction and coinduction from Mackey functor to $\UR$-modules.

\begin{prop}\label{prop-proj-inj-module}
Given a commutative Green functor $\UR$, the category $Mod_{\UR}$ has enough projective and injective objects.
\end{prop}

\begin{defi}\label{def-ext-tor}
Let $\U{Ext}_{\UR}^i(\UM,\UN)$ be the $i$-th right derived functor of $\U{Hom}_{\UR}(\UM,-)$. Let $\U{Tor}^{\UR}_i(\UM,\UN)$ be the $i$-th left derived functor of $\UM \square_{\UR} -$.
\end{defi}

By standard argument, we can use either projective resolutions of $\UM$ or injective resolutions of $\UN$ to compute $\U{Ext}_{\UR}^i(\UM,\UN)$. But in practice, projective objects and resolutions are easier to describe than injective objects and resolutions, therefore we will mostly use projective version to compute homological algebra.

These derived functors have all the expected basic properties.

\begin{prop}\label{prop-LES}
$\U{Tor}^{\UR}_*$ and $\U{Ext}^*_{\UR}$ are naturally $\UR$-modules, and converts short exact sequences in each variable into long exact sequences. And there are natural isomorphisms
\[
\U{Tor}^{\UR}_0(\UM,\UN) \cong \UM \square_{\UR} \UN
\]
and
\[
\U{Ext}^0_{\UR}(\UM,\UN) \cong \U{Hom}_{\UR}(\UM,\UN)
\]
\end{prop}

\begin{rmk}
As in the classical case, $\U{Ext}^1_{\UR}(\UM,\UN)$ can be interpreted as isomorphism classes of extensions between $\UN$ and $\UM$. More precisely, $\U{Ext}^1_{\UR}(\UM,\UN)(G/H)$ is the abelian group of extensions of $\UR$-modules between $i^*_H(\UN)$ and $i^*_H(\UM)$. Therefore, $\U{Ext}_{\UR}$ computation can help in resolving extension problems in spectral sequences of $\UR$-modules. An example of this application is \cite[Section~6.2]{Hill:eta}
\end{rmk}

In \cite{Greenlees-Mackey}, Greenlees proves that a Mackey functor has projective dimension either $0,1$ or $\infty$, which means that in computing $\U{Ext}_{\UA}^*$ and $\U{Tor}^{\UA}_*$, we usually encounter infinitely many nontrivial terms. However, the story is very different for the category $Mod_{\UZ}$.

\begin{thm}[{\cite[Theorem~1.7]{BSW}}{\cite{Arnold:Cyclic}}]\label{thm-dim3}

If $G$ is cyclic and finite, then $Mod_{\UZ}$ has global cohomological dimension $3$. More precisely, any $\UZ$-module has a projective resolution of length at most $3$.

\end{thm}

This means that when computing in $\UZ$-modules, we will only have nontrivial derived functors ranging from diemnsion $0$ to $3$.

Projective objects in $\UZ$-modules are easy to describe. They are fixed point Mackey functors of permutation $G$-modules.

\begin{prop}\label{prop-proj-Z}
For any $\UZ$-module $\UM$ there is a surjection $\U{P} \rightarrow \UM$ such that $\U{P}$ is a projective $\UZ$-module and is of the form $\U{\ZZ[X]}$ for some $G$-set $X$.
\end{prop}

\begin{proof}
Since $\UZ \square -$ is right exact, and for $\UZ$-module $\UM$, $\UM \square \UZ \cong \UM$, direct sum of $\UZ \square \UA_{G/H}$ forms enough projective objects in $Mod_{\UZ}$. However
\[
\UZ \square \UA_{G/H} \cong \UZ_{G/H} \cong \U{\ZZ[G/H]}
\]
\end{proof}

We can also define a levelwise $Hom$ and $Ext$, using the corresponding notations in abelian groups and the fact that $\MS{B}_G$ is self-dual. We will see how this definition interacts with the derived functors above.

\begin{defi}\label{def-Hom-Ext-L}
Given a Mackey functor $\UM$ and an abelian group $A$, $Hom_L(\UM,A)$ is the composition of functors
\[
\MS{B}_G \xrightarrow{D} \MS{B}_G \xrightarrow{\UM} Ab \xrightarrow{Hom_{Ab}(-,A)} Ab
\]
We use $\UM^*$ for $Hom_L(\UM,\ZZ)$.

Similarly, $Ext_L(\UM,A)$ is the composition
\[
\MS{B}_G \xrightarrow{D} \MS{B}_G \xrightarrow{\UM} Ab \xrightarrow{Ext_{Ab}(-,A)} Ab
\]
We use $\UM^E$ for $Ext_L(\UM,\ZZ)$.

Here $L$ stands for levelwise.
\end{defi}

These constructions have no derived categorical interpretations, but they will show up in computations of the internal $\U{Ext}_{\UZ}$.

\begin{exam}
By the above notations, we have
\[
\UZ_{t_1,t_2,...,t_n}^* \cong \UZ_{1-t_1,1-t_2,...,1-t_n}
\]
\end{exam}

\subsection{Pullback from quotient groups}
If $K$ is a normal subgroup of $G$, then we can consider the subcategory of $G$-$\UZ$-modules that are pullbacks of $G/K$-$\UZ$-modules. It turns out that homological invariants of this subcategory can be computed in $G/K$-$\UZ$-modules.

\begin{defi}\label{def-pullMF}
Let $K$ be a normal subgroup of $G$. Given a finite $G/K$-set $X$, the quotient map $G \rightarrow G/K$ gives a $G$-action on $X$, thus induces a functor
\[
    \psi : \MS{B} \UZ_{G/K} \rightarrow \MS{B} \UZ_{G}.
\]
Given a $G/K$-$\UZ$-module $\UM$, we define the \textbf{pullback of $\UM$},
\[
\Psi^*_K(\UM) = Lan_{\psi}\UM,
\]
the left Kan extension of $\UM$ along $\psi$, as in the following diagram
\[
\xymatrix
@R=7mm
@C=20mm
{
\MS{B} \UZ_{G/K} \ar[r]^{\UM} \ar[d]_{\psi} & Ab\\
\MS{B} \UZ_G \ar@{-->}[ur]_{\Psi^*_K(\UM)} &
}
\]
\end{defi}

Since left Kan extension preserves representable functors, we have
\begin{prop}\label{prop-pullFix}
For a $G/K$-set $X$,
\[
    \Psi_K^*(\U{\ZZ[X]}) \cong \U{\ZZ[X]}
\]
where we treat $X$ as a $G$-set on the right hand side.
\end{prop}

By the definition of Kan extension, $\Psi^*_K$ is the left adjoint of of the composition functor
\[
    \psi^*: Mod_{\UZ}^G \rightarrow Mod_{\UZ}^{G/K}
\]
that sends $\UM$ to $\UM \circ \psi$. If $K = G$, then by \cite[Proposition~2.10]{BlumHill:Rightadj}, $\Psi^*_G$ is strongly symmetric monoidal. A slight modification of their proof works in general.

\begin{prop}\label{prop-pullmonoidal}
Given a normal subgroup $K \subset G$, the pullback functor
\[
    \Psi^*_K : Mod_{\UZ}^{G/K} \rightarrow Mod_{\UZ}^G
\]
is strongly symmetric monoidal. That is, we have a natural isomorphism
\[
    \Psi^*_K(\UM) \square_{\UZ} \Psi^*_K(\UN) \cong \Psi^*_K(\UM \square_{\UZ} \UN).
\]
\end{prop}

We need a lemma first. This lemma generalizes the fact that the $G/G$-level of the internal Hom Mackey functor is the group of natural transformation.

\begin{lem}\label{lem-pull}
For any $G$-$\UZ$-module $\UX$ and $G/K$-$\UZ$-module $\UN$, we have
\[
\psi^*(\U{Hom}_{\UZ}^G(\Psi^*_K(\UN),\UX)) \cong \U{Hom}_{\UZ}^{G/K}(\UN,\psi^*(\UX)).
\]
\end{lem}

\begin{proof}
Evaluating the left hand side at a $G/K$-set $T$., we have
\begin{align*}
    \psi^*(\U{Hom}_{\UZ}^G(\Psi^*_K(\UN),\UX))(T) & \cong Hom_{\UZ}^{G/K}(\U{\ZZ[T]},\psi^*(\U{Hom}_{\UZ}^G(\Psi^*_K(\UN),\UX)))\\
                                                  & \cong Hom_{\UZ}^G(\U{\ZZ[T]},\U{Hom}_{\UZ}^G(\Psi^*_K(\UN),\UX))\\
                                                  & \cong \U{Hom}_{\UZ}^G(\Psi^*_K(\UN),\UX)(T)\\
                                                  & \cong Hom_{\UZ}^G(\Psi^*_K(\UN),\UX_T)\\
                                                  & \cong Hom_{\UZ}^{G/K}(\UN,\psi^*(\UX_T))
\end{align*}
Since $\psi:\MS{B}_{G/K} \rightarrow \MS{B}_{G}$ is strongly monoidal with respect to Cartesian product of $G$ and $G/K$-sets, we have a commutative diagram
\[
\xymatrix
{
    \MS{B}_{G/K} \ar[r]^{- \times T} \ar[d]^{\psi} & \MS{B}_{G/K} \ar[d]^{\psi} & \\
    \MS{B}_{G}   \ar[r]^{- \times T}               & \MS{B}_{G} \ar[r]^{\UX}  & Ab
}
\]
One path of composition is $\psi^*(\UX_T)$ and another is $\psi^*(\UX)_T$, thus
\[
    \psi^*(\UX_T) \cong \psi^*(\UX)_T.
\]
So we have
\begin{align*}
    Hom_{\UZ}^{G/K}(\UN,\psi^*(\UX_T)) & \cong Hom_{\UZ}^{G/K}(\UN,\psi^*(\UX)_T)\\
                                     & \cong \U{Hom}_{\UZ}^{G/K}(\UN,\psi^*(\UX))(T)
\end{align*}
And it is the right hand side evaluating at $T$.
\end{proof}
\begin{proof}[Proof of Proposition \ref{prop-pullmonoidal}]
Given any $G$-$\UZ$-module $\U{X}$, we have
\begin{align*}
    Hom_{\UZ}^G(\Psi^*_K(\UM) \square_{\UZ} \Psi^*_K(\UN),\UX) & \cong Hom_{\UZ}^G(\Psi^*_K(\UM),\U{Hom}_{\UZ}^G(\Psi^*_K(\UM),X))\\
                                                                & \cong Hom_{\UZ}^{G/K}(\UM,\psi^*(\U{Hom}_{\UZ}^G(\Psi^*_K(\UN),\UX))).\\
\end{align*}
By the lemma above, it is
\begin{align*}
    Hom_{\UZ}^{G/K}(\UM,\U{Hom}_{\UZ}^{G/K}(\UN,\psi^*(\UX))) & \cong Hom_{\UZ}^{G/K}(\UM \square_{\UZ} \UN, \psi^*(\UX))\\
                                                              & \cong Hom_{\UZ}^{G}(\Psi^*_K(\UM \square_{\UZ} \UN),\UX)
\end{align*}
\end{proof}

As a left adjoint, $\Psi^*_K$ is right exact. However, in the case of $\UZ$-module, it is actually an exact functor.

\begin{lem}\label{lem-exact}
Given a $G/K$-$\UZ$-module $\UM$, we have
\[
    \UM \cong \psi^*(\Psi^*_K(\UM)).
\]
That is, the unit of the adjunction is an isomorphism. Therefore, $\Psi^*_K$ is an exact functor.
\end{lem}

\begin{proof}
First, assume that $\UM$ is a representable functor, that is,
\[
\UM \cong \U{\ZZ[X]}
\]
for some $G/K$-set $X$. Then $\Psi^*_K(\UM) \cong \U{\ZZ[X]}$ by treating $X$ as a $G$-set. Now since $X$ is a $G/K$-set, $\ZZ[X]^K \cong \ZZ[X]$, so we ahve
\[
\U{\ZZ[X]} \cong \psi^*(\Psi^*_K(\U{\ZZ[X]})).
\]

Now for a general $G/K$-$\UZ$-module $\UM$, we can form a projective resolution $\U{P}_{\bullet}$ of $\UM$ by representable functors. Since $\Psi^*_K$ is right exact and $\psi_*$ is exact, we see that
\[
\psi^*(\Psi^*_K(\UM)) \cong coker(\psi^*(\Psi^*_K(\U{P}_0 \leftarrow \U{P}_1))) \cong coker(\U{P}_0 \leftarrow \U{P}_1) \cong \UM.
\]
\end{proof}

If $G = C_{p^n}$, and $K = C_{p^k}$, then we can describe $\Psi^*_K$ explicitly.

\begin{prop}\label{prop-pullP}
For $G = C_{p^n}$ and $K = C_{p^k}$, we have
\[
\Psi^*_K(\UM) \cong \left\{ \begin{array}{ll}
    \UM((G/K)/(K'/K)) & \textrm{for $K \subset K'$}\\
    \UM((G/K)/e)      & \textrm{for $K' \subset K$},
    \end{array} \right.
\]
with restriction $Res^K_e$ isomorphism.
\end{prop}

Since the pullback functor $\Psi^*_{K}: Mod^{G/K}_{\UZ} \rightarrow Mod^{G}_{\UZ}$ is strongly monoidal, exact, and preserves projective objects, it preserves all homological invariants.
\begin{cor}\label{cor-pullhom}
Let $K$ be a normal subgroup of $G$ and $\UM,\UN$ be $G/K$-$\UZ$-modules, we have
\[
    \Psi^*_K(\U{Tor}^{\UZ}_*(\UM,\UN)) \cong \U{Tor}^{\UZ}(\Psi^*_K(\UM),\Psi^*_K(\UN))
\]
and
\[
    \Psi^*_K(\U{Ext}^*_{\UZ}(\UM,\UN)) \cong \U{Ext}^*_{\UZ}(\Psi^*_K(\UM),\Psi^*_K(\UN)).
\]
\end{cor}

\begin{rmk}
If we replace $\MS{B} \UZ_G$ and $\MS{B} \UZ_{G/K}$ by $\MS{B}_G$ and $\MS{B}_{G/K}$, that is, consider all computation in coefficient $\UA$ instead of $\UZ$, then Proposition \ref{prop-pullFix} holds, if we replace fixed point Mackey functors by generalized Burnside Mackey functors. Therefore, Proposition \ref{prop-pullmonoidal} also holds. However, Lemma \ref{lem-exact} fails even for $\UM \cong \UA$. If we want to compute in $\UA$ coefficient, then we can first compute in $\UZ$, and make use of the splitting of augmentation ideal of $\UA \rightarrow \UZ$ in \cite{GreenMay:Mackey}. We will not pursue this direction here.
\end{rmk}

\subsection{Computation}
Now we start to compute $\U{Ext}_{\UZ}^*$ and $\U{Tor}^{\UZ}_*$ for $G = C_{p^n}$ and various $\UZ$-modules. Some computation and proof can be simplified, once we introduce equivariant Eilenberg-Mac Lane spectra, but computations here do not rely on equivariant stable homotopy theory.

we start with $G = C_p$. Notations of $\UZ$-modules used in the following examples are defined in Definition \ref{def-form-Z-index} and \ref{def-Hom-Ext-L}, and $\gamma$ will be a chosen generator for $G$.

\begin{exam}
For $G = C_p$, the $\UZ$-module $\UB_{1}$ has the following Lewis diagram:
\[
\xymatrix
@R=7mm
@C=10mm{
\ZZ/p \ar@/_/[d] \\
0     \ar@/_/[u]
}
\]

It has a projective resolution
\[
\UB_{1} \leftarrow \UZ \leftarrow \U{\ZZ[C_p]} \leftarrow \U{\ZZ[C_p]} \leftarrow \UZ
\]
In Lewis diagram, the resolution is the following:
\[
\xymatrix
@R=7mm
@C=10mm{
\ZZ/p \ar@/_/[d] & \ZZ \ar@/_/[d]_{1} \ar[l]_{1} & \ZZ \ar@/_/[d]_{\Delta} \ar[l]_{p} & \ZZ \ar@/_/[d]_{\Delta} \ar[l]_{0} & \ZZ \ar@/_/[d]_{1} \ar[l]_1 \\
0     \ar@/_/[u] & \ZZ \ar@/_/[u]_{p} \ar[l] & \ZZ[C_p] \ar@/_/[u]_{\nabla} \ar[l]_{\nabla} & \ZZ[C_p] \ar@/_/[u]_{\nabla} \ar[l]_{1-\gamma} & \ZZ \ar@/_/[u]_{p}\ar[l]_{\Delta}
},
\]

where $\nabla: \ZZ[G] \rightarrow \ZZ$ is the augmentation map defined by $\nabla(1) = 1$ and $\Delta$ is the diagonal embedding defined by $\Delta(1) = \sum\limits_{i = 0}^{p-1} \gamma^i$.

We can compute $\U{Ext}_{\UZ}^*(\UB_{1},\UZ)$ by applying $\U{Hom}_{\UZ}(-,\UZ)$ to the resolution. We then get the following chain complex:
\[
\xymatrix
@R=7mm
@C=10mm{
\ZZ \ar@/_/[d]_{1}  \ar[r]^{1}      & \ZZ \ar@/_/[d]_{\Delta} \ar[r]^{0}             & \ZZ \ar@/_/[d]_{\Delta} \ar[r]^{p}          & \ZZ \ar@/_/[d]_{1}  \\
\ZZ \ar@/_/[u]_{p}  \ar[r]^{\Delta} & \ZZ[C_p] \ar@/_/[u]_{\nabla} \ar[r]^{1-\gamma} & \ZZ[C_p] \ar@/_/[u]_{\nabla} \ar[r]^{\nabla}& \ZZ \ar@/_/[u]_{p}
}
\]

It is simply the projective resolution of $\UB_{1}$ in the opposite direction. Therefore we conclude that
\[
\U{Ext}_{\UZ}^i(\UB_{1},\UZ) = \left\{ \begin{array}{ll}
                                    \UB_{1} & \textrm{for $i = 3$} \\
                                    \U{0} & \textrm{otherwise}
    \end{array} \right.
\]

$\U{Ext}_{\UZ}^*(\UB_{1},\UB_{1})$ can be computed in a similar way. Since $\U{Hom}_{\UZ}(\UZ,\UB_1) \cong \UB_1$ and $\U{Hom}_{\UZ}(\U{\ZZ[C_p]},\UB_1) \cong \U{0}$, we conclude that
\[
\U{Ext}_{\UZ}^i(\UB_{1},\UB_{1}) = \left\{ \begin{array}{ll}
                                    \UB_{1} & \textrm{for $i = 0,3$} \\
                                    \U{0}   & \textrm{otherwise}
                                    \end{array} \right.
\]
By the short exact sequence
\[
    \U{0} \rightarrow \UZ_1 \rightarrow \UZ \rightarrow \UB_{1} \rightarrow \U{0}
\]
we see that
\[
\U{Ext}_{\UZ}^i(\UB_{1},\UZ_1) = \left\{ \begin{array}{ll}
                                        \UB_1 & \textrm{for $i = 1$}\\
                                        \U{0} & \textrm{otherwise}
                                    \end{array} \right.
\]
The nontrivial $\U{Ext}^1_{\ZZ}$ corresponds to the extension
\[
    \U{0} \rightarrow \UZ_1 \rightarrow \UZ \rightarrow \UB_1 \rightarrow \U{0}.
\]
\end{exam}

Then we can do some $C_{p^2}$ computation.
\begin{exam}
For $G = C_{p^2}$, first we consider $\UB_{0,1}$, who has the following Lewis diagram
\[
\xymatrix
@R=7mm
@C=10mm{
\ZZ/p \ar@/_/[d] \\
0 \ar@/_/[d] \ar@/_/[u]\\
0 \ar@/_/[u]
}
\]

Since $\UB_{0,1} \cong \Psi^*_{C_p}(\UB_{1})$, by Corollary \ref{cor-pullhom} we have
\[
\U{Ext}^i_{\UZ}(\UB_{0,1},\UZ) = \left\{ \begin{array}{ll}
                                    \UB_{0,1} & \textrm{if $i = 3$}\\
                                    \U{0}     & \textrm{otherwise.}
                                        \end{array} \right.
\]

For $\UB_{1,1}$, which has the following Lewis diagram
\[
\xymatrix
@R=7mm
@C=10mm{
\ZZ/p^2 \ar@/_/[d]_{1} \\
\ZZ/p \ar@/_/[d] \ar@/_/[u]_{p}\\
0 \ar@/_/[u]
}
\]
a similar projective resolution can be constructed:
\[
\UB_{1,1} \leftarrow \UZ \xleftarrow{\nabla} \U{\ZZ[C_{p^2}]} \xleftarrow{1-\gamma} \U{\ZZ[C_{p^2}]} \xleftarrow{\Delta} \UZ
\]
Therefore we see that
\[
    \U{Ext}^i_{\UZ}(\UB_{0,1},\UZ) = \left\{ \begin{array}{ll}
                                                \UB_{1,1} & \textrm{for $i = 3$}\\
                                                \U{0}     & \textrm{otherwise}
                                            \end{array} \right.
\]
Now we compute $\U{Ext}^*_{\UZ}(\UB_{1,0},\UZ)$. $\UB_{1,0}$ fits into a short exact sequence
\[\label{eq-SES-B}
\U{0} \rightarrow \UB_{0,1} \rightarrow \UB_{1,1} \rightarrow \UB_{1,0} \rightarrow \U{0}
\]
In Lewis diagram it is the following:
\[
\xymatrix
@R=7mm
@C=10mm{
\ZZ/p \ar@/_/[d] \ar@{^{(}->}[r]^{p} & \ZZ/p^2 \ar@/_/[d]_{1} \ar@{->>}[r]       & \ZZ/p \ar@/_/[d]_{1}\\
0 \ar@/_/[u] \ar@/_/[d] \ar[r]       & \ZZ/p \ar@/_/[u]_{p} \ar@/_/[d] \ar[r]^1  & \ZZ/p \ar@/_/[u]_{0} \ar@/_/[d]\\
0 \ar@/_/[u] \ar[r]                  & 0 \ar@/_/[u] \ar[r]                       & 0 \ar@/_/[u]
}
\]

Apply $\U{Ext}^i_{\UZ}(-,\UZ)$ to it, we get a long exact sequence of $\UZ$-modules, which is trivial for $i \neq 3$, and for $i = 3$, by the above computation, we have
\[
\U{0} \rightarrow \U{Ext}^3_{\UZ}(\UB_{1,0},\UZ) \rightarrow \U{Ext}^3_{\UZ}(\UB_{1,1},\UZ) \cong \UB_{1,1} \rightarrow \U{Ext}^3_{\UZ}(\UB_{0,1},\UZ) \cong \UB_{0,1} \rightarrow 0
\]

Therefore, $\U{Ext}^3_{\UZ}(\UB_{1,0},\UZ)$ is not $\UB_{1,0}$, but $\UB_{1,0}^E$ (see Definition \ref{def-Hom-Ext-L}), which is obtained from $\UB_{1,0}$ by flipping restrictions and transfers around. The Lewis diagram of $\UB_{1,0}^E$ is the following:
\[
\xymatrix
@R=7mm
@C=10mm{
\ZZ/p \ar@/_/[d]_0 \\
\ZZ/p \ar@/_/[u]_1 \ar@/_/[d]\\
0 \ar@/_/[u]
}
\]
\end{exam}

%The difference between $\UB_{1,0}$ and $\UB_{0,1}$ or $\UB_{1,1}$ is that $\UB_{1,0}$ has a more complicated projective resolution. Using the short exact sequence \ref{eq-SES-B}, we can use a double complex construction and projective resolutions of $\UB_{0,1}$ and $\UB_{1,1}$ to construct a projective resolution for $\UB_{0,1}$. However, it is a projective of length $4$. By Theorem \ref{thm-dim3}, there exists a projective resolution of length $3$ for $\UB_{1,0}$, but in practice this resolution is very complicated and awkward to use: it amounts to find a projective resolution of length $1$ of the fixed point Mackey functor $\U{\ZZ[C_{p^2}]\langle \frac{1-\gamma}{p} \rangle}$. Even though one can follow proofs in \cite{Arn} to construct such a resolution, it is not the most economical way of do computation.

The phenomenon we see in computing $\U{Ext}^*_{\UZ}(\UM,\UZ)$ where $\UM$ is torsion can be generalized into the following theorem.

\begin{thm}\label{thm-ext}
For $G = C_{p^n}$, if $\UM(G/e) \cong 0$, then
\[
\U{Ext}^i_{\UZ}(\UM,\UZ) = \left\{ \begin{array}{ll}
                                        \UM^E & \textrm{for $i = 3$}\\
                                        \U{0} & \textrm{otherwise}
                                    \end{array} \right.
\]
This isomorphism is natural in $\UM$.
\end{thm}

\begin{proof}
Let $\U{P}_*$ be a projective resolution of $\UM$ with length $3$ (see Theorem \ref{thm-dim3}). We can assume that $\U{P}_* \cong \U{\ZZ[X_*]}$ for a graded $G$-set $X_*$. Since $\U{\ZZ[X]}$ is self-dual in $\UZ$-modules, we have
\[
\U{Hom}_{\UZ}(\U{\ZZ[X]},\UZ) \cong \U{Hom_G(\ZZ[X],\ZZ[G])} \cong \U{\ZZ[X]}
\]
The last isomorphism is not natural. Since everything involved is fixed point Mackey functor, chain maps on $\U{Hom}_{\UZ}(\U{P}_*,\UZ)$ is determined by $G/e$-levels. Therefore, we have
\begin{align*}
    \U{Ext}^i_{\UZ}(\UM,\UZ)(G/H) & \cong \U{H}^i(\U{Hom}_{\UZ}(\U{P}_*,\UZ))(G/H) \\
                                  & \cong H^i(Hom_G(\ZZ[X_*],\ZZ[G])^H) \\
                                  & \cong H^i(Hom_G(\ZZ[G/H],Hom_G(\ZZ[X_*],\ZZ[G])))
\end{align*}
The Mackey functor structure is given by restriction and transfer $G$-module maps between $\ZZ[G/H]$ when $H$ varies as subgroups of $G$. Then by adjunction, we have
\begin{align*}
Hom_G(\ZZ[G/H],Hom_G(\ZZ[X_*],\ZZ[G])) & \cong Hom_G(\ZZ[G/H] \otimes_G \ZZ[X_*],\ZZ[G])\\
                                       & \cong Hom_{\ZZ}(\ZZ[G/H] \otimes_G \ZZ[X_*],\ZZ)
\end{align*}
That means, we can compute $\U{Ext}^i_{\UZ}(\UM,\UZ)$ as follows:
\begin{enumerate}
\item Take the underlying G-modules of projective resolution $\U{P}_*$, which is $\ZZ[X_*]$.
\item Form the orbit Mackey functor (see Example \ref{exam-orbitMF}) $\U{O}_* = \U{O}(\ZZ[X_*])$.
\item Take $Hom_L(\U{O}_*,\ZZ)$, the levelwise Hom into $\ZZ$, and compute cohomology of the resulting chain complex.
\end{enumerate}

Therefore, if we can prove that
\[
    \U{H}_i(\U{O}_*) = \left\{ \begin{array}{ll}
                                    \UM     & \textrm{for $i = 2$}\\
                                    \U{0}   & \textrm{otherwise,}
                                \end{array} \right.
\]
then
\[
    \U{H}_2(\U{O}_*) \leftarrow Ker(d_2(\U{O}_*)) \leftarrow \U{O}_3 \leftarrow 0
\]
is a free resolution of abelian groups of $\UM(G/H)$ for each $G/H$-level. Then taking $Hom_{\ZZ}(-,\ZZ)$ levelwise, and we compute both $\U{Ext}^3_{\UZ}(\UM,\UZ)$ and $Ext_L(\UM,Z)$.
\end{proof}

So we only need the following lemma.

\begin{lem}\label{lem-orbitH}
Using the same notation of the proof, we have
\[
    \U{H}_i(\U{O}_*) = \left\{ \begin{array}{ll}
                                    \UM     & \textrm{for $i = 2$}\\
                                    \U{0}   & \textrm{otherwise.}
                                \end{array} \right.
\]
\end{lem}
\begin{proof}
By direct computation using Mazur's formula, we see that
\[
\U{O}_* \cong \U{P}_* \square_{\UZ} \UZ^*
\]
So $\U{H}_i(\U{O}_*) \cong \U{Tor}^{\UZ}_i(\UM,\UZ^*)$. To compute $\U{Tor}$, instead of using $\U{P}_*$, a projective resolution of $\UM$, we can use a projective resolution of $\UZ^*$. The minimal one is the following:
\[
\UZ^* \xleftarrow{\nabla} \U{\ZZ[G]} \xleftarrow{1 - \gamma} \U{\ZZ[G]} \xleftarrow{\Delta} \UZ
\]
Now since $\UM(G/e) \cong 0$
\[
    \UM \square_{\UZ} \U{\ZZ[G]} \cong \UM_{G/e} \cong \U{\UM(G/e)} \cong \U{0}
\]
We have our result.
\end{proof}

We close this section with some $\U{Tor}^{\UZ}$ computations.

\begin{exam}
As in the proof of Lemma \ref{lem-orbitH}, we see that if $\UM(G/e) \cong 0$, then
\[
    \U{Tor}^{\UZ}_i(\UM,\UZ^*) = \left\{ \begin{array}{ll}
                                    \UM     & \textrm{for $i = 2$}\\
                                    \U{0}   & \textrm{otherwise.}
                                \end{array} \right.
\]

Now let $G = C_p$, apply $\U{Tor}^{\UZ}_*(\UB_{1},-)$ to the short exact sequence
\[
 0 \rightarrow \UZ^* \rightarrow \UZ \rightarrow \UB_{1} \rightarrow 0
\]
Since $\UZ$ is projective, we have
\[
\U{Tor}^{\UZ}_i(\UB_{1},\UB_{1}) = \left\{ \begin{array}{ll}
                                                \UB_{1} & \textrm{for $i = 0,3$}\\
                                                \U{0}   & \textrm{otherwise.}
                                            \end{array} \right.
\]
By applying $\U{Tor}^{\UZ}_i(\UZ^*,-)$ to the same short exact sequence, we have
\[\U{Tor}^{\UZ}_i(\UZ^*,\UZ^*) = \left\{ \begin{array}{ll}
                                                \UZ^*   & \textrm{for $i = 0$}\\
                                                \UB_{1} & \textrm{for $i = 1$}\\
                                                \U{0}   & \textrm{otherwise.}
                                            \end{array} \right.
\]
The same argument works for $G = C_{p^n}$. In $C_{p^n}$ we have
\[
\U{Tor}^{\UZ}_i(\UZ^*,\UZ^*) = \left\{ \begin{array}{ll}
                                                \UZ^*   & \textrm{for $i = 0$}\\
                                                \UB_{1,1,..,1} & \textrm{for $i = 1$}\\
                                                \U{0}   & \textrm{otherwise.}
                                            \end{array} \right.
\]
Notice that $\UB_{1,1,..,1}$ is the cokernel of the map $\UZ^* \rightarrow \UZ$ which is an underlying isomorphism.
\end{exam}

\section{Equivariant Orthogonal Spectra}\label{sec-EM}
\subsection{Equivariant spectra and commutative ring spectra}
In this paper, we use equivariant orthogonal spectra with positive complete model structure to model equivariant stable homotopy theory, which is written in detail in \cite[Appendix A,B]{HHR}. We use this specific setting because under the positive complete model structure, the category of modules over a commutative ring spectrum can be given a model structure, which will be used in Corollary \ref{cor-equivalence}. We use $(\mathcal{S}p^G, \wedge , S^{-0})$ for the symmetric monoidal model category of orthogonal $G$-spectra ($G$-spectra for short), and $ho\mathcal{S}p^G$ for its homotopy category. We use $[-,-]^G$ for the abelian group of homotopy classes of maps and $Fun_G(-,-)$ for the equivariant function spectrum. Given an equivariant commutative ring spectrum $R$, we use $\wedge_R$ and $Fun_R(-,-)$ for the induced smash product and function spectrum in $R$-modules.

\begin{prop}[{\cite[Proposition~B.138]{HHR}}]\label{prop-module-model}
Let $R$ be a commutative ring spectrum, then the forgetful functor
\[
    Mod_R \rightarrow \mathcal{S}p^G
\]
creates a cofibrantly generated symmetric monoidal model structure on $Mod_R$, in which fibrations and weak equivalences are underlying fibrations and weak equivalences.
\end{prop}

\begin{defi}\label{def-RO(G)}
The \textbf{group of orthogonal $G$-representation} $RO(G)$ is the Grothendieck group of finite dimensional $G$-representations under direct sum.
\end{defi}

Given an equivariant orthogonal spectrum $X$, we use $\U{\pi}_{\star}(X)$ for its $RO(G)$-graded homotopy Mackey functor.

In non-equivariant stable homotopy theory, for each abelian group $A$, there is an Eilenberg-Mac Lane spectrum $HA$ with the property that
\[
    \pi_{i}(HA) = \left\{ \begin{array}{ll}
                            A & \textrm{for $i = 0$}\\
                            0 & \textrm{for $i \neq 0$}.
                            \end{array} \right.
\]
Furthermore, $HA$ is unique up to homotopy, and if $A$ is an associative or commutative ring, $HA$ is an associative or commutative ring spectrum on the nose.
In orthogonal $G$-spectra, we can indeed construct Eilenberg-Mac Lane spectra out of Mackey functors.
\begin{thm}[{\cite[Theorem~5.3]{Greenlees-May}}]\label{thm-EM}
For a Mackey functor $\UM$, there is an Eilenberg-Mac Lane spcetrum $H\UM$, unique up to isomorphism in $ho\mathcal{S}p^G$. For Mackey functors $\UM$ and $\UN$,
\[
    [H\UM,H\UN]^G \cong \U{Hom}(\UM,\UN)(G/G).
\]
\end{thm}

Furthermore, one can show that if $\UM$ is a (commutative) Green functor, $H\UM$ is a (commutative) monoid in $ho\mathcal{S}p^G$. However, what we need for computation is something stronger: We wish $H\UM$ not only is a monoid in $ho\mathcal{S}p^G$, but a monoid in $\mathcal{S}p^G$ before passing to homotopy category. It turns out that this is more subtle than the non-equivariant case. The essential reason is that if $R$ is a commutative ring spectrum, $\U{\pi}_0(R)$ is not only a Green functor, but a Tambara functor.

\begin{thm}[{\cite[Example~5.14, Thoerem~1.4]{Angeltveit-Bohmann}}]
If $R$ is a commutative ring spectrum, then $\U{\pi}_0(X)$ is a Tambara functor.
\end{thm}

In fact, the theorem is more general than this: Their condition is weaker than being a commutative ring spectrum, and they show that the $RO(G)$-graded homotopy Mackey functor $\U{\pi}_{\star}(R)$ is an $RO(G)$-graded Tambara functor. However, we won't need these facts in this paper.

\begin{thm}[{\cite[Theorem~5.1]{Ullman:Tambara}}]
There is a functor
\[
        EM: Tamb_G \rightarrow Comm_G
\]
that taking value in cofibrant and fibrant Eilenberg-Mac Lane spectra such that the composition $\U{\pi}_0\circ EM$ is naturally isomorphic to the identity.
\end{thm}

Since $\UZ$ is a Tambara functor by Proposition \ref{prop-fixTambara}, we have

\begin{cor}\label{cor-ZTambara}
$H\UZ$ is a commutative ring spectrum and $Mod_{H\UZ}$ is a cofibrantly generated model category.
\end{cor}

\begin{rmk}
The point here is to show $H\UZ$ is a commutative monoid in $G$-spectra, and the category of $H\UZ$-modules is a cofibrantly generated model category. The same argument applies to any Tambara functors. However, for a Green functor $\UR$, if we are willing to alter the smash product in $G$-spectra, we can also make $H\UR$ into a commutative monoid. If we use the smash product of the spectral Mackey functors in \cite{Guillou-May3} or the smash product corresponding to the trivial $N_{\infty}$-operad in \cite{BlumHill:Smash}, then $H\UR$ is a commutative monoid there.
\end{rmk}

\subsection{Fix point and orbit spectra}Fixed points and orbits constructions are the main bridges connecting equivariant objects to non-equivariant objects. In orthogonal $G$-spectra, we will make use of several different fixed points, namely (derived) \emph{fixed point spectrum}, \emph{homotopy fixed point spectrum} and \emph{geometric fixed point spectrum}.

By only considering trivial representations in $\MS{J}_G$, from an orthogonal $G$-spectrum $X$ we can obtain a non-equivariant spectrum with $G$-action
$i^*_0X$.
\begin{defi}[{\cite[Section~2.5]{HHR}}]\label{def-fixpoint}
The \textbf{fixed point spectrum} of $X$, $X^G$ is the non-equivariant orthogonal spectrum obtained by taking levelwise fixed point space of $i^*_0X$.
\end{defi}

This functor is not homotopic. However, if $X$ is fibrant, we have an isomorphism
\[
    \pi_*(X^G) \cong \pi_*^G X,
\]

therefore we will always consider the derived fixed point functor. By the above isomorphism, the derived fixed point spectrum reflects parts of the information homotopy Mackey functor carries. Since $\U{\pi}_*(-)$ is lax monoidal, and in general $(\UM \square \UN)(G/G)$ is not isomorphic to $\UM(G/G) \otimes \UN(G/G)$, we see that the derived fixed point does not commute with smash product. The derived fixed point also does not commute with the suspension functor. This can be seen by taking the example $\Sigma^{\infty} G_+$.

The next fixed point functor is \emph{the homotopy fixed point}.
\begin{defi}\label{def-htpyfix}
The \textbf{homotopy fixed point spectrum} of a $G$-spectrum $X$ is
\[
        X^{hG} := Fun_G(EG_+,X)^G,
\]
where $EG$ is a contractible $G$-space with free $G$-action.
\end{defi}

Sometimes the $G$-spectrum $X^h := Fun_G(EG_+,X)$ is also useful, we will use \emph{the homotopy fixed point $G$-spectrum} referring to the function spectrum as a $G$-spectrum before taking fixed points, and use \emph{the homotopy fixed point spectrum} referring to the non-equivariant spectrum $X^{hG}$.

Along with homotopy fixed points, we have \emph{the homotopy orbit functor}.

\begin{defi}\label{def-horbit}
The \textbf{homotopy orbit spectrum} of a $G$-spectrum $X$ is
\[
    X_{hG} := (EG_+ \wedge X)^G
\]
\end{defi}

We will call $X_h := EG_+ \wedge X$ \emph{the homotopy orbit $G$-spectrum} and $X_{hG}$ \emph{the homotopy orbit spectrum}.

There are canonical maps $X^G \rightarrow X^{hG}$ and $X_{hG} \rightarrow X^G$, induced by the collapsing map $EG_+ \rightarrow S^0$. These maps will be useful in our computation.

Since $EG_+$ is built out of only free $G$-cells, smashing with $EG_+$ or taking maps from it will forget a lot of information in the world of $G$-spectra.

\begin{prop}[{\cite[Proposition~1.1]{GreeMay-Tate}}]\label{prop-hfphoequi}
If $f: X \rightarrow Y$ is a map of $G$-spectra that is a weak equivalence of underlying non-equivariant spectra, then the induced map
\[
    Fun(1,f): X^h \rightarrow Y^h
\]
and
\[
    1 \wedge f: X_h \rightarrow Y_h
\]
are weak equivalences of $G$-spectra. Thus $X^{hG} \rightarrow Y^{hG}$ and $X_{hG} \rightarrow Y_{hG}$ are weak equivalences.
\end{prop}

An advantage of the homotopy fixed points and homotopy orbits is that they are very computable via \emph{homotopy fixed point spectral sequences} and \emph{homotopy orbit spectral sequences}. They are spectral sequences arise from the cellular structure of $EG_+$.
\begin{thm}\label{thm-hfpss}
There are spectral sequences with
\[
E_2^{s,t} = H^t(G,\pi_s(X)) \Rightarrow \pi_{s-t}(X^{hG})
\]
and
\[
E^2_{s,t} = H_t(G,\pi_s(X)) \Rightarrow \pi_{s+t}(X_{hG})
\]
\end{thm}

\begin{rmk}
If we consider all subgroups $H \subset G$, then group (co)homology has the structure of $\UZ$-modules (see Example \ref{exam-gpcoh}), therefore these spectral sequences are spectral sequences of $\UZ$-modules (Though the extensions might not respect the $\UZ$-module structure). We can also consider homotopy fixed points or homotopy orbits of $S^V \wedge X$, to obtain $RO(G)$-graded spectral sequences. In this way, we can think about homotopy fixed point and homotopy orbit spectral sequences are $RO(G)$-graded spectral sequences computing $\U{\pi}_{\star}(X^h)$ and $\U{\pi}_{\star}(X_h)$. This is the version we use in our computation. Details of $RO(G)$-graded homotopy fixed point spectral sequences appear in \cite[Section~2.3]{Hill-Meier}.
\end{rmk}

The last fixed point functor we introduce is \emph{the geometric fixed point}. Consider the space $E\mathcal{P}$, which is characterized by the property
\[
    (E\mathcal{P})^H \simeq \left\{ \begin{array}{ll}
                                    \emptyset & H = G\\
                                    pt        & H \neq G,
                                \end{array} \right.
\]
and the cofibre sequence $E\mathcal{P}_+ \rightarrow S^0 \rightarrow \widetilde{E\mathcal{P}}$.

\begin{defi}\label{def-geofp}
The \textbf{geometric fixed point spectrum} of a $G$-spectrum $X$ is
\[
    \Phi^G(X) = (\widetilde{E\mathcal{P}} \wedge X)^G.
\]
\end{defi}

Similarly, we use \emph{the geometric fixed point $G$-spectrum} for $\Phi(X) = \widetilde{E\mathcal{P}} \wedge X$ and \emph{the geometric fixed point spectrum} for $\Phi^G(X)$.

\begin{rmk}\label{rmk-geofp}
When $G = C_p$, we have $E\mathcal{P} = EG$, therefore we have a cofibre sequence in $G$-spectra
\[
    X_h \rightarrow X \rightarrow \Phi(X)
\]
and a cofibre sequence in spectra
\[
    X_{hG} \rightarrow X^G \rightarrow \Phi^G(X).
\]
However, this is not true for other groups.
\end{rmk}

The geometric fixed point has the best formal properties among all fixed point functors.
\begin{prop}[{\cite[Proposition~2.43]{HHR}}]\label{prop-geofp}
The geometric fixed point functor $\Phi^{G}$ has the following properties:
\begin{enumerate}
\item $\Phi^G$ preserves weak equivalences.
\item $\Phi^G$ commutes with filtered homotpy colimits.
\item Given a $G$-space $A$ and an actual $G$ representation $V$, there is a weak equivalence
\[
    \Phi^G(S^{-V} \wedge A) \approx S^{-V^G} \wedge A,
\]
where $V^G$ is the $G$-invariant subspace of $V$.
\item For $G$-spectra $X$ and $Y$, there is a natural chain of weak equivalences connecting
\[
    \Phi^G(X \wedge Y) \text{ and } \Phi^{G}(X) \wedge \Phi^{G}(Y).
\]
\end{enumerate}
\end{prop}

Our main computation tool is the Tate diagram, a diagram that relates the homotopy orbit, the fixed point and the homotopy fixed point of a $G$-spectrum. It is constructed in \cite{GreeMay-Tate}, and is the main topic of the memoir.

Consider the cofibre sequence
\[
EG_+ \rightarrow S^0 \rightarrow \widetilde{EG}.
\]
Smashing it with the cannonical map $X \rightarrow X^h$ we obtain a commutative diagram
\[
\xymatrix{
X_h     \ar[r] \ar[d] & X \ar[r] \ar[d]& {\widetilde{EG}} \wedge X \ar[d]\\
(X^h)_h \ar[r]                 & X^h \ar[r]     & {\widetilde{EG}} \wedge X^h
}
\]

Since the left vertical map induces isomorphism on the underlying homotopy groups, by Proposition \ref{prop-hfphoequi}, we have
\begin{prop}
The left vertical map
\[
    X_h \rightarrow (X^h)_h
\]
is a weak equivalence in $G$-spectra.
\end{prop}

We will use $\widetilde{X}$ for $\widetilde{EG} \wedge X$ and $X^t$ for $\widetilde{EG} \wedge X^h$. The latter is called the Tate spectrum of $X$.

\begin{defi}\label{def-Tate}
The \textbf{Tate diagram} of a $G$-spectrum $X$ is the commutative diagram of cofibrations
\[
\xymatrix{
X_h     \ar[r] \ar[d]^{\simeq} & X \ar[r] \ar[d]& {\widetilde{X}} \ar[d]\\
X_h     \ar[r]                 & X^h \ar[r]     & X^t
}.
\]
\end{defi}

\subsection{Equivariant Anderson duality}
Non-equivariantly, there is a universal coefficient exact sequence between integral homology and cohomology
\[
    0 \rightarrow Ext^1(H_{*-1}(X;\ZZ),\ZZ) \rightarrow H^*(X;\ZZ) \rightarrow Hom(H_*(X;\ZZ),\ZZ) \rightarrow 0.
\]
One way of generalize this exact sequence is the Anderson duality \cite{AndDual}. Consider the short exact sequence
\[
    0 \rightarrow \ZZ \rightarrow \QQ \rightarrow \QQ/\ZZ \rightarrow 0,
\]
since both $\QQ$ and $\QQ/\ZZ$ are injective abelian groups, by Brown Representability Theorem, $Hom(\pi_*(-),\QQ)$ and $Hom(\pi_*(-), \QQ/\ZZ)$ represents cohomology theories $I_{\QQ}$ and $I_{\QQ/\ZZ}$. Let $I_{\ZZ}$ be the fibre of $I_{\QQ} \rightarrow I_{\QQ/\ZZ}$ and $I_{\ZZ}(X) = Fun(X,I_{\ZZ})$. Then for any spectrum $E$, viewed as a (co)homology theory, we have a universal coefficient exact sequence of abelian groups
\[
    0 \rightarrow Ext^1(E_{*-1}(X),\ZZ) \rightarrow I_{\ZZ}(E)^*(X) \rightarrow Hom(E_{*}(X),\ZZ) \rightarrow 0
\]

Equivariantly, we consider $Hom(\pi_*^G(-),\QQ)$ and $Hom(\pi_*^G(-),\QQ/\ZZ)$. By an equivariant version of Brown Representability Theorem (e.g. \cite[Corollary~XIII.3.3]{May:Alaska}), we see that as in the non-equivariant case, they are represented by $G$-spectra $I^G_{\QQ}$ and $I^G_{\QQ/\ZZ}$. Let $I^G_{\ZZ}$ be the homotopy fibre of $I^G_{\QQ} \rightarrow I^G_{\QQ/\ZZ}$.
\begin{defi}\label{def-Anderson}
The \textbf{equivariant Anderson dual} of a $G$-spectrum $X$ is
\[
    I^G_{\ZZ}(X) = Fun_G(X,I_{\ZZ}^G).
\]
\end{defi}

Equivariant Anderson duality for $G = C_2$ is studied in detail in \cite[Section~3.2]{RicDual}, which includes the $C_2$-version of all propositions here. Since the proof is pretty much identical for any finite group, we would not reprove them here.

By the same argument as the non-equivariant case, we can obtain a universal coefficient exact sequence from Anderson duality. Furthermore, since the short exact sequence is natural both in the cohomology theory $E$ and the $G$-spectra $X$, it respects the Mackey functor structure. By smashing with representation spheres, we can also use $RO(G)$-grading instead of integer grading.
\begin{prop}\label{prop-Anderson}
Given a $G$-spectra $E$ and $X$, we have an $RO(G)$-graded short exact sequence
\[
    0 \rightarrow Ext_L(E_{\star-1}(X),\ZZ) \rightarrow I^G_{\ZZ}(E)^{\star}(X) \rightarrow Hom_L(E_{\star}(X),\ZZ) \rightarrow 0
\]
$Ext_L$ and $Hom_L$ are defined in Definition \ref{def-Hom-Ext-L}.
\end{prop}

\begin{prop}\label{prop-AndersonRing}
Let $R$ be an equivariant homotopy commutative ring spectrum and $M$ an $R$-module in the homotopy category, then $I_{\QQ/\ZZ}^G(X)$ and $I_{\ZZ}^G(M)$ are naturally an $R$-modules.
\end{prop}

\begin{prop}\label{prop-AndersonHom}
Let $R$ be an equivariant commutative ring spectrum and $M,N$ be $R$-modules in $\mathcal{S}p^G$, then
\[
Fun_R(M,I_{\ZZ}(N)) \simeq I_{\ZZ}(M \wedge_R N)
\]
\end{prop}

\begin{proof}
Since $M \wedge_R -$ is the left adjoint of $Fun_R(M,-)$ and the forgetful functor $i^*: \mathcal{S}p^G \rightarrow \mathcal{S}p$ is the left adjoint of $Fun(R,-)$, we have
\begin{align*}
    Fun_R(M, I_{\ZZ}(N)) & \cong Fun_R(M, Fun(N,I_{\ZZ}))\\
                         & \cong Fun_R(M, Fun_R(N,Fun(R,I_{\ZZ})))\\
                         & \cong Fun_R(M \wedge_R N, Fun(R,I_{\ZZ}))\\
                         & \cong Fun(M \wedge_R N, I_{\ZZ})\\
                         & \cong I_{\ZZ}(M \wedge_R N)
\end{align*}
\end{proof}

\subsection{Universal coefficient and K\"unneth spectral sequences}
Another way of generalizing the universal coefficient exact sequence of integer (co)homology is the universal coefficient spectral sequences. The non-equivariant version appears in \cite[Lecture~1]{Ad:99}. The equivariant analog is the main topic of \cite{LewisMandell}, which uses homological algebra of Mackey functors (see Definition \ref{def-ext-tor}) in an essential way. Let $E$ be an equivariant commutative ring spectrum and $X,Y$ be $G$-spectra.
\begin{thm}[Lewis-Mandell]\label{thm-UCSS}
The \textbf{equivariant K\"unneth spectral sequence} is the strongly convergent spectral sequences of Mackey functors
\[
\U{E}^2_{s,t} = \U{Tor}^{\U{E}_*}_{s,t}(\U{E}_*(X),\U{E}_*(Y)) \Rightarrow \U{E}_{s+t}(X \wedge Y).
\]
The \textbf{equivariant universal coefficient spectral sequence} is the conditionally convergent spectral sequence of Mackey functors
\[
\U{E}_2^{s,t} = \U{Ext}^{s,t}_{\U{E}_*}(\U{E}_*(X),\U{E}_*) \Rightarrow \U{E}^{t-s}(X).
\]
\end{thm}
\begin{rmk}
The index $t$ and $*$ can be understood as either integer grading or $RO(G)$-grading, and different choice of indexing groups will give very different spectral sequences. In this paper we will only use the integer grading version of these spectral sequences.
\end{rmk}

\subsection{Representations and representation spheres}\label{sec-rep}

Before we do any computation, we need to understand the index group $RO(G)$ for $G = C_{p^n}$ and the corresponding representation spheres. They are analyzed in detail in \cite{HHR:RO(G)}, and we follow their approach.

Given a primitive $p^n$-th root of unity $\mu_{p^n}$, it determines a group homomorphism $\mu_{p^n}: C_{p^n} \rightarrow S^1$. Let $\lambda(k)$ be the representation given by composition of $\mu_{p^n}$ with a degree $k$ map $k:S^1 \rightarrow S^1$. This is a representation of $C_{p^n}$ on $\RR^2$.

The \textbf{regular representation} of $G$ is $\rho_G = \RR[G]$, which has a decomposition
\[
    \rho_G = 1 \oplus \bigoplus_{i = 1}^{\frac{p^n - 1}{2}} \lambda(i)
\]
for $p > 2$ and
\[
    \rho_{C_{2^n}} = 1 \oplus \sigma \oplus \bigoplus_{i = 1}^{2^{n-1}-1} \lambda(i),
\]
where $\sigma$ is the sign representation of $C_{2^n}$.

We can build cellular structures on representation spheres. For $S^{\lambda(rp^k)}$ where $p \nmid r$, we consider $p^{n-k}$ rays (1-cells) passing the origin of $\RR^2$ that divide $\RR^2$ into $p^{n-k}$ parts equivalently (2-cells). Thus a cellular structure of $S^{\lambda(rp^k)}$ is the following:
\[
    S^0 \cup C_{p^n}/C_{p^k+} \wedge e^1 \cup C_{p^n}/C_{p^k+} \wedge e^2
\]

We can then obtain a cellular structure on any $S^V$ by smashing various $S^{\pm \lambda(rp^k)}$ together. However, this is a cellular structure that is too big to compute with. When $V$ is an actual representation, we can simplify the cellular structure by two steps.

The first step is to identify all $\lambda(rp^k)$ for $p \nmid r$. If we localize at $p$, then all $S^{\lambda(rp^k)}$ are homotopy equivalent to each other, since the degree $r$ map is invertible now. If we don't localize, then by \cite{Kawakubo}, different $S^{\lambda(rp^k)}$ are not even stably equivalent. However, we have the following.
\begin{prop}[{\cite[Lemma~1]{HuKrizCoef}}]\label{prop-choice}
$S^{\lambda(rp^k)} \wedge H\UZ \simeq S^{\lambda(p^k)} \wedge H\UZ$ for all $p \nmid r$.
\end{prop}

\begin{proof}
Using the cellular structures above, we see that the $\UZ$-coefficient cellular chain for $S^{\lambda(rp^k)}$, $\U{C}_*(S^{\lambda(rp^k)})$ is the following
\[
\xymatrix
@R=0mm{
    0 & 1 & 2\\
    {\UZ} & \ar[l] {\U{\ZZ[C_{p^n}/C_{p^k}]}} & {\U{\ZZ[C_{p^n}/C_{p^k}]}.} \ar[l]_{1-\gamma^r}
}
\]
The cellular chain for $S^{-\lambda(p^k)}$, $\U{C}_*(S^{-\lambda(p^k)})$ is the dual chain for $\U{C}_*(S^{\lambda(p^k)})$
\[
\xymatrix
@R=0mm{
    -2 & -1 & 0\\
    {\U{\ZZ[C_{p^n}/C_{p^k}]}} & {\U{\ZZ[C_{p^n}/C_{p^k}]}} \ar[l]_{1-\gamma} & {\UZ} \ar[l]
}
\]

Then, $\U{\pi}_*(S^{\lambda(rp^k) - \lambda(p^k)} \wedge H\UZ)$ can be computed by the total homology of the double complex
\[
\U{C}_*(S^{rp^k}) \square_{\UZ} \U{C}_*(S^{-p^k}).
\]
By direct computation, $\U{\pi}_*(S^{\lambda(rp^k) - \lambda(p^k)} \wedge H\UZ)$ is $\UZ$ concentrated in degree $0$, therefore by uniqueness of Eilenberg-Mac Lane spectra we know that
\[
    S^{\lambda(rp^k) - \lambda(p^k)} \wedge H\UZ \simeq H\UZ,
\]
thus
\[
    S^{\lambda(rp^k)} \wedge H\UZ \simeq S^{\lambda(p^k)} \wedge H\UZ.
\]
\end{proof}

Either way, we will not distinguish $S^{\lambda(rp^k)}$ for different $r$ that $p \nmid r$, and use $S^{\lambda_k}$ for them. By equating all $S^{\lambda(rp^k)}$, we obtain a quotient group $JO(G)$ from $RO(G)$. For $p$ odd, $JO(G)$ is freely generated by $\lambda_k$ for $0 \leq k \leq n-1$ and the trivial representation. When $p = 2$, $JO(C_{2^n})$ is freely generated by $\lambda_k$ for $0 \leq k \leq n-2$, the sign representation $\sigma$ and the trivial representation, and $\lambda_{n-1} = 2\sigma$. We will still use the word ``$RO(G)$-grading", but it will actually mean $JO(G)$-grading. We will use $\lambda$ for $\lambda_0$.

Now if $V$ is an actual representation of $C_{p^n}$, we can assume that
\[
V = \Sigma_{i = 0}^{n-1} a_i\lambda_i + a_n.
\]

The second step comes from the simple fact that if $H \subset K \subset G$, there is no $G$-map $G/K \rightarrow G/H$. So lower skeletons of a $G$-CW-complex always have larger stabilizer groups. If we apply $i^*_{C_{p^k}}$ for $0 < k < n$ then only the first $k$ $\lambda_i$ are nontrivial, so we have
\begin{align*}
    S^V = & S^{a_n} \cup C_{p^n}/C_{p^{n-1}+} \wedge e^{a_n+1} \cup_{1-\gamma} C_{p^n}/C_{p^{n-1}+} \wedge e^{a_n+2} \cup ...\\
          & \cup C_{p^n}/C_{p^{n-1}+} \wedge e^{a_n + 2a_{n-1}} \cup C_{p^n}/C_{p^{n-2}+} \wedge e^{a_n + 2a_{n-1} + 1} \cup ...\\
          & \cup C_{p^n+} \wedge e^{\Sigma a_i}.
\end{align*}

\begin{defi}\label{def-orientable}
Let $V \in RO(G)$, we say $V$ is \textbf{orientable} if
\[
V = V_1 - V_2
\]
where $V_i$ are actual representations and the maps $V_i: G \rightarrow O(|V_{i}|)$ factor through $SO(|V_{i}|)$.
\end{defi}

If $G = C_{p^n}$ and $p$ is odd, every virtual representation is orientable. If $p = 2$, then orientable representations form an index $2$ subgroup of $RO(G)$, with quotient generated by $\sigma$, the sign representation of $C_{2^n}$ on $\RR$.

We end this section with a simple but useful lemma.

\begin{lem}\label{lem-miracle}
For $G = C_{p^n}$,
\[
    I_{\ZZ}H\UZ \simeq H\UZ^* \simeq \Sigma^{2-\lambda} H\UZ.
\]
\end{lem}

\begin{proof}
The first equivalence is straightforward from the short exact sequence of Anderson duality.

The cellular chain complex of $S^{-\lambda}$ in $\UZ$ coefficient is
\[
    \U{\ZZ[C_{p^n}]} \xleftarrow{1 - \gamma} \U{\ZZ[C_{p^n}]} \xleftarrow{\Delta} \ZZ,
\]
whose homology is $\UZ^*$ concentrated in degree $-2$. The result then comes from uniqueness of Eilenberg-Mac Lane spectra.
\end{proof}

\section{$H\UZ$ and its modules}\label{sec-HZ1}

In this section we start to compute around the Eilenberg-Mac Lane spectrum $H\UZ$ for $G = C_{p^n}$. The main goals of this section are the following:
\begin{enumerate}
\item A topological proof of Theorem \ref{thm-ext}, which we restate here.
\begin{customthm}{\ref{thm-ext}}
For $G = C_{p^n}$, if $\UM(G/e) \cong 0$, then
\[
\U{Ext}^i_{\UZ}(\UM,\UZ) = \left\{ \begin{array}{ll}
                                        \UM^E & \textrm{for $i = 3$}\\
                                        \U{0} & \textrm{otherwise}
                                    \end{array} \right.
\]
This isomorphism is natural in $\UM$.
\end{customthm}

\item A proof of the following theorem.
\begin{customthm}{\ref{thm-formz}}
For $G = C_{p^n}$, if $\UM$ is a form of $\UZ$ (see Definition \ref{def-form-Z}), then
\[
    H\UM \simeq \Sigma^{V} H\UZ,
\]
for some $V \in RO(G)$.
\end{customthm}
\end{enumerate}

\subsection{A topological proof of Theorem \ref{thm-ext}}
The strategy here is to convert the $\U{Ext}$ computation into a topological setting, and then make use of the equivariant Anderson duality. The following theorem of Schwede and Shipley is crucial.
\begin{thm}[{\cite[Theorem~5.1.1]{Schwede-Shipley3}}]\label{thm-modeleq}
Let $\MS{C}$ be a simplicial, cofibrantly generated, stable model category and $A$ a ringoid. Then the following conditions are equivalent:
\begin{enumerate}
\item There is a chain of Quillen equivalences between $\MS{C}$ and the model category of chain complexes of $\MS{A}$-modules.
\item The homotopy category of $\MS{C}$ is triangulated equivalent to $\MS{D}(\MS{A})$, the unbounded derived category of the ringoid $\MS{A}$.
\item $\MS{C}$ has a set of compact generators and the full subcategory of compact objects in $ho(\MS{C})$ is triangulated equivalent to $K^b(proj-\MS{A})$, the homotopy category of bounded chain complexes of finitely generated projective $\MS{A}$-modules.
\item The model category $\MS{C}$ has a set of tiltors whose endomorphism ringoid in $ho(\MS{C})$ is isomorphic to $\MS{A}$.
\end{enumerate}
\end{thm}

A ringoid $\MS{A}$ is a category enriched over $(Ab,\otimes,\ZZ)$, such as $\MS{B}_G$ and $\MS{B} \UZ_G$. The category of modules over a ringroid $\MS{A}$ is the category of contravariant additive enriched functors from $\MS{A}$ to $Ab$, for example $Mack_G$ is the category of modules over $\MS{B}_G$ and $Mod_{\UZ}$ is the category of modules over $\MS{B} \UZ_G$. The model structure of $\MS{A}$-modules here is the projective model structure, thus it computes the correct derived functor. A set of tiltors in a stable model category $\MS{C}$ is a set of compact generators $\mathbb{T}$ such that for any $T,T' \in \mathbb{T}$, $Ho(\MS{C})(T,T')_*$ is concentrated in $* = 0$.

\begin{cor}\label{cor-equivalence}
There is a chain of Quillen equivalences between $Mod_{H\UZ}$ and the category of chain complexes of $\UZ$-modules. Therefore, derived functors of $Hom_{\UZ}$ and $\square_{\UZ}$ can be computed in $Mod_{H\UZ}$. More explicitly we have
\[
    \U{Ext}^i_{\UZ}(\UM,\UN) \cong \U{\pi}_{-i}(Fun_{H\UZ}(H\UM,H\UN))
\]
and
\[
    \U{Tor}^{\UZ}_i(\UM,\UN) \cong \U{\pi}_i(H\UM \wedge_{H\UZ} H\UN).
\]
\end{cor}

\begin{proof}
In $Mod_{H\UZ}$, a set of tiltors $\mathbb{T}$ can be chosen as the set
\[
\{X_+ \wedge H\UZ|X \text{ is a finite $G$-set}\}.
\]
Then we have
\[
ho(Mod_{H\UZ})(X_+ \wedge H\UZ, Y_+ \wedge H\UZ) \cong [X_+, Y_+ \wedge H\UZ]^G \cong Hom_G(\ZZ[X],\ZZ[Y]).
\]
Therefore the endomorphism ringoid of $\mathbb{T}$ is isomorphic to $\MS{B} \UZ_G$ in Definition \ref{def-BZG}. By Proposition \ref{prop-FunctorZmod}, the category of modules over $\MS{B} \UZ_G$ is $Mod_{\UZ}$.
\end{proof}

\begin{rmk}
The exactly same proof works for any fixed point Mackey functors of commutative rings where $G$-acts through ring isomorphisms (e.g. $\U{\FF_p}$), since these Mackey functors are automatically Tambara functors and have similar ringoid descriptions, see \cite[Example~1.3.1, 1.4.5]{MazurMackey}.
\end{rmk}

We need a simple lemma for the proof of Theorem \ref{thm-ext}.

\begin{lem}\label{lem-torsionlambda}
If $\U{M}(G/e) \cong 0$, then $\Sigma^{\lambda} H\UM \simeq H\UM$.
\end{lem}

\begin{proof}
Consider $\U{C}_*(S^{\lambda})$, the cellular chain of $S^{\lambda}$, which is
\[
    \UZ \xleftarrow{\nabla} \U{\ZZ [G]} \xleftarrow{1-\gamma} \U{\ZZ [G]}
\]
Now, $\U{\pi}_*(\Sigma^{\lambda}H\UM) = \U{H}_*(C_*(S^{\lambda}) \square_{\UZ}\UM)$. However,
\[
    (\U{\ZZ [G]} \square_{\UZ}\UM)(X) \cong \UM(X \times G) \cong 0
\]
Since $X \times G$ is a free $G$-set and $\UM$ evaluating on free $G$-set is $0$ since $\UM(G/e) \cong 0$. Therefore $\U{\pi}_* \Sigma^{\lambda}H\UM$ is concentrated in degree $0$, and $\U{\pi}_0 = \UZ \square_{\UZ} \UM = \UM$.
\end{proof}

\begin{proof}[Topological proof of {Theorem \ref{thm-ext}}]
By Corollary \ref{cor-equivalence},
\[
    \U{Ext}^i_{\UZ}(\UM,\UZ) = \U{\pi}_{-i}(Fun_{H\UZ}(H\UM,H\UZ))
\]
Now by Lemma \ref{lem-miracle}, $H\UZ \cong \Sigma^{\lambda-2} I^G_{\ZZ}(H\UZ)$, and by Proposition \ref{prop-AndersonHom} we have
\begin{align*}
    Fun_{H\UZ}(H\UM,H\UZ) & \simeq Fun_{H\UZ}(H\UM,\Sigma^{\lambda-2} I^G_{\ZZ}(H\UZ)) \\
                          & \simeq \Sigma^{\lambda-2} I^G_{\ZZ}(H\UM \wedge_{H\UZ} H\UZ) \\
                          & \simeq \Sigma^{\lambda-2} I^G_{\ZZ}(H\UM) \\
                          & \simeq \Sigma^{\lambda-2} \Sigma^{-1} H\UM^E \\
                          & \simeq \Sigma^{-3} H\UM^E
\end{align*}
\end{proof}

\subsection{Some elements of $\U{\pi}_{\star}(H\UZ)$}

Before we can do more computation, we need to introduce some special elements of $\U{\pi}_{\star}(H\UZ)$, namely the $a$ and $u$ families. It turns out that every element in $\U{\pi}_{\star}(H\UZ)$ is a fraction or an image under connecting homomorphisms of these families.

\begin{prop}[{\cite[Proposition~3.3]{HHR:C4}}]\label{prop-tophom}
If $V$ is an actual orientable representation for $G = C_{p^n}$ of dimension $n$, then
\[
    \U{H}_{n}(S^{V};\UZ) = \U{\pi}_{n - V}(S^{V} \wedge H\UZ) \cong \UZ.
\]
\end{prop}

\begin{defi}\label{def-au}
\begin{enumerate}
\item For an actual representation $V$ with $V^G = 0$, let $a_{V} \in \U{\pi}_{-V}(S^0)$ be the map $S^0 \rightarrow S^V$ which embeds $S^0$ to $0$ and $\infty$ in $S^V$. We will also use $a_V$ for its Hurewicz image in $\U{\pi}_{-V}(H\UZ)$.\\
\item For an actual orientable representation $W$ of dimension $n$, let $u_W$ be the generator of $\U{H}_{n}(S^W;\UZ)(G/G)$ which restricts to the choice of orientation in
    \begin{displaymath}
    \U{H}_{n}(S^W;\UZ)(G/e) \cong H_{n}(S^{n};\ZZ).
    \end{displaymath}
    In homotopy grading, $u_W \in \U{\pi}_{n - W}(H\UZ)(G/G)$.
\end{enumerate}
\end{defi}

\begin{prop}[{\cite[Lemma~3.6]{HHR:C4}}]\label{prop-au}
Elements $a_{V} \in \U{\pi}_{-V}(H\UZ)(G/G)$ and $u_{W} \in \U{\pi}_{|W| - W}(H\UZ)(G/G)$ satisfy the following:
\begin{enumerate}
\item $a_{V_1 + V_2} = a_{V_1}a_{V_2}$ and $u_{W_1 + W_2} = u_{W_1}u_{W_2}$.
\item $Res^G_H(a_V) = a_{i^*_H(V)}$ and $Res^G_H(u_V) = u_{i^*_H(V)}$
\item $|G/G_V| a_V = 0$, where $G_V$ is the isotropy subgroup of $V$.
\item \textbf{The gold relation.} For $V,W$ oriented representations of degree $2$, with $G_V \subset G_W$,
        \begin{displaymath}
            a_Wu_V = |G_W/G_V| a_Vu_W
        \end{displaymath}
        In terms of oriented irreducible representations of $C_{p^n}$, the gold relation reads
        \begin{displaymath}
            \textrm{For $0 \leq i < j < n$, } a_{\lambda_j}u_{\lambda_i} = p^{i-j}a_{\lambda_i}u_{\lambda_j}
        \end{displaymath}
\item The subring consists of $\U{\pi}_{i-V}(H\UZ)(G/G)$ where $V$ is an actual representation is
        \begin{displaymath}
            \ZZ[a_{\lambda_i},u_{\lambda_i}]/(p^{n-i}a_{\lambda_i} = 0, \textrm{gold relations}) \textrm{    for $0 \leq i,j < n$}
        \end{displaymath}
\end{enumerate}
We will call this subring $BB_G$, standing for "basic block". We will use $\U{BB}_G$ for the graded Green functor in the corresponding $RO(G)$-degree of $BB_G$. We will omit $G$ if there is no ambiguity.

\end{prop}

As discussed in Section \ref{sec-rep}, every representation sphere $S^V$ has a $G$-CW-complex structure. If $V \in RO(G)$ is in the image of the map $RO(G/K) \rightarrow RO(G)$ for a normal subgroup $K \subset G$, then in the $\UZ$-coefficient cellular chain complex of $S^V$, all $\UZ$-modules involved are projective $\UZ$-modules in the image of the pullback $\Psi^*_K$. By Corollary \ref{cor-pullhom}, we have
\begin{cor}\label{cor-pullback}
Let $V \in RO(G)$ be a representation in the image of $RO(G/K)$, then
\[
    \U{H}_*(S^V;\UZ) \cong \Psi^*_K(\U{H}_*(S^V;\UZ)).
\]
The right hand side is computed $G/K$-equivariantly.

Furthermore, this isomorphism sends elements $a_V$ and $u_V$ in $G/K$-homology to $a_V$ and $u_V$ in $G$-homology.
\end{cor}

If $G = C_{p^n}$ and $K$ is the subgroup of order $p$, then $a_{\lambda_i}$ and $u_{\lambda_i}$ in $G/K$ pullback to $a_{\lambda_{i+1}}$ and $u_{\lambda_{i+1}}$ respectively in $G$.

\subsection{Forms of $\UZ$} We prove the following theorem.
\begin{thm}\label{thm-formz}
For $G = C_{p^n}$, if $\UM$ is a form of $\UZ$ (see Definition \ref{def-form-Z}), then
\[
    H\UM \simeq \Sigma^{V} H\UZ,
\]
for some $V \in RO(G)$.
\end{thm}

\begin{proof}
By Lemma \ref{lem-miracle}, the theorem is true for $C_p$, since the only forms of $\UZ$ are $\UZ$ and $\UZ^*$. Now assume that the theorem is true for $G/K$, where $K$ is the unique subgroup of order $p$. If $Res^p_1:\UM(G/K) \rightarrow \UM(G/e)$ is an isomorphism, then by Proposition \ref{prop-pullP} $\UM \cong \Psi^*_K(\U{M}')$, where $\U{M}'$ is a form of $\UZ$ for $G/K$. By induction hypothesis,
\[
H\U{M}' \simeq S^{V_{G/K}} \wedge H\UZ
\]
for some $V_{G/K} \in RO(G/K)$. Using the quotient map $G \rightarrow G/K$ and Corollary \ref{cor-pullback}, we see
\[
H\UM \simeq S^{V} \wedge H\UZ,
\]
where $V$ is the pullback of $V_{G/K}$ into $RO(G)$.

If $Res^p_1:\UM(G/K) \rightarrow \UM(G/e)$ is multiplication by $p$, then the corresponding restriction for $\UM^*$ (see Definition \ref{def-Hom-Ext-L}) is an isomorphism. By the above argument we have
\[
    H\UM^* \simeq S^V \wedge H\UZ
\]
for some $V \in RO(G)$. Then by Proposition \ref{prop-Anderson} and Lemma \ref{lem-miracle}
\[
    H\UM \simeq I_{\ZZ}(H\UM^*) \simeq S^{2-\lambda_0 - V} \wedge H\UZ.
\]
\end{proof}

\begin{rmk}
We can construct $V$ from $\UM$ by going backwards: If $Res^p_1$ is an isomorphism, then we do nothing and go to the quotient group $G/N$. If $Res^p_1$ is multiplication by $p$, then we record $2-\lambda_0$ then go to the quotient group and repeat. In the end, $V$ will be an alternative sum of $\lambda_i$, with identification $2 = \lambda_n$.
\end{rmk}

\begin{exam}
Let $G = C_8$ and $\UM = \UZ_{1,0,1}$ (see Definition \ref{def-form-Z-index}) is the following form of $\UZ$
\[
\xymatrix{
\ZZ \ar@/_/[d]_2 \\
\ZZ \ar@/_/[d]_1 \ar@/_/[u]_1\\
\ZZ \ar@/_/[d]_2 \ar@/_/[u]_2\\
\ZZ \ar@/_/[u]_1
}
\]
Now since the bottom restriction is $2$, we know
\[
H\UM \simeq S^{2 - \lambda_0 - V_1} \wedge H\UZ,
\]
where $V_1$ is a representation from $C_8/C_2$ and in $C_8/C_2$, since $\UM^* \cong \Psi^*_{C_2}(\UZ_{1,0})$, we have
\[
H\UZ_{1,0} \simeq S^{V_1} \wedge H\UZ.
\]

In $C_8/C_2$, by the same argument, we see $V_1 = 2 - \lambda_1 - V_2$ (here $\lambda_1$ is the $\lambda_1$ on $C_8$, which factors through $C_8/C_2$) for some $V_2$ from $C_8/C_4$. Finally in $C_8/C_4$, we see that $V_2 = 2 - \lambda_2$. Therefore we have
\[
    H\UM \simeq \Sigma^{-\lambda_0 + \lambda_1 - \lambda_2 + 2} H\UZ.
\]
\end{exam}

\section{Computation of $\U{\pi}_{\star}(H\UZ)$ for $C_{p^2}$}\label{sec-HZ2}
In this section, we first gives a complete computation of $\U{\pi}_{\star}(H\UZ)$ for $G = C_{p^2}$ using the Tate diagram and induction on quotient groups. Then we analyze both equivariant Anderson duality and Spanier-Whitehead duality of $H\UZ$-modules and show how they interact with each other. Finally, using these computation, we can give more computations in homological algebra of $\UZ$-modules, which are more difficult using purely algebraic methods.
\subsection{The main computation}
Now we start to compute $\U{\pi}_{\star}(H\UZ)$ for $G = C_{p^2}$. We will first compute the case when $\star$ is orientable, which includes all $RO(G)$ for $p$ odd, then use the cofibre sequence
\[
    C_{2^n}/C_{2^{n-1}+} \rightarrow S^0 \rightarrow S^{\sigma}
\]
to compute the case for $p = 2$. The result for $C_{p^2}$ where $p$ is odd is Theorem \ref{thm-Cp2oriented} and for $p = 2$ is discussed around Example \ref{exam-C4}.

The Tate diagram for $H\UZ$ is the following:
\[
\xymatrix{
H\UZ_h     \ar[r] \ar[d]^{\simeq} & H\UZ \ar[r] \ar[d]& {\widetilde{H\UZ}} \ar[d]\\
H\UZ_h     \ar[r]                 & H\UZ^h \ar[r]     & H\UZ^t
}.
\]

We start with the second row, which consists of the homotopy fixed point, homotopy orbit and Tate spectrum of $H\UZ$. By smashing with representation spheres, the corresponding spectral sequences can be made into $RO(G)$-grading. By using $\UZ$-module structure in group homology and cohomology (see Example \ref{exam-gpcoh}), these spectral sequences can be made into spectral sequences of $\UZ$-modules. If $p$ is odd, then all representation spheres are orientable, therefore we have the homotopy fixed point spectral sequence
\[
        \U{E}^2_{V,t} = \U{H}^{t}(G;\pi_V(H\UZ)) \Rightarrow \U{\pi}_{t-V}(H\UZ^h)
\]

By degree reason, this spectral sequence collapse at $E^2$-page. We will use the same name in Proposition \ref{prop-au} to name elements in $\U{\pi}_{\star}(H\UZ^h)$. That is, we have $a_{\lambda_0} \in \U{\pi}_{-\lambda_0}(H\UZ^h)$ and $u_{\lambda_0} \in \U{\pi}_{2 - \lambda_0}(H\UZ^h)$. If $p = 2$ we also have $a_{\sigma} \in \U{\pi}_{-\sigma}(H\UZ^h)$.
\begin{prop}\label{prop-gpcoh}
For $G = C_{p^n}$ and $\star$ orientable, as modules over $\U{BB}_{C_{p^n}}$ (see Proposition \ref{prop-au}) we have
\[
    \U{\pi}_{\star}(H\UZ^h)(G/G) = \ZZ[a_{\lambda_0},u_{\lambda_i}^{\pm}]/(p^n a_{\lambda_0}) \text{ for $0 \leq i < n$}
\]
The Mackey functor structure is determined by the following: any fraction of $u_V$ generates a $\UZ$, and monomials containing positive powers of $a_{\lambda_0}$ generates $\UB_{1,1,...,1}$.

By the gold relation in Proposition \ref{prop-au}, we have $a_{\lambda_i} = 2^i a_{\lambda_0}\frac{u_{\lambda_i}}{u_{\lambda_0}}$.

\[
    \U{\pi}_{\star}(H\UZ^t)(G/G) = \ZZ/p^n[a_{\lambda_0}^{\pm},u_{\lambda_i}^{\pm}] \text{ for $0 \leq i < n$.}
\]
All Mackey functors are $\UB_{1,1,...,1}$.

\[
    \U{\pi}_{\star}(H\UZ_h)(G/G) = p^n\ZZ[u_{\lambda_i}^{\pm}] \oplus \ZZ/p^n \langle \Sigma^{-1} \frac{P}{a_{\lambda_0}^j} \rangle,
\]
where $0 \leq i < n$, $j > 0$ and $P$ is any fraction of monomials of $u_{\lambda_i}$s. All torsion free generators generate $\UZ^*$ and all torsions are $\UB_{1,1,...,1}$. The $\Sigma^{-1}$ means that these elements are coming from connecting homomorphism from $H\UZ^t$.

\end{prop}

\begin{proof}
$\U{\pi}_{\star}(H\UZ^h)$ is computed directly from group cohomology. By Lemma \ref{lem-inverta} below, we obtain
\[
    \U{\pi}_{\star}(H\UZ^t) \cong a_{\lambda_0}^{-1} \U{\pi}_{\star}(H\UZ^h).
\]
Finally, one can compute $\U{\pi}_{\star}(H\UZ_h)$ by taking kernel and cokernel (as $\UZ$-modules) of the $a_{\lambda_0}$-localization map
\[
    \U{\pi}_{\star}(H\UZ^h) \rightarrow \U{\pi}_{\star}(H\UZ^t).
\]
\end{proof}

\begin{lem}\label{lem-inverta}
For $G = C_{p^n}$, we have
\[
    \U{\pi}_{\star}(\widetilde{X}) \cong a_{\lambda_0}^{-1} \U{\pi}_{\star}(X).
\]
\end{lem}

\begin{proof}
Since the unit sphere of $\infty \lambda_0$ is contractible with free $G$-action, its cone $S^{\infty \lambda_0}$ models $\widetilde{EG}$. We can write $S^{\infty \lambda_0}$ as the colimit of
\[
    S^0 \xrightarrow{a_{\lambda_0}} S^{\lambda_0} \xrightarrow{a_{\lambda_0}} S^{2\lambda_0} \rightarrow ...
\]
Therefore smashing with $\tilde{E}G$ is the same as inverting $a_{\lambda_0}$.
\end{proof}

In general, we will compute $\U{\pi}_{\star}(H\UZ)$ by induction on quotient groups. Let $G = C_{p^n}$ and $K = C_{p}$, we have a four-steps induction:
\begin{enumerate}
\item Assume that we know $\U{\pi}_{\star}(H\UZ)$ for $G/K$, then by Corollary \ref{cor-pullback} we know $\U{\pi}_V(H\UZ)$ for all $V$ in the image of $RO(G/K) \rightarrow RO(G)$. That is, all $V$ that contains no copies of $\lambda_0$.
\item We can understand the map $\U{\pi}_V(H\UZ_h) \rightarrow \U{\pi}_V(H\UZ)$ by comparing names of elements, therefore compute $\U{\pi}_V(\widetilde{H\UZ})$ for all $V$ containing no $\lambda_0$.
\item Since $\U{\pi}_{\star}(\widetilde{H\UZ})$ is $a_{\lambda_0}$-periodic, we then understand all $\U{\pi}_{\star}(\widetilde{H\UZ})$.
\item Use the cofibre sequence again to compute $\U{\pi}_{\star}(H\UZ)$ from $\U{\pi}_{\star}(H\UZ_h)$ and $\U{\pi}_{\star}(\widetilde{H\UZ)}$.
\end{enumerate}

We start the computation with $G = C_p$, to illustrate the computational method and to keep track of elements. $C_p$ computation is well known and is written in \cite[Section~2.C]{Greenlees-Four}. In diagrams and spectral sequences, it is awkward to use the notation in Definition \ref{def-form-Z-index}. Instead we will use notations in Table \ref{tab-Cp}. In general, a box shape means a form of $\UZ$ and other shapes mean torsion $\UZ$-modules.

\begin{table}
\includegraphics{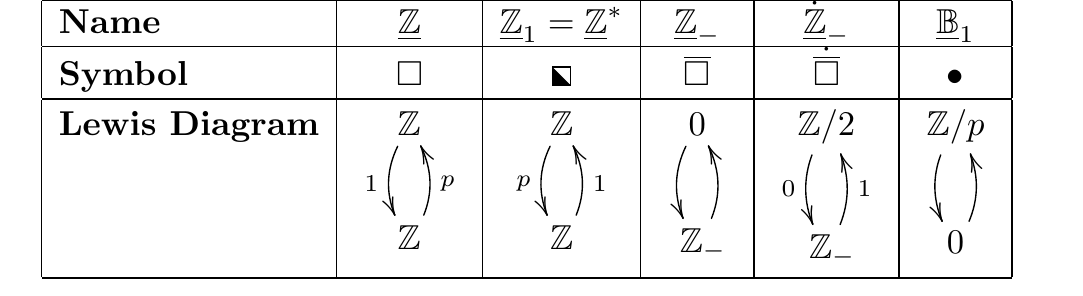}
\caption{Table of $C_p$ Mackey functors}\label{tab-Cp}
\end{table}

Notice that $\overline{\square}$ and $\dot{\overline{\square}}$ are only defined for $p = 2$.\\

Usually, we will write our index $\star = * - V$ for $V \in RO(G)$ with $V^G = 0$. The reason is
\[
    \U{\pi}_{* - V}(H\UZ) \cong \U{H}_*(S^V;\UZ).
\]
In this way, it is easier to compare and verify the result for some special $V$, especially those that $S^V$ has a simple cellular structure.

For $G = C_p$, by Proposition \ref{prop-gpcoh} we have
\[
    \U{\pi}_{\star}(H\UZ^h)(G/G) = \ZZ[a_{\lambda},u_{\lambda}^{\pm}]/(pa_{\lambda}),
\]
\[
    \U{\pi}_{\star}(H\UZ^t)(G/G) = \ZZ/p[a_{\lambda}^{\pm},u_{\lambda}^{\pm}],
\]
and
\[
    \U{\pi}_{\star}(H\UZ_h)(G/G) = p\ZZ[u_{\lambda}^{\pm}] \oplus \ZZ/p\langle \Sigma^{-1} \frac{u_{\lambda}^i}{a_{\lambda}^j} \rangle,
\]
where $i \in \ZZ$ and $j > 0$.

Specifically, we see that
\[
    \U{\pi}_*(H\UZ_h)(G/G) = p\ZZ \oplus \ZZ/p \langle \Sigma^{-1} \frac{u_{\lambda}^i}{a_{\lambda}^i} \rangle \text{ for $i > 0$}.
\]
Since $\U{\pi}_*(H\UZ)$ is $\UZ$ concentrated in $* = 0$, by the long exact sequence, we see that
\[
    \U{\pi}_*(\widetilde{H\UZ})(G/G) = \ZZ/p[\frac{u_{\lambda}}{a_{\lambda}}].
\]
Now $\widetilde{H\UZ}$ is $a_{\lambda}$-periodic, therefore
\[
    \U{\pi}_{\star}(\widetilde{H\UZ})(G/G) = \ZZ/p[a_{\lambda}^{\pm},u_{\lambda}]
\]
and all $\UZ$-modules are $\bullet$.

Now, consider the connecting homomorphism
\[
    \U{\pi}_{\star}(\widetilde{H\UZ}) \rightarrow \U{\pi}_{\star-1}(H\UZ_h).
\]

If $\star = i - m\lambda$ for $m \geq 0$, then
\[
\U{\pi}_{* - m\lambda}(\widetilde{H\UZ}) = \ZZ/p \langle \frac{a_{\lambda}^m u_{\lambda}^i}{a_{\lambda}^i} \rangle \textrm{For $i \geq 0$.}
\]
and
\[
\U{\pi}_{* - m\lambda}(H\UZ_h) = \ZZ \langle pu_{\lambda}^m \rangle \oplus \ZZ/p \langle \Sigma^{-1}\frac{u_{\lambda}^m u_{\lambda}^i}{a_{\lambda}^i} \rangle \textrm{For $i > 0$.}
\]
Therefore, for $0 \leq i < m$, elements $\frac{a_{\lambda}^m u_{\lambda}^i}{a_{\lambda}^i} = a_{\lambda}^{m-i} u_{\lambda}^i$ maps to $0$ under connecting homomorphism, and thus gives elements in $\U{\pi}_{2i - m\lambda}(H\UZ)$. For $i = m$, we have a nontrivial extension
\[
    0 \rightarrow \ZZ \langle pu_{\lambda}^m \rangle \rightarrow \ZZ \langle u_{\lambda}^m \rangle \rightarrow \ZZ/p\langle u_{\lambda}^m \rangle \rightarrow 0
\]
obtained from comparison with the bottom row of the Tate diagram. In terms of $\UZ$-modules, it is
\[
    \U{0} \rightarrow \twobox \rightarrow \square \rightarrow \bullet \rightarrow \U{0}.
\]

The following picture shows the case $m = 2$, with the bottom row $\U{\pi}_{*-2\lambda}(\widetilde{H\UZ})$ and top row $\U{\pi}_{*-2\lambda}(H\UZ_h)$. Arrows indicate connecting homomorphism, and green vertical line means an extension involving an exotic restriction.

\begin{center}
\includegraphics{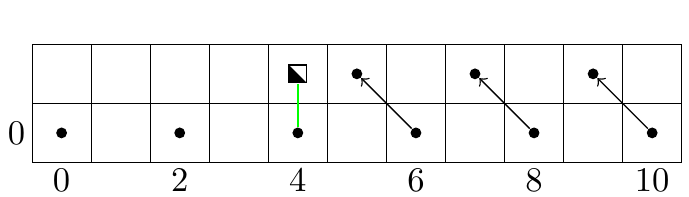}
\end{center}

By considering all $m \geq 0$, this gives exactly the part $\U{BB}_{C_p}$ in Proposition \ref{prop-au}.

If $\star = i - m\lambda$ for $m < 0$, since the source of connecting homomorphism only exits in degrees where $i \geq 0$, the element $u_{\lambda}^m$ would not receive a nontrivial extension, and elements $\Sigma^{-1} \frac{u_{\lambda}^m u_{\lambda}^i}{a_{\lambda}^i}$ for $0 < i < -m$ will not be killed by the connecting homomorphism. For $m = -2$, the following picture shows the connecting homomorphism.

\begin{center}
\includegraphics{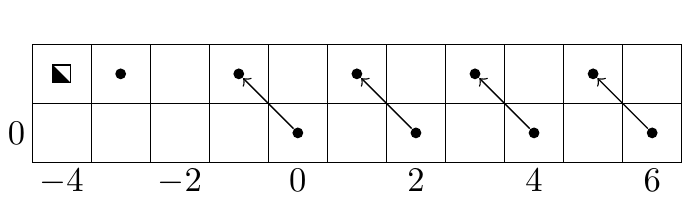}
\end{center}

By summarizing the computation, we have
\begin{prop}\label{prop-Cporientable}
If $G = C_p$ and $\star \in RO(G)$ is orientable, then
\begin{align*}
    \U{\pi}_{\star}(H\UZ)(G/G) = & \ZZ[u_{\lambda},a_{\lambda}]/(pa_{\lambda}) \oplus p\ZZ[u_{\lambda}^{-i}] & \textrm{for $i > 0$}\\
                                 & \oplus \ZZ/p \langle \Sigma^{-1}u_{\lambda}^{-j}a_{\lambda}^{-k} \rangle & \textrm{for $j,k > 0$.}
\end{align*}
As $\UZ$-modules, each monomial contains powers of $a_{\lambda}$ generates a $\bullet$, $u_{\lambda}^i$ generates $\square$ and $pu_{\lambda}^{-i}$ generates $\twobox$.
\end{prop}

Figure \ref{fig-Cp} shows the result intuitively, with horizontal coordinate the trivial representation and vertical coordinate $\lambda$. Vertical lines mean $a_{\lambda}$-multiplications.
\begin{figure}
\includegraphics{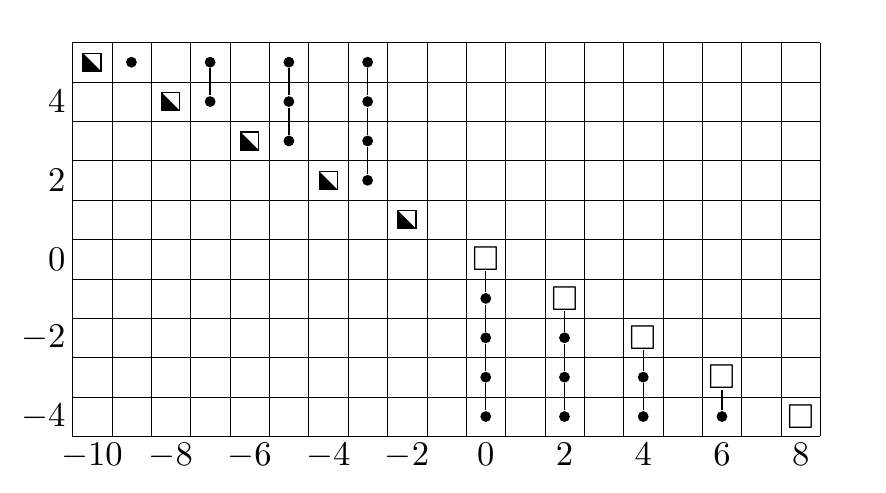}
\caption{$\U{\pi}_{\star}(H\UZ)$ for $G = C_p$ and $p$ odd}\label{fig-Cp}
\end{figure}

For $p = 2$ the above proposition covers half of $RO(C_2)$. We can identify $\lambda = 2\sigma$ and smash $S^V$ with the cofibre sequence
\[
    C_{2+} \rightarrow S^0 \rightarrow S^{\sigma}
\]
to compute $\U{\pi}_{\star}(H\UZ)$. By definition, $\U{\pi}_{\star}(C_{2+}\wedge X) \cong \U{\pi}_{\star}(X)_{C_{2+}}$, and $(\UB_1)_{C_{2+}} \cong \U{0}$ and $\UZ_{C_{2+}} \cong \UZ^*_{C_{2+}} \cong \U{\ZZ[C_2]}$. Finally, we have exact sequences
\[
   \U{0} \rightarrow \UZ_- \rightarrow \U{\ZZ[C_2]} \rightarrow \UZ \rightarrow \UB_1 \rightarrow 0.
\]
and
\[
   \U{0} \rightarrow \UZ_- \rightarrow \U{\ZZ[C_2]} \rightarrow \UZ^* \rightarrow \U{0}
\]
Thus, if we compute $\U{H}_{*}(S^{V+ \sigma};\UZ)$ from $\U{H}_{*}(S^V;\UZ)$, all $\UB_1$ will remain, and each $\UZ$ in $S^V$ will be replaced by a new $\UB_1$, with a $\UZ_-$ in degree $1$ higher. For $\UZ^*$, a new $\UZ_-$ appear in degree $1$ higher. However, if $V \neq -2\sigma$, then in degree $|V| + 1$, there is also a $\UB_1$, and the extension problem here can be solved by the following lemma.
\begin{lem}[{\cite[Lemma~4.2]{HHR:C4}}]\label{lem-TAR}
Let $G = C_{2^n}$ with sign representation $\sigma$. Let $K \subset G$  be the index $2$-subgroup and $X$ a $G$-spectrum. Then we have an exact sequence of abelian groups
\[
    \U{\pi}_{\star}(X)(G/K) \xrightarrow{Tr^G_K} \U{\pi}_{\star}(X)(G/G) \xrightarrow{a_{\sigma}} \U{\pi}_{\star-\sigma}(X)(G/G) \xrightarrow{Res^G_K} \U{\pi}_{\star-\sigma}(X)(G/K).
\]
\end{lem}
By degree reason, the torsion in $\U{H}_{-n}(S^{-n\sigma};\UZ)$ for $n > 2$ odd is annihilated by $a_{\sigma}$, therefore is in the image of transfer. Thus we know that
\[
    \U{H}_{-n}(S^{-n\sigma};\UZ) \cong \dot{\UZ}_- \text{ for $n >2$ and odd}.
\]
Thus we have
\begin{prop}\label{prop-C2}
For $G = C_2$,
\begin{align*}
    \U{\pi}_{\star}(H\UZ)(G/G) = & \ZZ[u_{2\sigma},a_{\sigma}]/(2a_{\sigma}) \oplus 2\ZZ[u_{2\sigma}^{-i}] & \textrm{for $i > 0$}\\
                            & \oplus \ZZ/2 \langle \Sigma^{-1}u_{2\sigma}^{-j}a_{\sigma}^{-k} \rangle & \textrm{for $j , k> 0$}
\end{align*}
As $\UZ$-modules, each monomial with powers of $a_{\sigma}$ generates $\UB_1$, except the power is $-1$, then it generates $\dot{\UZ}_-$. Each power of $u_{2\sigma}$ generates $\UZ$ and in $\U{\pi}_{n - n\sigma}(H\UZ)$ for $n > 0$ odd, there is a $\UZ_-$, which has trivial $G/G$-level.
\end{prop}

Figure \ref{fig-C2} shows the result for $C_2$, with horizontal coordinate the trivial representations and vertical coordinate $\sigma$. Vertical lines are $a_{\sigma}$-multiplications.

\begin{figure}
\includegraphics{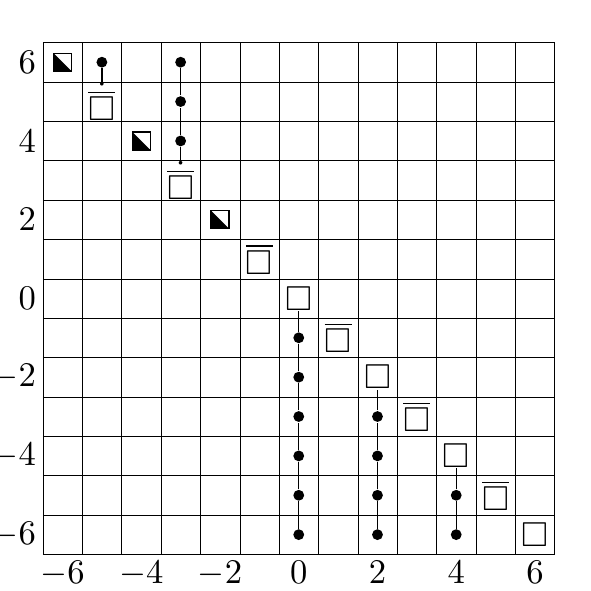}
\caption{$\U{\pi}_{\star}(H\UZ)$ for $G = C_2$}\label{fig-C2}
\end{figure}

\begin{rmk}
For $G = C_p$, the Tate diagram may not be the cleanest way of doing such a computation. Modulo trivial representations, all $V \in RO(G)$ is either an actual representation or the opposite of one, thus $S^V$ has very simple cellular structure. However, the Tate diagram computation can keep track of the multiplicative structure, and is the one that can be easily generalized.
\end{rmk}

Now we compute $\U{\pi}_{\star}(H\UZ)$ for $G = C_{p^2}$. First we need symbols for all $\UZ$-modules that appear which are in Table \ref{tab-Cp2}. We use the same symbol for a $C_{p}$-$\UZ$-module and its pullback in $C_{p^2}$.

\begin{table}
\includegraphics{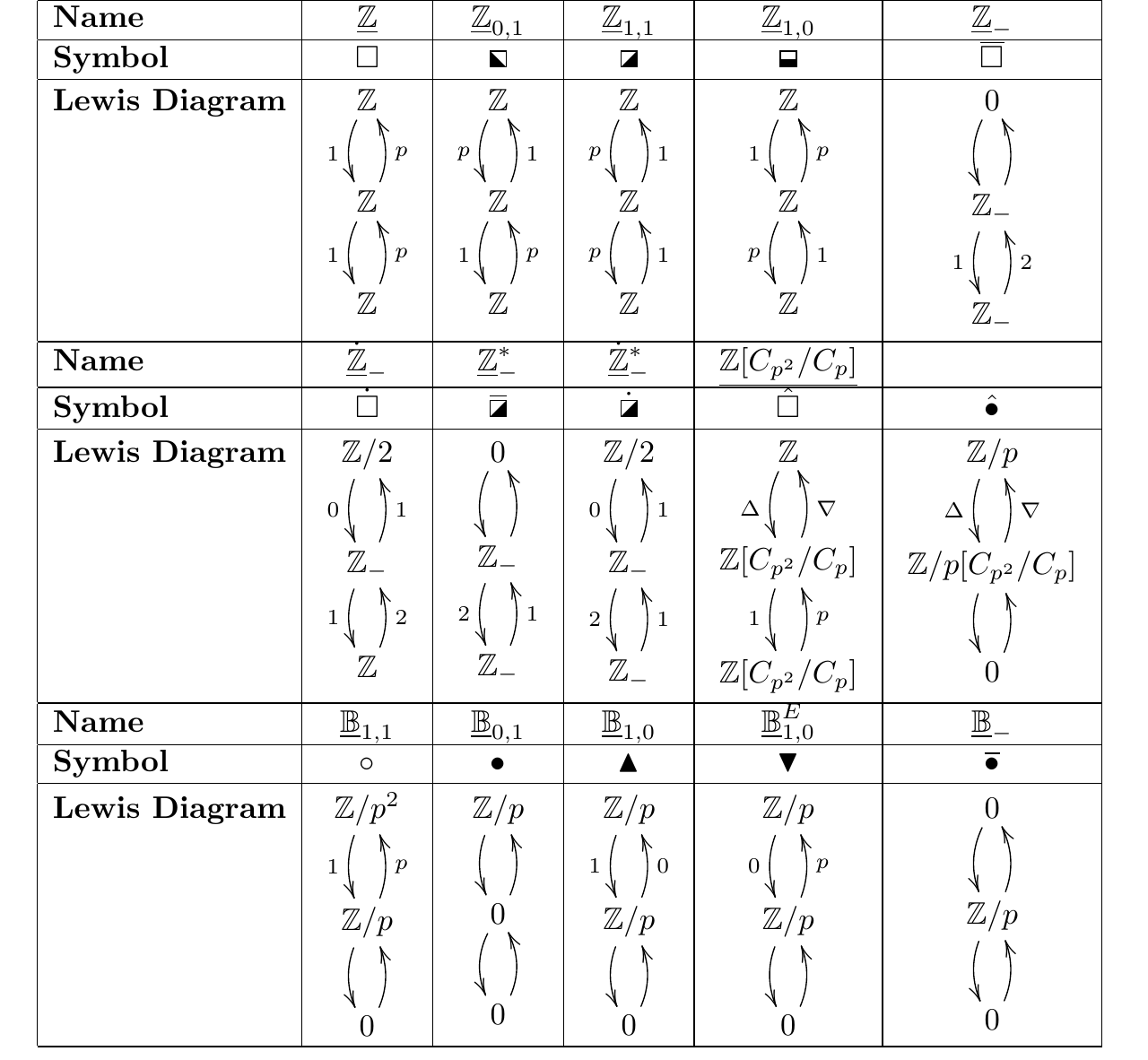}
\caption{Table of $C_{p^2}$ Mackey functors}\label{tab-Cp2}
\end{table}

First, we want to describe $\U{\pi}_{\star}(\widetilde{H\UZ})$. We know $\U{\pi}_V(H\UZ)$ for all $V$ coming from $RO(C_{p^2}/C_p)$, which is exactly Proposition \ref{prop-Cporientable} with replacing $\lambda$ in $C_p$ by its image $\lambda_1$ in $C_{p^2}$. By Proposition \ref{prop-gpcoh}, we have
\[
    \U{\pi}_{\star}(H\UZ_h)(G/G) = p^2\ZZ[u_{\lambda_0}^{\pm},u_{\lambda_1}^{\pm}] \oplus \ZZ/p^2 \langle \Sigma^{-1}u_{
    \lambda_0}^m u_{\lambda_1}^n\frac{u_{\lambda_0}^i}{a_{\lambda_0}^i}\rangle \ \textrm{for $m,n \in \ZZ$ and $i > 0$.}
\]
In particular, in degree $\star = * - n\lambda_1$, we have
\[
    \U{\pi}_{* - n\lambda_1}(H\UZ_h) = p^2\ZZ \langle u_{\lambda_1}^n \rangle \oplus \ZZ/p^2 \langle \Sigma^{-1} u_{\lambda_1}^n \frac{u_{\lambda_0}^i}{a_{\lambda_0}^i} \rangle \textrm{ for $i > 0$.}
\]
On the other hand, if $n \geq 0$,
\[
    \U{\pi}_{* - n\lambda_1}(H\UZ) = \ZZ \langle u_{\lambda_1}^n \rangle \oplus \ZZ/p \langle a_{\lambda_1}^iu_{\lambda_1}^{n-i}\rangle \ \textrm{for $0 < i \leq n$}
\]
The map $H\UZ_h \rightarrow H\UZ$ induces $\UZ^* \rightarrow \UZ$ on forms of $\UZ$, and trivial otherwise by degree reason. So we know that for $n \geq 0$
\begin{align*}
    \U{\pi}_{*-n\lambda_1}(\widetilde{H\UZ})(G/G) &  =  \ZZ/p \langle a_{\lambda_1}^iu_{\lambda_1}^{n-i}\rangle &\textrm{for $0 < i \leq n$}\\
    &\oplus \ZZ/p^2 \langle u_{\lambda_1}^n\frac{u_{\lambda_0}^j}{a_{\lambda_0}^j}\rangle &\textrm{for $j \geq 0$.}
\end{align*}
Here all $p$-torsion generates $\UB_{0,1}$ and $p^2$-torsion generates $\UB_{1,1}$.

If $n < 0$, $\U{\pi}_{* - n\lambda_1}(H\UZ_h)$ has the same description as above. However for $\U{\pi}_{* - n\lambda_1}(H\UZ)$ we have
\[
    \U{\pi}_{*-n\lambda_1}(H\UZ)(G/G) = p\ZZ \langle u_{\lambda_1}^n \rangle \oplus \ZZ/p \langle \Sigma^{-1} u_{\lambda_1}^n \frac{u_{\lambda_1}^i}{a_{\lambda_1}^i} \rangle \ \textrm{for $0 < i < |n|$}
\]
with $pu_{\lambda_1}^n$ generates $\UZ_{0,1} = \Psi^*_{C_p}(\UZ_1)$ and all torsions are $\UB_{0,1} = \Psi^*_{C_p}(\UB_1)$. The map $H\UZ_h \rightarrow H\UZ$ induces an isomorphism on the $G/e$-level of forms of $\UZ$, since it is an underlying equivalence. Thus on $\UZ$-modules we have a short exact sequence
\[
    \U{0} \rightarrow \fourbox \rightarrow \twobox \rightarrow \JJ \rightarrow \U{0}.
\]
On torsion classes, by the gold relation in Proposition \ref{prop-au}, we have $a_{\lambda_1}u_{\lambda_0} = pa_{\lambda_0}u_{\lambda_1}$, therefore
\[
    \frac{u_{\lambda_0}^i}{a_{\lambda_0}^i} = p^i \frac{u_{\lambda_1}^i}{a_{\lambda_1}^i} = 0
\]
since the latter is a $p$-torsion and $i > 0$. Therefore, the map is trivial on all torsion classes. For $n < 0$ we have
\begin{align*}
    \U{\pi}_{*-n\lambda_1}(\widetilde{H\UZ})(G/G) =\ZZ/p\langle pu_{\lambda_1}^n \rangle &\oplus \ZZ/p \langle \Sigma^{-1}u_{\lambda_1}^{n}\frac{u_{\lambda_1}^i}{a_{\lambda_1}^i} \rangle &\textrm{for $0 < i < |n|$}\\
    &\oplus \ZZ/p^2 \langle u_{\lambda_1}^{n}\frac{u_{\lambda_0}^j}{a_{\lambda_0}^j}\rangle &\textrm{for $j > 0$}
\end{align*}
In terms of $\UZ$-modules, the class $pu_{\lambda_1}^n$ generates a $\UB_{1,0}$, all other $p$-torsions generate $\UB_{0,1}$ and $p^2$-torsions generate $\UB_{1,1}$.

Summarizing the computation above, we have
\begin{prop}\label{prop-Cp2-a-local}
For $G = C_{p^2}$, the Green functor structure of $\U{\pi}_{\star}(\widetilde{H\UZ})$ can be described as follows:
If $n \geq 0$,
\begin{align*}
    \U{\pi}_{*-n\lambda_1}(\widetilde{H\UZ})(G/G) &  =  \ZZ/p \langle a_{\lambda_1}^iu_{\lambda_1}^{n-i}\rangle &\textrm{for $0 < i \leq n$}\\
    &\oplus \ZZ/p^2 \langle u_{\lambda_1}^n\frac{u_{\lambda_0}^j}{a_{\lambda_0}^j}\rangle &\textrm{for $j \geq 0$.}
\end{align*}
All $p$-torsion generates $\UB_{0,1}$ and $p^2$-torsion generates $\UB_{1,1}$.

If $n < 0$,
\begin{align*}
    \U{\pi}_{*-n\lambda_1}(\widetilde{H\UZ})(G/G) =\ZZ/p\langle pu_{\lambda_1}^n \rangle &\oplus \ZZ/p \langle \Sigma^{-1}u_{\lambda_1}^{n}\frac{u_{\lambda_1}^i}{a_{\lambda_1}^i} \rangle &\textrm{for $0 < i < |n|$}\\
    &\oplus \ZZ/p^2 \langle u_{\lambda_1}^{n}\frac{u_{\lambda_0}^j}{a_{\lambda_0}^j}\rangle &\textrm{for $j > 0$}
\end{align*}
The class $pu_{\lambda_1}^n$ generates a $\UB_{1,0}$, all other $p$-torsions generate $\UB_{0,1}$ and $p^2$-torsions generate $\UB_{1,1}$.

Finally, every element is $a_{\lambda_0}$-periodic (see Lemma \ref{lem-inverta}), that is,
\[
    \U{\pi}_{* - n\lambda_1 - m\lambda_0}(\widetilde{H\UZ}) = a_{\lambda_0}^m \U{\pi}_{* - n\lambda_1}(\widetilde{H\UZ})
\]

Multiplication is determined by name of elements and fraction form of the gold relation, and product of two elements with $\Sigma^{-1}$ is $0$.
\end{prop}

Figure \ref{fig-tilde} shows $\U{\pi}_{\star}(\widetilde{H\UZ})$, with horizontal coordinate the trivial representations and vertical coordinate $\lambda_1$. Vertical lines are $a_{\lambda_1}$-multiplication, while dash lines mean sending generators to $p$-times generators and firm lines mean surjections. Since everything is $a_{\lambda_0}$-periodic, we omit $\lambda_0$-coordinate.
\begin{figure}[H]
\includegraphics{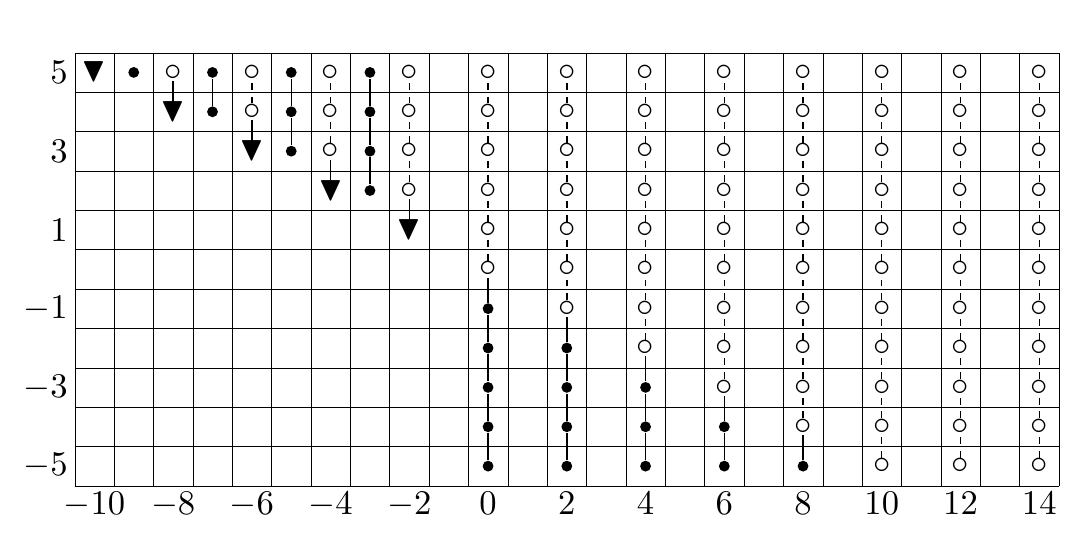}
\caption{$\U{\pi}_{\star}(\widetilde{H\UZ})$ for $G = C_{p^2}$ and $p$ odd}\label{fig-tilde}
\end{figure}

Now we understand $\U{\pi}_{\star}(H\UZ_h)$, $\U{\pi}_{\star}(\widetilde{H\UZ})$ and the connecting homomorphism between them, we can compute $\U{\pi}_{\star}(H\UZ)$. What essentially happen in this computation is that $H\UZ_h$ is $u_{\lambda_0}$-periodic while $\widetilde{H\UZ}$ is $a_{\lambda_0}$-periodic and this difference in periodicity produces a lot of classes in $H\UZ$.

First, by Corollary \ref{cor-pullback}, if $\star = * - n\lambda_1$, then $\U{\pi}_{\star}(H\UZ)$ can be computed from Proposition \ref{prop-Cporientable} by applying the pullback functor $\Psi^*_{C_p}$.

Then we start with $\U{\pi}_{* - n\lambda_1 - m\lambda_0}(H\UZ)$ for $n , m \geq 0$. By Proposition \ref{prop-gpcoh}, we have
\[
    \U{\pi}_{*-n\lambda_1-m\lambda_0}(H\UZ_h)(G/G) = p^2\ZZ \langle u_{\lambda_0}^m u_{\lambda_1}^n \rangle \oplus \ZZ/p^2 \langle \Sigma^{-1} u_{\lambda_0}^m u_{\lambda_1}^n \frac{u_{\lambda_0}^i}{a_{\lambda_0}^i} \rangle \ \textrm{for $i > 0$}.
\]
By Proposition \ref{prop-Cp2-a-local} we have
\begin{align*}
    \U{\pi}_{*-n\lambda_1-m\lambda_0}(\widetilde{H\UZ})(G/G) & = \ZZ/p \langle a_{\lambda_0}^m a_{\lambda_1}^iu_{\lambda_1}^{n-i} \rangle & \textrm{for $0 < i \leq n$}\\
    & \oplus \ZZ/p^2 \langle a_{\lambda_0}^m u_{\lambda_1}^n\frac{u_{\lambda_0}^j}{a_{\lambda_0}^j} \rangle & \textrm{for $0 \leq j$.}
\end{align*}
Simply by comparing names of elements, we see that elements $a_{\lambda_0}^m u_{\lambda_1}^n\frac{u_{\lambda_0}^j}{a_{\lambda_0}^j}$ for $j > m$ kill the corresponding elements with $\Sigma^{-1}$ under connecting homomorphism, while when $j = m$, it is $u_{\lambda_1}^n u_{\lambda_0}^m$, which is involved in an extension of the form
\[
 \U{0} \rightarrow \fourbox \rightarrow \square \rightarrow \circ \rightarrow \U{0},
\]
where the $\UZ_{1,1}$ is generated by $p^2u_{\lambda_0}^m u_{\lambda_1}^n$ in $H\UZ_h$. Therefore we see that for $m , n \geq 0$,
\begin{align*}
    \U{\pi}_{* - n\lambda_1 - m\lambda_0}(H\UZ) = \ZZ \langle u_{\lambda_0}^m u_{\lambda_1}^n \rangle &\oplus \ZZ/p \langle a_{\lambda_1}^ia_{\lambda_0}^{m}u_{\lambda_1}^{n-i} \rangle & \textrm{for $0 < i \leq n$}\\
    &\oplus \ZZ/p^2 \langle a_{\lambda_0}^ju_{\lambda_0}^{m-j}u_{\lambda_1}^n \rangle & \textrm{for $0 < j \leq m$},
\end{align*}
where $u_{\lambda_0}^m u_{\lambda_1}^n$ generates $\UZ$, all $p$-torsions generate $\UB_{0,1}$ and $p^2$-torsions generate $\UB_{1,1}$. By considering the gold relation $a_{\lambda_1}u_{\lambda_0} = p a_{\lambda_0}u_{\lambda_1}$, this is precisely $\U{BB}_{C_{p^2}}$ in Proposition \ref{prop-au}. A picture indicating connecting homomorphism and extensions for $m = n = 2$ is the following:
\begin{center}
\includegraphics{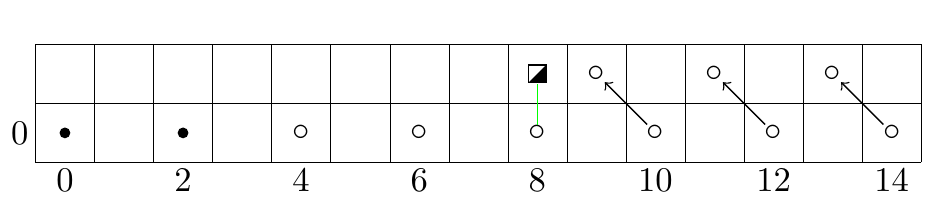}
\end{center}

The next case is $m,n < 0$. $\U{\pi}_{* - n\lambda_1 - m\lambda_0}(H\UZ_h)$ has the same description in every case, and by Proposition \ref{prop-Cp2-a-local} we have
\begin{align*}
    \U{\pi}_{*-n\lambda_1-m\lambda_0}(\widetilde{H\UZ})(G/G) = \ZZ/p\langle pa_{\lambda_0}^n u_{\lambda_0}^{m}\rangle &\oplus \ZZ/p \langle \Sigma^{-1}a_{\lambda_0}^n u_{\lambda_1}^{i-n} \rangle &\textrm{for $0 < i < |m|$}\\
    &\oplus \ZZ/p^2 \langle a_{\lambda_0}^n u_{\lambda_1}^m\frac{u_{\lambda_0}^j}{a_{\lambda_0}^j}\rangle &\textrm{for $j > 0$}
\end{align*}

The connecting homomorphism can be computed in the same way. In this case, the class $ pa_{\lambda_0}^n u_{\lambda_0}^{m}$ and all $p^2$-torsion classes in $\U{\pi}_{* - n\lambda_0 - m\lambda_1}(\widetilde{H\UZ})$ kill the corresponding classes with $\Sigma^{-1}$ in $\U{\pi}_{* - n\lambda_0 -m\lambda_1}(H\UZ_h)$. The connecting homomorphism on $p^2$-torsions are isomorphism, while on $pa_{\lambda_0}^n u_{\lambda_0}^{m}$ it fits into the following short exact sequence
\[
    \U{0} \rightarrow \JJ \rightarrow \circ \rightarrow \bullet \rightarrow \U{0}.
\]
So we have for $m,n < 0$,
\begin{align*}
    \U{\pi}_{*-n\lambda_1-m\lambda_0}(H\UZ)(G/G) = \ZZ \langle p^2u_{\lambda_0}^m u_{\lambda_1}^n \rangle &\oplus \ZZ/p^2 \langle  \Sigma^{-1} u_{\lambda_0}^m u_{\lambda_1}^n \frac{u_{\lambda_0}^i}{a_{\lambda_0}^i}\rangle &\textrm{for $0< i <|m|$}\\
    &\oplus \ZZ/p \langle \Sigma^{-1}a_{\lambda_0}^m u_{\lambda_1}^n \frac{u_{\lambda_1}^j}{a_{\lambda_1}^j} \rangle &\textrm{for $0 \leq j < |n|$}
\end{align*}
The torsion free class generates a Mackey functor $\UZ_{1,1}$, the $\ZZ/p$-torsions generate $\UB_{1,1}$ and the $\ZZ/p$-torsions generate $\UB_{0,1}$.

For $m = n = -2$, the picture indicting the connecting homomorphism is the following.
\begin{center}
\includegraphics{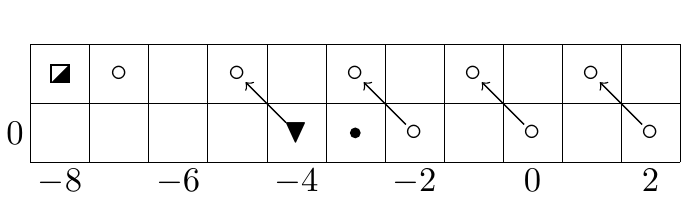}
\end{center}

If $n = 0$ and $m < 0$, a similar argument of the case $n,m < 0$ works, and we have
\[
    \U{\pi}_{* - m\lambda_0}(H\UZ)(G/G) = \ZZ \langle p^2 u_{\lambda_0}^m \rangle \oplus \ZZ/p^2\langle \Sigma^{-1} u_{\lambda_0}^m \frac{u_{\lambda_0}^i}{a_{\lambda_0}^i} \rangle \textrm{ for $0 < i < |m|$.}
\]

The next case is $n < 0$ and $m > 0$. By reading $\U{\pi}_{*-n\lambda_1 - m\lambda_0}(\widetilde{H\UZ})$ from Proposition \ref{prop-Cp2-a-local} and $\U{\pi}_{*-n\lambda_1 - m\lambda_0}(H\UZ_h)$, we can compute as above, and find that
\begin{align*}
    \U{\pi}_{*-n\lambda_1-m\lambda_0}(H\UZ) = \ZZ \langle u_{\lambda_0}^m u_{\lambda_1}^n \rangle \oplus \ZZ/p\langle pa_{\lambda_0}^m u_{\lambda_1}^n\rangle &\oplus \ZZ/p \langle \Sigma^{-1}a_{\lambda_0}^mu_{\lambda_1}^{n}\frac{u_{\lambda_1}^i}{a_{\lambda_1}^i} \rangle &\textrm{for $0 < i < |n|$}\\
    &\oplus \ZZ/p^2 \langle a_{\lambda_0}^m u_{\lambda_1}^{n}\frac{u_{\lambda_0}^j}{a_{\lambda_0}^j}\rangle &\textrm{for $0 < j < m$}.
\end{align*}
The torsion free class generates $\UZ$, the class $pa_{\lambda_0}^m u_{\lambda_1}^n$ generates $\UB_{1,0}^E$, all other $p$-torsions generate $\UB_{0,1}$ and all $p^2$-torsions generate $\UB_{1,1}$.

For $n = - 2$ and $m = 2$ the connecting homomorphism is the following:

\begin{center}
\includegraphics{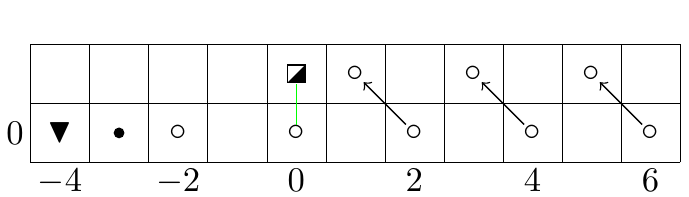}
\end{center}

Finally we compute for $n > 0$ and $m < 0$. In this case there are some subtlety. We start with listing homotopy Mackey functors of $H\UZ_h$ and $\widetilde{H\UZ}$ again.

\begin{align*}
    \U{\pi}_{*-n\lambda_1-m\lambda_0}(\widetilde{H\UZ})(G/G) & = \ZZ/p \langle a_{\lambda_0}^m a_{\lambda_1}^iu_{\lambda_1}^{n-i} \rangle & \textrm{for $0 < i \leq n$}\\
    & \oplus \ZZ/p^2 \langle a_{\lambda_0}^m u_{\lambda_1}^n\frac{u_{\lambda_0}^j}{a_{\lambda_0}^j} \rangle & \textrm{for $0 \leq j$.}
\end{align*}

\[
    \U{\pi}_{*-n\lambda_1-m\lambda_0}(H\UZ_h)(G/G) = p^2\ZZ \langle u_{\lambda_0}^m u_{\lambda_1}^n \rangle \oplus \ZZ/p^2 \langle \Sigma^{-1} u_{\lambda_0}^m u_{\lambda_1}^n \frac{u_{\lambda_0}^i}{a_{\lambda_0}^i} \rangle \ \textrm{for $i > 0$}.
\]

First we assume $m < -1$. In this case, the Mackey functor $\UB_{0,1}$ in $\U{\pi}_{* - n\lambda_1 - m\lambda_0}(\widetilde{H\UZ})$ with highest integer degree, is generated by $a_{\lambda_0}^ma_{\lambda_1}u_{\lambda_1}^{n-1}$. Its target under connecting homomorphism is $\UB_{1,1}$ generated by $\Sigma^{-1}\frac{a_{\lambda_0}^{m+1}u_{\lambda_1}^n}{u_{\lambda_0}}$ (If $m = -1$ this element doesn't exist). By the gold relation, we have
\[
    a_{\lambda_1}u_{\lambda_0} = pa_{\lambda_0} u_{\lambda_1}.
\]
Therefore, the connecting homomorphism in this degree is
\[
    \U{0} \rightarrow \bullet \xrightarrow{p} \circ \rightarrow \JJJ \rightarrow \U{0}.
\]
For other $\UB_{0,1}$ in $\U{\pi}_{* - n\lambda_1 - m\lambda_0}(\widetilde{H\UZ})$, by similar argument, we see that the map is multiplication by higher powers of $p$ thus is trivial. For $\UB_{1,1}$ in $\U{\pi}_{* - n\lambda_1 - m\lambda_0}(\widetilde{H\UZ})$, they maps isomorphically under the connecting homomorphism. Finally, in degree $* = 2n + 2m$, there is a potential extension
\[
    \U{0} \rightarrow \fourbox \rightarrow ? \rightarrow \bullet \rightarrow \U{0}.
\]
For $m < -1$, by the same gold relation argument, the extension is trivial. Therefore we conclude that
\begin{align*}
    \U{\pi}_{* - n\lambda_1 - m\lambda_0}(H\UZ) = \ZZ \langle p^2u_{\lambda_0}^m u_{\lambda_1}^n \rangle & \oplus \ZZ/p \langle \Sigma^{-1}\frac{a_{\lambda_0}^{m+1}u_{\lambda_1}^n}{u_{\lambda_0}} \rangle & \\
    &\oplus \ZZ/p \langle a_{\lambda_0}^m a_{\lambda_1}^iu_{\lambda_1}^{n-i} \rangle & \textrm{for $1 < i \leq n$}\\
    &\oplus \ZZ/p^2 \langle \Sigma^{-1} u_{\lambda_0}^{m}u_{\lambda_1}^n \frac{u_{\lambda_0}^i}{a_{\lambda_0}^i} \rangle & \textrm{for $0 < i < |m| - 1$}
\end{align*}

The torsion free class generates a $\UZ_{1,1}$, the class $\Sigma^{-1}\frac{a_{\lambda_0}^{m+1}u_{\lambda_1}^n}{u_{\lambda_0}}$ generates a $\UB_{1,0}$, all other $p$-torsions generate $\UB_{0,1}$ and all $p^2$-torsions generate $\UB_{1,1}$.

For $n = 4$ and $m = -3$, the picture of the connecting homomorphism is the following:
\begin{center}
\includegraphics{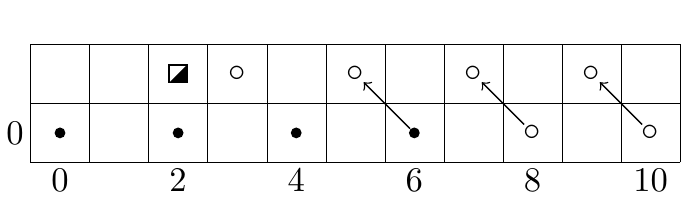}
\end{center}

When $m = -1$, by the gold relation again, the potential extension is nontrivial:
\[
    \U{0} \rightarrow \UZ_{1,1} \rightarrow \UZ_{1,0} \rightarrow \UB_{0,1} \rightarrow \U{0}.
\]
Thus we have
\[
    \U{\pi}_{* - n\lambda_1 + \lambda_0}(H\UZ) =  \ZZ \langle \frac{pu_{\lambda_1}^n}{u_{\lambda_0}}\rangle \oplus \ZZ/p \langle a_{\lambda_0}^{-1} a_{\lambda_1}^iu_{\lambda_1}^{n-i} \rangle \ \textrm{for $1 < i \leq n$}.
\]
The torsion-free class generates $\UZ_{1,0}$ and $p$-torsions generate $\UB_{0,1}$.

A picture indicating the case $n = 3$ and $m = -1$ is the following:

\begin{center}
\includegraphics{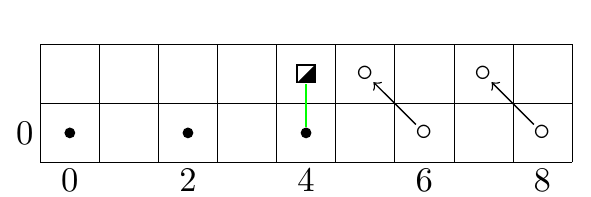}
\end{center}

\begin{rmk}
If $n = 1$ and $m = -1$, we see that
\[
S^{-\lambda_1 + \lambda_0} \wedge H\UZ \simeq H\UZ_{1,0},
\]
which can also be obtained by Theorem \ref{thm-formz}.
\end{rmk}

\textbf{An exotic multiplication.} In the case $n > 0$ and $m < 0$, if $m = -1$, then the torsion free class comes from an extension and is $\UZ_{1,0}$. However, when $m < -1$, by the gold relation this extension doesn't happen, and the corresponding $\UZ$-module in degree $\star = 2(m + n) - n\lambda_1 - m\lambda_0$ is $\UZ_{1,1} \oplus \UB_{0,1}$. Since the map $H\UZ \rightarrow \widetilde{H\UZ}$ is a map of ring spectra and the generator $\frac{pu_{\lambda_1}^n}{u_{\lambda_0}}$ of $\UZ_{1,0}$ maps to $a_{\lambda_0}^{-1}a_{\lambda_1}u_{\lambda_1}^{n-1}$ in $\U{\pi}_{\star}(\widetilde{H\UZ})$, whose $k$-th powers in $\widetilde{H\UZ}$ is nontrivial for all $k \geq 1$. Therefore, $(\frac{pu_{\lambda_1}^n}{u_{\lambda_0}})^k$ maps to $a_{\lambda_0}^{-k}a_{\lambda_1}^ku_{\lambda_1}^{k(n-1)}$ in $\U{\pi}_{\star}(\widetilde{H\UZ})$. On the other hand, $\frac{pu_{\lambda_1}^n}{u_{\lambda_0}}$ restricts to $p$-times of the chosen generator of $H_0(S^0;\ZZ)$ in $G/e$-level, thus its $k$-th power restricts to $p^k$-times of the generator. So we conclude
\begin{prop}\label{prop-exoticmul}
Consider the class $\frac{pu_{\lambda_1}^n}{u_{\lambda_0}} \in \U{\pi}_{2(n - 1) - n\lambda_1 + \lambda_0}(H\UZ)$. For $k > 1$, its $k$-th power is
\[
    (\frac{pu_{\lambda_1}^n}{u_{\lambda_0}})^k = \frac{p^ku_{\lambda_1}^{kn}}{u_{\lambda_0}^k} + a_{\lambda_0}^{-k} a_{\lambda_1}^k u_{\lambda_1}^{k(n-1)}.
\]
That is, the sum of $p^{k-2}$-times of transfer of the underlying orientation class and the $p$-torsion class.
\end{prop}

Finally, we can describe the $RO(C_{p^2})$-graded ring structure for $\U{\pi}_{\star}(H\UZ)(G/G)$, which combining with Proposition \ref{prop-Cporientable}, describes the $RO(C_{p^2})$-graded Green functor structure of $\U{\pi}_{\star}(H\UZ)$.

\begin{thm}\label{thm-Cp2oriented}
As a module over $BB_{C_{p^2}}$ (see Proposition \ref{prop-au}), $\U{\pi}_{\star}(H\UZ)(G/G)$ is
\begin{itemize}
    \item If $m,n \geq 0$,
        \begin{align*}
            \U{\pi}_{* - n\lambda_1 - m\lambda_0}(H\UZ) = \ZZ \langle u_{\lambda_0}^m u_{\lambda_1}^n \rangle &\oplus \ZZ/p \langle a_{\lambda_1}^ia_{\lambda_0}^{m}u_{\lambda_1}^{n-i} \rangle & \textrm{for $0 < i \leq n$}\\
            &\oplus \ZZ/p^2 \langle a_{\lambda_0}^ju_{\lambda_0}^{m-j}u_{\lambda_1}^n \rangle & \textrm{for $0 < j \leq m$}.
        \end{align*}
    \item If $m,n < 0$,
        \begin{align*}
            \U{\pi}_{*-n\lambda_1-m\lambda_0}(H\UZ)(G/G) = \ZZ \langle p^2u_{\lambda_0}^m u_{\lambda_1}^n \rangle &\oplus \ZZ/p^2 \langle  \Sigma^{-1} u_{\lambda_0}^m u_{\lambda_1}^n \frac{u_{\lambda_0}^i}{a_{\lambda_0}^i}\rangle &\textrm{for $0< i <|m|$}\\
            &\oplus \ZZ/p \langle \Sigma^{-1}a_{\lambda_0}^m u_{\lambda_1}^n \frac{u_{\lambda_1}^j}{a_{\lambda_1}^j} \rangle &\textrm{for $0 \leq j < |n|$}
        \end{align*}
    \item If $n = 0$ and $m < 0$,
        \[
            \U{\pi}_{* - m\lambda_0}(H\UZ)(G/G) = \ZZ \langle p^2 u_{\lambda_0}^m \rangle \oplus \ZZ/p^2\langle \Sigma^{-1} u_{\lambda_0}^m \frac{u_{\lambda_0}^i}{a_{\lambda_0}^i} \rangle \textrm{ for $0 < i < |m|$.}
        \]
    \item If $n < 0$ and $m > 0$,
        \begin{align*}
            \U{\pi}_{*-n\lambda_1-m\lambda_0}(H\UZ) = \ZZ \langle u_{\lambda_0}^m u_{\lambda_1}^n \rangle \oplus \ZZ/p\langle pa_{\lambda_0}^m u_{\lambda_1}^n\rangle &\oplus \ZZ/p \langle \Sigma^{-1}a_{\lambda_0}^mu_{\lambda_1}^{n}\frac{u_{\lambda_1}^i}{a_{\lambda_1}^i} \rangle &\textrm{for $0 < i < |n|$}\\
            &\oplus \ZZ/p^2 \langle a_{\lambda_0}^m u_{\lambda_1}^{n}\frac{u_{\lambda_0}^j}{a_{\lambda_0}^j}\rangle &\textrm{for $0 < j < m$}.
        \end{align*}
    \item If $n > 0$ and $m < -1$,
        \begin{align*}
            \U{\pi}_{* - n\lambda_1 - m\lambda_0}(H\UZ) = \ZZ \langle p^2u_{\lambda_0}^m u_{\lambda_1}^n \rangle & \oplus \ZZ/p \langle \Sigma^{-1}\frac{a_{\lambda_0}^{m+1}u_{\lambda_1}^n}{u_{\lambda_0}} \rangle & \\
            &\oplus \ZZ/p \langle a_{\lambda_0}^m a_{\lambda_1}^iu_{\lambda_1}^{n-i} \rangle & \textrm{for $1 < i \leq n$}\\
            &\oplus \ZZ/p^2 \langle \Sigma^{-1} u_{\lambda_0}^{m}u_{\lambda_1}^n \frac{u_{\lambda_0}^i}{a_{\lambda_0}^i} \rangle & \textrm{for $0 < i < |m| - 1$}
        \end{align*}
    \item If $n > 0$ and $m = -1$,
        \[
            \U{\pi}_{* - n\lambda_1 + \lambda_0}(H\UZ) =  \ZZ \langle \frac{pu_{\lambda_1}^n}{u_{\lambda_0}}\rangle \oplus \ZZ/p \langle a_{\lambda_0}^{-1} a_{\lambda_1}^iu_{\lambda_1}^{n-i} \rangle \ \textrm{for $1 < i \leq n$}.
        \]
\end{itemize}
Multiplication is given by the following rules:
\begin{itemize}
    \item If at least one of the elements in multiplication is without $\Sigma^{-1}$, then multiplication can be calculated by tracking names of elements, with a special case by Proposition \ref{prop-exoticmul}.
    \item Multiplication of two elements with $\Sigma^{-1}$ is trivial.
\end{itemize}
\end{thm}

\begin{proof}
    We only need to prove the square zero extension part. Assume $x,y \in \U{\pi}_{*-n\lambda_1-m\lambda_0}(H\UZ)(G/G)$ are elements with $\Sigma^{-1}$.

    First, if $x,y$ are torsions in $m,n \leq 0$, then $xy = 0$ since all torsions there are in odd dimension. Now if at least one of $x,y$ is from the part $m > 0$ and $n < 0$, then by the module structure, $x  = a_{\lambda_0}^i x'$ and $y = a_{\lambda_0}^j y'$, where $x'$ and $y'$ are torsions in $m , n \leq 0$, therefore $xy = a_{\lambda_0}^{i+j}x'y' = 0$.

    If $x,y$ are from the part where $m < 0$ and $n > 0$, then they are $a_{\lambda_0}$-torsion. However, all torsion classes in even degree for $m < 0$ and $n > 0$ are not annihilated by $a_{\lambda_0}$, thus $xy = 0$.
\end{proof}

For $p = 2$, Theorem \ref{thm-Cp2oriented} computes $\U{\pi}_{\star}(H\UZ)$ for $\star$ in the index $2$ subgroup of $RO(C_4)$ of orientable representations. In $C_4$, we have $\lambda_1 = 2\sigma$, where $\sigma$ is the sign representation of $C_4$. The non-orientable representations are precisely those with odd copies of $\sigma$. We have a cofibre sequence
\[
    C_4/C_{2+} \rightarrow S^0 \rightarrow S^{\sigma}.
\]
By smashing it with $S^{V} \wedge H\UZ$, we can compute non-orientable representations from orientable ones. We have
\[
    \U{\pi}_{\star}(C_4/C_{2+} \wedge X ) \cong \U{\pi}_{\star}(X)_{C_4/C_2},
\]
and the map $C_4/C_{2+} \rightarrow S^0$ is the transfer map. Therefore modulo extension and multiplication, for each $\UZ$-module $\UM$, we need to compute the kernel and cokernel of the Mackey functor transfer map
\[
\U{Tr}^4_2 : \UM_{C_4/C_2} \rightarrow \UM.
\]
Extensions can be resolved by Proposition \ref{lem-TAR}, and multiplications are controlled by Theorem \ref{thm-Cp2oriented}. All exact sequences about the Mackey functor transfer map $\U{Tr}^4_2$ are listed below:
\[
\xymatrix{
{\U{0}} \ar[r] & \overline{\square} \ar[r] & {\hat{\square}} \ar[r]^{\U{Tr}^4_2} & \square \ar[r] & \bullet \ar[r] & {\U{0}}\\
{\U{0}} \ar[r] & \overline{\fourbox}\ar[r] & {\hat{\fourbox}} \ar[r]^{\U{Tr}^4_2}&  {\fourbox} \ar[r]&\bullet \ar[r] & {\U{0}}\\
{\U{0}} \ar[r] & \overline{\square} \ar[r] & {\hat{\square}} \ar[r]^{\U{Tr}^4_2} & {\twobox} \ar[r] & {\U{0}}\\
{\U{0}} \ar[r] & \overline{\fourbox}\ar[r] & {\hat{\fourbox}}\ar[r]^{\U{Tr}^4_2} & {\halfbox} \ar[r]& {\U{0}}
}
\]
\[
\xymatrix{
{\U{0}} \ar[r]^{\U{Tr}^4_2} & \bullet \ar[r] & {\U{0}}\\
{\U{0}} \ar[r] & \overline{\bullet} \ar[r] & {\hat{\bullet}} \ar[r]^{\U{Tr}^4_2} & \circ \ar[r] & \bullet \ar[r] & {\U{0}}\\
{\U{0}} \ar[r] & \JJJ \ar[r] & {\hat{\bullet}} \ar[r]^{\U{Tr}^4_2} & \JJJ \ar[r] & \bullet \ar[r] & {\U{0}}\\
{\U{0}} \ar[r] & \overline{\bullet} \ar[r] & {\hat{\bullet}} \ar[r]^{\U{Tr}^4_2} & \JJ \ar[r] & {\U{0}}
}
\]

From here, what's left is simply routine computation. We will not list the complete result, since it is very tedious, but give some examples of computation. We shall have no difficulty in computing all $\U{\pi}_{\star}(H\UZ)$ by this method on demand.

\begin{exam}\label{exam-C4}
Consider $\U{\pi}_{* - 4\sigma + 3\lambda_0}(H\UZ)$, which by Theorem \ref{thm-Cp2oriented} is the following:
\[
    \U{\pi}_{i - 4\sigma + 3\lambda_0}(H\UZ) = \left\{ \begin{array}{ll}
                                                            \halfbox & \textrm{for $i = -2$}\\
                                                            \circ &    \textrm{for $i = -1$}\\
                                                            \bullet &  \textrm{for $i = 0$}\\
                                                            \JJJ    &  \textrm{for $i = 1$}\\
                                                            \U{0}   &   \textrm{otherwise.}
                                                        \end{array} \right.
\]
By applying the corresponding exact sequences, we have
\[
    \U{\pi}_{i - 3\sigma + 3\lambda_0}(H\UZ) = \left\{ \begin{array}{ll}
                                                           \UM_1 & \textrm{for $i = -1$}\\
                                                           \UM_2 & \textrm{for $i = 0$}\\
                                                           \bullet & \textrm{for $i = 1$}\\
                                                           \JJJ    & \textrm{for $i = 2$}\\
                                                           \U{0}   & \textrm{otherwise,}
                                                       \end{array} \right.
\]
where $\UM_1$ and $\UM_2$ fits into the following extensions
\[
\U{0} \rightarrow \bullet \rightarrow \UM_1 \rightarrow \overline{\fourbox} \rightarrow 0
\]
and
\[
\U{0} \rightarrow \bullet \rightarrow \UM_2 \rightarrow \overline{\bullet} \rightarrow 0.
\]
For the first extension, since $\U{\pi}_{-1 - 2\sigma + 3\lambda_0}(H\UZ) = \U{0}$, the $2$-torsion is annihilated by $a_{\sigma}$, so it is in the image of $Tr^4_2$, that means the extension on $\UM_1$ is nontrivial and $\UM_1 = \dot{\overline{\fourbox}}$. For the second extension, we can check that the $2$-torsion in $C_4/C_4$-level is not annihilated by $a_{\sigma}$, thus the extension is trivial, and $\UM_2 = \bullet \oplus \overline{\bullet}$.
\end{exam}
\subsection{Dualities} \label{sec-duality}

Here we describe how the equivariant Anderson duality in Definition \ref{def-Anderson}, the universal coefficient spectral sequence and the K\"unneth spectral sequence in Theorem \ref{thm-UCSS} interacts with $\U{\pi}_{\star}(H\UZ)$. Most of the conclusion works for $G = C_{p^n}$, but we will focus on examples with $C_{p^2}$, as we have a complete description in Theorem \ref{thm-Cp2oriented}.

First we start with Anderson duality. Assume $\U{\pi}_{* - V}(H\UZ)$ is known for some $V \in RO(G)$. Since $I_{\ZZ}(H\UZ) \simeq \Sigma^{2 - \lambda_0}H\UZ$, the short exact sequence in Proposition \ref{prop-Anderson} computes $\U{\pi}_{-* + V + (2 - \lambda_0)}(H\UZ)$ from $\U{\pi}_*(* - V)(H\UZ)$. There is only one potential nontrivial extension when $* = |V|$. When $V$ is orientable, it can be solved by the following:
\begin{prop}\label{prop-noext}
Let $\UM$ be a form of $\UZ$ and $\U{T}$ be a levelwise torsion $\UZ$-module, then
\[
    Ext^1_{\UZ}(\UM,\U{T}) = \U{Ext}^1_{\UZ}(\UM,\U{T})(G/G) \cong 0.
\]
\end{prop}

\begin{proof}
    Consider the following extension
    \[
        \U{0} \rightarrow \U{T} \rightarrow \U{X} \rightarrow \UM \rightarrow \U{0},
    \]
    and consider $G/K$-level that $\U{T}(G/K)$ is nontrivial. In this level, as $G$-modules, the extension is trivial. Therefore, if the extension as $\UZ$-module is nontrivial, the only possibility is that the image of $\U{T}(G/K)$ in $\U{X}(G/K)$ receives an exotic restriction or transfer from a torsion free class. However, in the quotient $\UM$ all restrictions and transfers are injective, so exotic restriction or transfers cannot happen.
\end{proof}

\begin{rmk}
This proposition does not hold when $p = 2$ and $\UZ$-modules with $\ZZ_-$ are involved. For example in $C_2$ we have a nontrivial extension
\[
    \U{0} \rightarrow \UB_{1} \rightarrow \dot{\UZ}_- \rightarrow \UZ_- \rightarrow 0,
\]
which is an exotic transfer.
\end{rmk}

\begin{cor}\label{cor-anderson}
If $G = C_{p^n}$ and $V \in RO(G)$ is orientable, then $\U{\pi}_{* - V}(H\UZ)$ completely determines $\U{\pi}_{-* + V + (2 - \lambda_0)}(H\UZ)$ via Anderson duality.
\end{cor}
Some multiplicative properties are also preserved in Anderson duality.

\begin{prop}\label{prop-Adnerson-alambda}
The classes $a_{\lambda_i} \in \U{\pi}_{-\lambda_i}(H\UZ)$ is self-Anderson dual.
\end{prop}

The universal coefficient spectral sequence is more subtle. If $\U{\pi}_{* - V}(H\UZ)$ is known, then by Theorem \ref{thm-UCSS} there is a spectral sequence with $E^2$-page
\[
    E^2_{s,t} = \U{Ext}^s(\U{\pi}_{t - V}(H\UZ),\UZ) \Rightarrow \U{\pi}_{-t - s + V}(H\UZ).
\]

Since all torsion $\UZ$-modules in $\U{\pi}_{\star}(H\UZ)$ has trivial $G/e$-level, by Theorem \ref{thm-ext}, we can compute their internal $Ext$. If $\UM = \UZ_{t_1,t_2,...,t_n}$ is a form of $\UZ$, then by the short exact sequence
\[
    \U{0} \rightarrow \UZ_{t_1,t_2,...,t_n} \rightarrow \UZ \rightarrow \UB_{t_1,t_2,..,t_n} \rightarrow \U{0},
\]
we have
\begin{prop}\label{prop-extZ}
\[
    \U{Ext}^i_{\UZ}(\UZ_{t_1,t_2,...,t_n},\UZ) = \left\{ \begin{array}{ll}
                                                            \UZ & \textrm{for $i = 0$}\\
                                                            \UB_{t_1,t_2,...,t_n}^E & \textrm{for $i = 2$}\\
                                                            \U{0} & \textrm{otherwise}
                                                        \end{array} \right.
\]
\end{prop}
Therefore we completely understand $E^2$-page of the universal coefficient spectral sequence. Since most of the elements are concentrated in filtration $3$ by Theorem \ref{thm-ext}, and the only element in filtration $2$ and $0$ are produced by the form of $\UZ$ via the above proposition, we see that this spectral sequence will only have a single potential $d_3$, from $\UZ$ to a torsion $\UZ$-module in one degree lower. Finally, there can be two potential extensions, one being an extension between a torsion in filtration $3$ and a torsion in filtration $2$ from the above proposition. Another one is a potential extension from a torsion in filtration $3$ to the form of $\UZ$ in degree $0$. This extension is trivial by Proposition \ref{prop-noext}.

The $d_3$ in the universal coefficient spectral sequence happens very often: If $\UM$ is a form of $\UZ$, $\U{Ext}^0_{\UZ}(\UM,\UZ) \cong \UZ$. However the orientation $\UZ$-module in $\U{H}_{|V|}(S^V;\UZ)$ is usually not $\UZ$, but other forms of $\UZ$, so they are just the kernel of this $d_3$. We can see this by examples.

\begin{exam}
Let $G = C_p$ and $V = 2\lambda_0$. Then we know
\[
    \U{H}_i(S^V;\UZ) = \left\{ \begin{array}{ll}
                                    \UB_1    & \textrm{for $i = 0,2$}\\
                                    \UZ      & \textrm{for $i = 4$}\\
                                    \U{0}    & \textrm{otherwise.}
                                \end{array} \right.
\]
We compute, by Theorem \ref{thm-ext} and Proposition \ref{prop-extZ}, the $E^2$-page of the universal coefficient spectral sequence with $d_3$ is
\begin{center}
\includegraphics{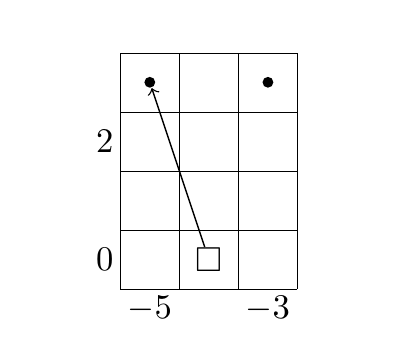}
\end{center}
Here we are using the Adams index.
\end{exam}

In fact, if we can compute the form of $\UZ$ in $\U{H}_{|V|}(S^V;\UZ)$ before computing all $\U{\pi}_{\star}(H\UZ)$ and solve the extension between filtration $3$ and $2$ (It is trivial in all known computations.), then we can use Anderson duality and the universal coefficient spectral sequence to compute all $\U{\pi}_{\star}(H\UZ)$ inductively. Let $G = C_{p^n}$
\begin{enumerate}
    \item Assume that we know $\U{\pi}_{\star}(H\UZ)$ for $C_{p^{n-1}}$, then we can use Corollary \ref{cor-pullback}  to compute $\U{\pi}_{V}(H\UZ)$ for all $V \in RO(G)$ containing no $\lambda_0$.
    \item Using equivariant Anderson duality and Lemma \ref{lem-miracle}, we can compute $\U{\pi}_{* + V - \lambda_0}(H\UZ)$ from $\U{\pi}_{* - V}(H\UZ)$. By Proposition \ref{prop-noext}, the potential extension is trivial.
    \item Using the universal coefficient spectral sequence, we compute $\U{\pi}_{* - V + \lambda_0}(H\UZ)$ from $\U{\pi}_{* + V -\lambda_0}(H\UZ)$.
    \item Repeat $(2)$ and $(3)$ to add arbitrary copies of $\lambda_0$ into $V$.
\end{enumerate}
The problem that which form of $\UZ$ appears in $\U{H}_{|V|}(S^V;\UZ)$ is nontrivial: By Theorem \ref{thm-formz} we know that every form of $\UZ$ appears in $\U{\pi}_{\star}(H\UZ)$.

If $G = C_p$, it is an entertaining practice to compute $\U{\pi}_{\star}(H\UZ)$ starting from the defining property of $H\UZ$ in Theorem \ref{thm-EM}, and we can see that the figures \ref{fig-Cp} and \ref{fig-C2} are self-dual in both ways. For $G = C_{p^2}$, Figure \ref{fig-Cp2} and \ref{fig-Cp2'} are $\U{\pi}_*(S^{2\lambda_1 + m\lambda_0} \wedge H\UZ)$ and $\U{\pi}_*(S^{-2\lambda_1 + m\lambda_0}\wedge H\UZ)$ respectively, with horizontal coordinate $*$, vertical coordinate $m$ and vertical lines are $a_{\lambda_0}$-multiplications. Both Anderson duality and the universal coefficient spectral sequence can be seen in these pictures.\\
\begin{figure}[H]
\includegraphics{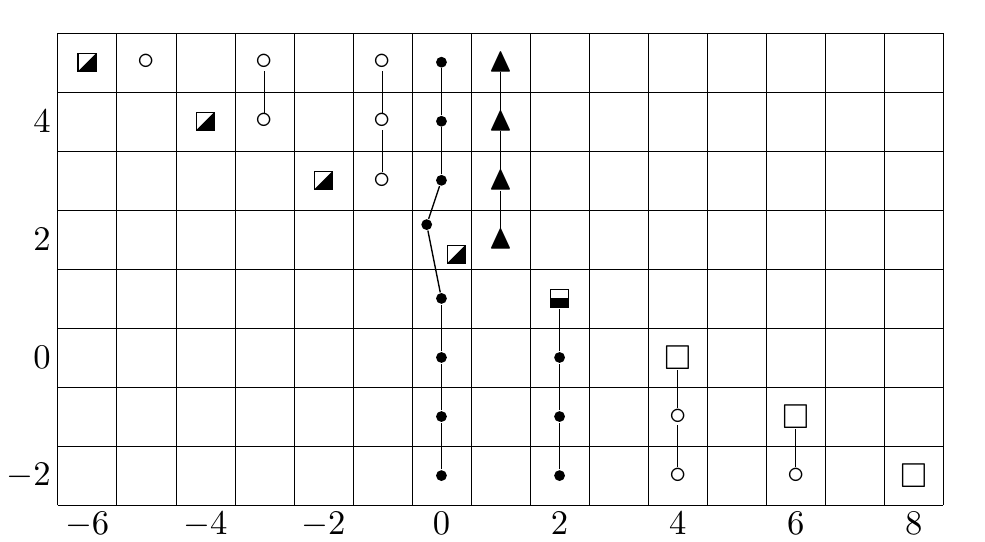}
\caption{$\U{\pi}_*(S^{2\lambda_1 + m\lambda_0} \wedge H\UZ)$}\label{fig-Cp2}
\end{figure}
\begin{figure}[H]
\includegraphics{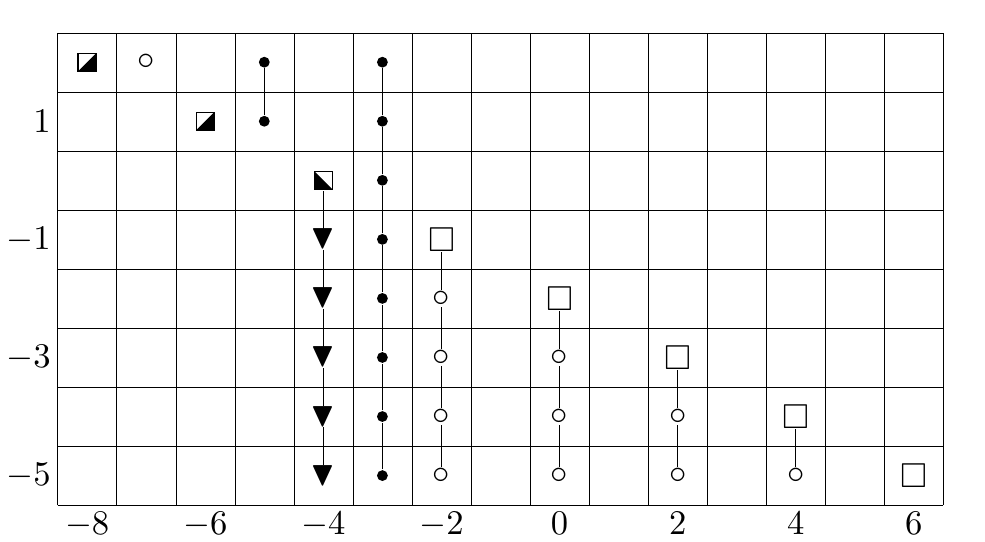}
\caption{$\U{\pi}_*(S^{-2\lambda_1 + m\lambda_0}\wedge H\UZ)$}\label{fig-Cp2'}
\end{figure}

We can also compute in the K\"unneth spectral sequence in Theorem \ref{thm-UCSS}. The advantage of K\"unneth spectral sequence is that it keeps track of most multiplicative structure in filtration $0$. However, it has many differentials and nontrivial extensions.

\begin{exam}
Let $G = C_p$ and $V = \lambda_0$. Since
\[
    \U{Tor}^{\UZ}_i(\UB_1,\UB_1) = \left\{ \begin{array}{ll}
                                                \UB_1 & \textrm{for $i = 0,3$}\\
                                                \U{0} & \textrm{otherwise.}
                                            \end{array} \right.
\]
We know the $E_2$-page of the spectral sequence computing $\U{H}_*(S^{2V};\UZ)$ from
\[
\U{Tor}^{\UZ}_{i,j}(\U{H}_*(S^{V}),\U{H}_*(S^V))
\]
is the following:
\begin{center}
\includegraphics{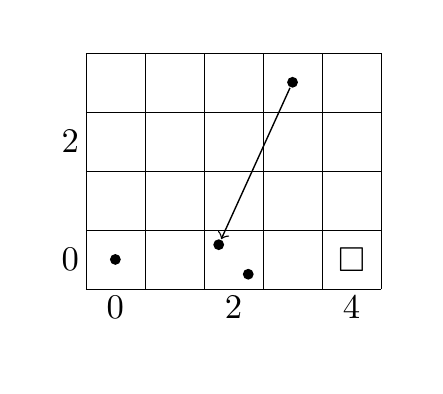}
\end{center}
Here the differential kills the sum of two torsions, making two different multiplications of $a_{\lambda_0}$ and $u_{\lambda_0}$ equal to each other in the quotient.
\end{exam}

\subsection{More homological computation}
Through the universal coefficient and K\"unneth spectral sequences, we can use our understanding of $\U{\pi}_{\star}(H\UZ)$ to compute more homological algebra of $\UZ$-modules. Essentially, the Tate diagram gives us a new way of computing homological invariants of $\UZ$-modules.

\begin{prop}\label{prop-extfromS}
Let $G = C_{p^n}$ and $\UM,\UN$ be forms of $\UZ$. Let $H\UM \simeq S^{V_1} \wedge H\UZ$ and $H\UN \simeq S^{V_2} \wedge H\UZ$ (see Theorem \ref{thm-formz}). We have
\[
    \U{Ext}^*_{\UZ}(\UM,\UN) \cong \U{\pi}_{-*}(S^{V_2 - V_1} \wedge H\UZ) \cong \U{H}_{-*}(S^{V_2 - V_1};\UZ).
\]
and
\[
    \U{Tor}^{\UZ}_*(\UM,\UN) \cong \U{\pi}_*(S^{V_1 + V_2} \wedge H\UZ) \cong \U{H}_*(S^{V_1 + V_2};\UZ).
\]
\end{prop}
\begin{proof}
By Corollary \ref{cor-equivalence},
\begin{align*}
    \U{Ext}^*_{\UZ}(\UM,\UN) & \cong \U{\pi}_{-*}Fun_{H\UZ}(H\UM ,H\UN)\\
                             & \cong \U{\pi}_{-*}Fun_{H\UZ}(S^{V_1} \wedge H\UZ, S^{V_2} \wedge H\UZ)\\
                             & \cong \U{\pi}_{-*}(S^{V_2 - V_1} \wedge H\UZ)
\end{align*}
The $\U{Tor}$ part shares a similar proof.
\end{proof}

\begin{exam}
Let $G = C_{p^2}$, $\UM = \UZ_{1,0}$ and $\UN = \UZ_{0,1}$. Then
\[
H\UM \simeq S^{\lambda_1 - \lambda_0} \wedge H\UZ
\]
and
\[
H\UN \simeq S^{2-\lambda_1} \wedge H\UZ.
\]
By reading $\U{\pi}_*(S^{-2\lambda_1 + \lambda_0} \wedge H\UZ)$ from Theorem \ref{thm-Cp2oriented}, we have
\[
        \U{Ext}^i_{\UZ}(\UZ_{1,0},\UZ_{0,1}) = \left\{ \begin{array}{ll}
                                                \UZ & i = 0\\
                                                \UB_{0,1} & i = 1\\
                                                \UB_{1,0}^E & i =2\\
                                                \U{0} & \textrm{otherwise}
                                                    \end{array} \right.
\]
Specifically, $\U{Ext}^1_{\UZ}(\UZ_{1,0},\UZ_{0,1})(G/G) \cong \ZZ/2$ is nontrivial. Therefore, there is a nontrivial extension
\[
    \U{0} \rightarrow \UZ_{0,1} \rightarrow \UM \rightarrow \UZ_{1,0} \rightarrow \U{0}.
\]
Such an extension can be constructed as follows:

Let $\UM$ be the $\UZ$-module with Lewis diagram
\begin{displaymath}
\xymatrix
@R=7mm
@C=10mm{
\ZZ \oplus \ZZ \ar@/_/[d]\\
\ZZ \oplus \ZZ \ar@/_/[d] \ar@/_/[u]\\
\ZZ \oplus \ZZ \ar@/_/[u]
}
\end{displaymath}
We call generators of $\ZZ \oplus \ZZ$ in $C_{p^2}/C_{p^i}$-level $a_i$ and $b_i$ for $0 \leq i \leq 2$, then restrictions and transfers are defined as follows:
\begin{align*}
    Res^{p^2}_p(a_2) = pa_1 &, \: Res^{p^2}_p(b_2) =  b_1 & Tr^{p^2}_p(a_1) = a_2 &, \:  Tr^{p^2}_p(b_1) =  pb_2 \\
    Res^{p}_1(a_1) = a_0    &, \: Res^{p}_1(b_1) = a_0 + pb_0 & Tr^p_1(a_0) = pa_1    & , \: Tr^p_1(b_0) = -a_1 + b_1
\end{align*}
$\UZ_{0,1} \rightarrow \UM$ maps generator in each $\ZZ$ to the corresponding $a_i$ and $\UM \rightarrow \UZ_{1,0}$ maps $a_i$ to $0$ and $b_i$ to generators in each level.

If a section $\UZ_{1,0} \rightarrow \UM$ exists, then it sends $1 \in \UZ_{1,0}(C_{p^2}/e)$ to $ta_0 + b_0$ for some $t \in \ZZ$. This forces $1$ in $C_{p^2}/C_p$-level going to $(tp-1)a_1 + b_1$, and therefore $p$ in $C_{p^2}/C_{p^2}$-level goes to $(tp-1)a_2 + pb_2$. But this element is not divisible by $p$, so we have a contradiction.
\end{exam}

\begin{exam}\label{exam-tor}
Let $G = C_{p^2}$ and $\UM = \UZ_{1,0}$. By Proposition \ref{prop-extfromS},
\[
\U{Tor}^{\UZ}_i(\UZ_{1,0},\UZ_{1,0}) \cong \U{\pi}_i(S^{2\lambda_1 - 2\lambda_0} \wedge H\UZ),
\]
therefore we have
\[
    \U{Tor}^{\UZ}_i(\UZ_{1,0},\UZ_{1,0}) = \left\{ \begin{array}{ll}
                                                \UZ_{1,1} \oplus \UB_{0,1} & \textrm{for $i = 0$}\\
                                                \UB_{1,0}                  & \textrm{for $i = 1$}\\
                                                \U{0}                      & \textrm{otherwise.}
                                                    \end{array} \right.
\]
\end{exam}

\bibliography{math}{}
\bibliographystyle{alpha}
\end{document}